\documentclass[a4paper,11pt,leqno]{article}
\usepackage[utf8]{inputenc}
\usepackage[T1]{fontenc} 
\usepackage{verbatim}
\usepackage{calligra}
\usepackage{enumerate}
\usepackage{dsfont}
\usepackage{amsfonts}
\usepackage{amsmath}
\usepackage{amsthm}
\usepackage{amssymb}
\usepackage{mathtools}
\usepackage{mathrsfs}
\usepackage{color}
\usepackage{graphicx}
\usepackage{bm}
\usepackage{hyperref}
\usepackage[all]{xy}

\usepackage{tikz}

\usepackage{graphicx}		
\usepackage[hypcap=false]{caption}

\DeclareMathOperator*{\tend}{\longrightarrow}

\DeclareMathOperator*{\D}{\rm{div}}

\DeclareMathOperator*{\wtend}{\,\, \rightharpoonup \,\,}
\DeclareMathOperator*{\limii}{\underline{\lim}}
\DeclareMathOperator*{\limss}{\overline{\lim}}

\theoremstyle{definition}
\newtheorem{defi}{Definition}[section]

\newtheorem{rmk}[defi]{Remark}

\theoremstyle{plane}
\newtheorem{thm}[defi]{Theorem}
\newtheorem{prop}[defi]{Proposition}
\newtheorem{cor}[defi]{Corollary}
\newtheorem{lemma}[defi]{Lemma}

\newcommand{\tsl}{\textsl}

\newcommand{\mc}{\mathcal}

\newcommand{\what}{\widehat}

\newcommand{\R}{\mathbb{R}}

\newcommand{\N}{\mathbb{N}}
\newcommand{\Z}{\mathbb{Z}}

\renewcommand{\P}{\mathbb{P}}

\newcommand{\curl}{{\rm curl}\,}

\newcommand{\dx}{ \, {\rm d} x}
\newcommand{\dt}{ \, {\rm d} t}
\newcommand{\B}{B^s_{\infty, r}}

\allowdisplaybreaks

\begin{document}

\newcommand{\dimitri}[1]{\textcolor{red}{[***DC: #1 ***]}}
\newcommand{\fra}[1]{\textcolor{blue}{[***FF: #1 ***]}}

\title{\Large{\textbf{\textsc{On the Well-Posedness of a Fractional Stokes-Transport System}}}}

\author{\normalsize\textsl{Dimitri Cobb}\vspace{.5cm} \\
\footnotesize{\textsc{Universität Bonn}} \\
{\footnotesize \it Mathematisches Institut} \vspace{.1cm} \\
{\footnotesize Endenicher Allee 60, 53115 Bonn, Germany} \vspace{.1cm} \\
\footnotesize{\ttfamily{cobb@math.uni-bonn.de}}
}

\vspace{.2cm}
\date\today

\date\today

\maketitle

\subsubsection*{Abstract}
{\footnotesize The purpose of this paper is to study the existence, uniqueness and lifespan of solutions for a fractional Stokes-Transport system. This problem should be understood as a model for sedimentation in a fluid where the viscosity law is given by a fractional Lapalce operator $(- \Delta)^{\alpha/2}$, with $\alpha = 2$ corresponding to the case of a normal viscous fluid, and $\alpha = 0$ reducing the problem to the Inviscid Incompressible Porous Media equation.
For each value of $\alpha \in [0, d]$, we prove various results related to well-posedness in critical function spaces, such as the existence of global weak solutions (for $\alpha > 0$), local existence and uniqueness (for $\alpha \geq 0$), global existence and uniqueness (for $\alpha \geq 1$), as well as study the lifespan of local solutions (for $0 \leq \alpha < 1$).
In particular, we show that gravity stratification leads to a directional blow-up criterion for local solutions (for $\alpha \in [0, 1[$) and find a lower bound for the lifespan of solutions which depends on the value of the dissipation parameter $\alpha \in [0, 1[$.
}

\paragraph*{\small 2020 Mathematics Subject Classification:}{\footnotesize 35Q35
(primary); 
76D03, 
35Q49, 
35S10, 
76D50  
(secondary).}

\paragraph*{\small Keywords: }{\footnotesize Stokes-Transport System, Fractional Laplacian, Incompressible Fluids, Critical Regularity, Lifespan of Solutions.}

\section{Introduction}

In this article, we study the existence, uniqueness and lifespan of solutions for the initial value problem related to the following active scalar equation:
\begin{equation}\label{ieq:AIPM}
\begin{cases}
\big( \partial_t + u \cdot \nabla \big) \rho = 0\\
(- \Delta)^{\alpha/2}u + \nabla \pi = \rho g\\
\D(u) = 0.
\end{cases}
\end{equation}
Here, the equations are set in the whole space $\R^d$ with $d \geq 2$, although our results could certainly be extended to the torus $\mathbb{T}^d$ with minimal adaptations. The unknowns are the scalar function $\rho(t, x) \in \R$ which represents a particle density, the vector field $u(t, x) \in \R^d$ which is a fluid velocity, and a scalar pressure function $\pi(t,x) \in \R$. In the second equation, $g = e_d = (0, ..., 0, 1) \in \R^d$ is a constant vector and $\alpha \geq 0$ a real number. The fractional Laplace operator $(- \Delta)^{\alpha/2}$ is is defined by its Fourier transform:
\begin{equation*}
    \forall f \in \mc S, \qquad \mc F [(- \Delta)^{\alpha/2} f ](\xi) = |\xi|^s \what{f}(\xi).
\end{equation*}

The physical background for the equations of \eqref{ieq:AIPM} is the study of the sedimentation of a cloud of particles under the effect of gravity and lying in a fluid with a generalized --fractionary-- viscosity law. As a matter of fact, the special cases $\alpha = 2$ and $\alpha = 0$, where the second equation then reduces to the Stokes equation and Darcy's law,
\begin{equation*}
-\Delta u + \nabla \pi = \rho g \qquad \text{and} \qquad u + \nabla \pi = \rho g,
\end{equation*}
correspond, respectively, to the Stokes-transport equation derived by Höfer \cite{Hofer2018} as the homogenization limit of particle system in a Stokes fluid, and the Inviscid Incompressible Porous Medium equation, studied in \textsl{e.g.} \cite{Elgindi}.

\medskip

The main motive for considering a fractional viscosity law is to better understand the mechanisms that provide existence and uniqueness of solutions in the special cases: when $\alpha = 2$, the initial value problem is known to have global solutions (this will be further discussed below), whereas even the existence of global weak solutions is challenging in the extreme, in the case $\alpha = 0$. Although the study of well-posedness for problem \eqref{ieq:AIPM} is interesting in itself and already presents a number of issues, it is also meant to contribute to a better understanding the ``classical'' problems $\alpha \in \{0, 2\}$. In particular, investigating the questions related to well-posedness for all intermediate values $\alpha \in [0, 2]$ will allow us to better isolate the difficulties when dealing with well-posedness when $\alpha = 0$.

To be precise, we will investigate the existence of global weak solutions (for $\alpha > 0$), local well-posedness (for $\alpha \geq 0$), global well-posedness (for $\alpha \geq 1$), as well as properties of the lifespan of the local solutions (for $0 \leq \alpha < 1$) such as lower bounds and continuation criteria.

\subsection{Broader Context: Active Scalar Equations}

The fractional Stokes-Transport equation is a part of a larger family of PDE problems which have been studied extensively in the past years, namely active scalar equations. These take the following form: the unknown is a scalar quantity $\rho(t, x)$ that is advected by a vector field $u(t, x)$, thus solving the transport equation
\begin{equation*}
\partial_t \rho + u \cdot \nabla \rho = 0,
\end{equation*}
and the vector field is given as a linear, and possibly non-local, function of $\rho$. In our case, it is fairly easy to see how $u$ can be expressed as a function of $\rho$: define the Leray projection operator $\P = {\rm Id} + \nabla (- \Delta)^{-1} \D$, which is the $L^2$-orthogonal projection on the space of divergence-free vector fields. Applying this operator to the second equation in \eqref{ieq:AIPM} gives
\begin{equation}\label{ieq:fracLapOp}
u = (- \Delta)^{- \alpha / 2} \P (\rho g).
\end{equation}
Many other equations fall into this general form. The best known example is probably the 2D Euler equations in their vorticity form, where the velocity field is given by
\begin{equation*}
u = -\nabla^\perp (- \Delta)^{-1} \rho
\end{equation*}
and $\rho$ represents the (scalar) vorticity of the fluid. Similarly, one may think of the Surface Quasi-Geostrophic (SQG) equation, where $u$ is given by $u = -\nabla^\perp (- \Delta)^{-1/2} \rho$. Both the 2D Euler equation and the SQG problem are instances of the so-called generalized SQG equations (gSGQ in the literature), where the velocity law is given by 
\begin{equation*}
u = \nabla^\perp (- \Delta)^{- \beta/2} m(D) \rho,
\end{equation*}
and $\beta \in \R$ and $m(D)$ is a positive Fourier multiplication operator, such as $m(D) = \log({\rm Id} - \Delta)^\mu$ with $\mu \in \R$ (see \cite{CCCGW} for an example of such a velocity law, or \cite{CW} in the case of a generalized Boussinesq equation). Many such variants exists, but the unifying feature of gSQG equations is the presence of the orthogonal gradient $\nabla^\perp$ in the velocity law, so that $u$ is the image of $\rho$ by a skew-symmetric operator $\nabla^\perp (- \Delta)^{- \beta/2}m(D)$. This endows the non-linearity $u \cdot \nabla \rho$ with a commutator structure and allows for solutions to exist even when the velocity should be \textsl{more} singular than the unknown $0 \leq \beta < 1$ (see \cite{Marchand} or \cite{CCCGW} for explanations). Because the Leray projection operator $\P$ is symmetric, this is not true in the case of \eqref{ieq:fracLapOp}, even in 2D when $\P = \nabla^\perp (- \Delta)^{-1} \curl$.

\medskip

Another type of velocity law consists in abandoning the world of incompressible fluids and using velocity fields that are not divergence free. For example, one may take $u = - \nabla (- \Delta)^{- \beta/2} \rho$, as in \cite{Chae2014}, where the problem has its own particularities. Specially, in the case $\beta = 0$, it reduces to the Hamilton-Jacobi equation
\begin{equation*}
\partial_t \rho + |\nabla \rho|^2 = 0.
\end{equation*}
Other examples would include Burgers equation $u = \rho$ (when $d=1$), magnetogeostrophic equations, or agregation equations (see \cite{BGb} or \cite{BK} for examples).

\medskip

Many of these equations have their own physical relevance, but a number of them are also studied as toy-models for more complex equations: thus the case $u = H\rho$ in dimention $d=1$, where $H$ is the Hilbert transform, is sometimes understood as a model for the 3D Euler equations. On a different note, we point out the proximity between active scalar equations and the 2D non-resistive incompressible magnetohydrodynamics system, which reads
\begin{equation*}
\begin{cases}
\partial_t a + u \cdot \nabla a = 0\\
\big( \partial_t + u \cdot \nabla - \Delta \big) u + \nabla P = (b \cdot \nabla)b\\
\D(u) = 0\\
b = \nabla^\perp a,
\end{cases}
\end{equation*}
where $u(t, x) \in \R^d$ is the velocity of the fluid, $b(t,x)$ is the magnetic field, $P(t,x)$ the MHD pressure and $a(t, x) \in \R$ is the (scalar) magnetic potential. Here, the velocity is no longer given as a linear function of the advected scalar quantity, but as the solution of a (non-linear) Navier-Stokes problem. This last system concentrates huge difficulties, and even existence of global weak solutions remains unknown, despite progess on the front of local well-posedness (see \cite{FMRR1} or \cite{LTY} for instance). The study of the simpler active scalar equations can also be seen as a first step to understand these more sophisticated models.

\subsection{Presentation of the Main Results}

In this paragraph, we give an overview of our main results as well as some of the main ideas involved in the proofs. At the same time, we attempt a comparision with other existing results in the literature.

\subsubsection{Global Weak Solutions: $\alpha > 0$}

Our first result concerns the existence of global weak solutions in Lebesgue spaces when $\alpha > 0$ (see Theorem \ref{t:globalWeakSol}).

\begin{thm}\label{it:globalWeakSol}
Consider $\alpha \in ]0, d[$ and an exponent $p$ such that
\begin{equation}\label{ieq:pCondition}
\frac{2}{1 + \frac{\alpha}{d}} < p < \frac{d}{\alpha}.
\end{equation}
For all $\rho_0 \in L^p(\R^d)$, there exists a global weak solution $\rho_0 \in L^\infty(\R_+ ; L^p(\R^d))$ of \eqref{ieq:AIPM} associated to the initial datum $\rho_0$. 
\end{thm}

\begin{rmk}
The reader will note that an important case is not covered by Theorem \ref{it:globalWeakSol}, that is $\alpha = d = 2$, where \eqref{ieq:AIPM} becomes the 2D Stokes-Transport system for which global well-posedness is known in bounded domains and in the infinite strip $\R \times ]0, 1[$, \cite{Leblond}, or for compactly supported solutions \cite{Grayer}. But this does not really affect much the significance of our result: when $\alpha = d$, we will see that the system in fact has \textsl{global and unique} solutions (see Theorem \ref{it:globalWPBesov} below). The same remark applies for the statements pertaining to local solutions, Theorems \ref{it:localWP}, \ref{it:ContCrit} and \ref{it:lifespanIncrease} below.
\end{rmk}

This result is based on one of the essential distinctive features of our problem \eqref{ieq:AIPM}, that the velocity field is divergence free $\D(u) = 0$. This implies that we have at our disposal the conservation of all $L^p$ norms, namely, provided the velocity field is regular enough,
\begin{equation*}
\| \rho(t) \|_{L^p} = \| \rho_0 \|_{L^p}.
\end{equation*}
Theorem \ref{it:globalWeakSol} combines the conservation of $L^p$ norms with the velocity law of \eqref{ieq:fracLapOp} to provide global estimates on $u$ in spaces of functions that are compactly embedded in $L^p_{\rm loc}$.

\medskip

However, while this is the main principle of the proof, we should point out the two main difficulties. The first one is that the conservation of $L^p$ norms mentioned above can only be taken for granted if the velocity field possesses sufficient regularity. For example, $u$ must be at least $W^{1, 1}_{\rm loc}$ if Di Perna-Lions theory is to be applied, whereas $u$ will have at best an $\alpha$ index of regularity by \eqref{ieq:fracLapOp}. Our way around that will be to first start by constructing $L^2$ solutions by means of a Friedrichs scheme, and then proceeding to general $L^p$ solutions through regularization of the velocity field. But because we are below the regularity level where Di Perna-Lions theory applies, the weak solutions may display some decay of $\| \rho(t) \|_{L^p}$, instead of the conservation we expect from strong solutions.

Secondly, when handling the nonlinear term $u \cdot \nabla \rho = \D(\rho u)$, caution should be taken with the product $\rho u$ when $p < 2$, as it is not obvious it is properly defined. Condition \eqref{ieq:pCondition} is written so that the $s = \alpha$ regularity of $u$ and embeddings of well-chosen function spaces will make the product $\rho u$ well defined.

\medskip

We mention that results of this kind already exist for SQG and gSQG equations. Resnick proved existence of global $L^2$ solutions for SQG in his PhD dissertation \cite{Resnick}, while a result in $L^p$ with $p > 4/3$ is due to Marchand \cite{Marchand}. This has been extended to gSQG by Lazar and Xue, who construct global weak $L^1 \cap L^2$ solutions in \cite{LX}. These results rely on the commutator structure the non-linear terms in gSQG equations present, and hence are capable of covering more singular velocities that we can with the fractional Stokes-Transport equation: we must have $\alpha > 0$.

On the other hand, our methods do not use the structure of the velocity law in any special way: only that $u$ is a divergence free function and that $(- \Delta)^{- \alpha/2} \P g$ is a Fourier multiplier of order $- \alpha$. Our results should also hold with no modification for gSQG equations with velocity law
\begin{equation*}
u = \nabla^\perp (- \Delta)^{- \frac{1}{2}( \alpha + 1)} \rho
\end{equation*}
and $\alpha > 0$ (so that the velocity is not more singular than the density), thus extending, in the case $\alpha > 0$, the $L^1 \cap L^2$ existence of \cite{LX} to $L^p$ solutions with \eqref{ieq:pCondition}.

\subsubsection{Local well-posedness: $\alpha \geq 0$}

We now turn to the question of uniqueness of solutions. Theorem \ref{it:localWP} below states the existence of solutions that are regular enough to be unique.

\begin{thm}\label{it:localWP}
Consider $\alpha \in [0, d[$ as well as indices $s > 0$ and $p \in [1, + \infty]$ such that
\begin{equation*}
p < \frac{d}{\alpha}, \qquad \text{ and } \qquad s \geq 1 - \alpha.
\end{equation*}
For all initial datum $\rho_0 \in B^s_{\infty, 1}(\R^d) \cap L^p (\R^d)$, there exists a time $T > 0$ such that the system \eqref{ieq:AIPM} has a unique solution in that space $\rho \in C^0([0, T[; B^s_{\infty, 1}(\R^d) \cap L^p (\R^d))$. \end{thm}

In addition to the local existence of a unique solution, we also prove a lower bound for the lifespan of solutions and a continuation criterion. These are contained in Theorem \ref{t:localWP} below.

\medskip

The essence of Theorem \ref{it:localWP} is finding \textsl{a priori} estimates at a level where the velocity field is Lipschitz, thus enabling us to perform stability estimates and bound the difference of two solutions. The condition $s \geq 1 - \alpha$ is designed so that the velocity law \eqref{ieq:fracLapOp} will bring $u$ to a $B^1_{\infty, 1} \subset W^{1, \infty}$ level of regularity.

\medskip

Local well-posedness result already exists for many active scalar equations with fractional velocity laws, such as \cite{CCCGW} for gSQG, \cite{CJ} for fractional porous medium flows or \cite{Chae2014} for two different types of problems, including \eqref{ieq:AIPM}. As we have already noted above, the velocity law $u = (- \Delta)^{- \alpha/2} \P g$ is somewhat different from gSQG problems because the latter possess a commutator structure that provides well-posedness even when the velocity is more singular that the unknown. Consequently, existing results on \eqref{ieq:AIPM} only hold for $\alpha \geq 0$, as mentioned in \cite{Chae2014}, and ill-posedness has even been proved for $\alpha < 0$, see \cite{FGSV}.

As far as we are aware, all well-posedness results existing in the literature concern initial data that are in $H^s$ Sobolev spaces, with appropriate $s$. 

\medskip

At this point, we should note that Theorem \ref{it:localWP} is not optimal when $\alpha \geq 1$. In that case, as we will see, solutions are global and we will dedicate a whole Section to global well-posedness with initial data in \textsl{critical} spaces (see below).

\subsubsection{A Directional Continuation Criterion}

So far, none of our results have used any specific feature of the velocity law \eqref{ieq:fracLapOp}, other than incompressibility $\D(u) = 0$ and the order $\alpha$ of the operator. But this does not mean that \eqref{ieq:fracLapOp} has no interesting underlying structure. Here, we use the way gravity stratification affects the velocity to prove a directional continuation criterion.

\begin{thm}\label{it:ContCrit}
Assume that $0 \leq \alpha < 1$. Under the assumptions and notations of Theorem \ref{it:localWP}, the unique solution $\rho \in C^0([0, T[; B^{1 - \alpha}_{\infty, 1}(\R^d) \cap L^p(\R^d))$ may be continued beyond time $T$ if and only if
\begin{equation}\label{ieq:thContcritPartialD}
\int_0^T \big\| \partial_d \rho \big\|_{B^{- \alpha}_{\infty, 1}} \dt < + \infty
\end{equation}
\end{thm}

Understanding \eqref{ieq:thContcritPartialD} is quite simple: by writing explicitly the Leray projection operator $\P = {\rm Id} + \nabla (- \Delta)^{-1} \D$ and remembering that the gravity vector is $g = e_d$, we may express the nonlinearity in \eqref{ieq:AIPM} as
\begin{equation*}
u \cdot \nabla \rho = (- \Delta)^{- \alpha / 2} \rho \, \partial_d \rho + \nabla \rho \cdot \nabla (- \Delta)^{- \alpha / 2} (- \Delta)^{-1} \partial_d \rho.
\end{equation*}
Both products in the above involve a factor that is linearly dependent on the derivative $\partial_d \rho$. This implies that blow-up of the solution is prevented provided $\partial_d \rho$ remains bounded in an appropriate space.

\subsubsection{Global well-posedness: $\alpha \geq 1$}

We now turn our attention to the range $\alpha \geq 1$, in which case solutions are global. Roughly, the idea is that conservation of $L^q$ norms yielded by incompressibility automatically grant $B^\alpha_{q, \infty}$ regularity to the velocity field through \eqref{ieq:fracLapOp}, and the embeddings
\begin{equation*}
B^\alpha_{q, \infty} \subset B^1_{\infty, 1} \subset W^{1, \infty}
\end{equation*}
will hold as long as $q > d/(1 - \alpha)$, hence global Besov-Lipschitz estimates on the velocity field. However, with this in mind, the limit case $q = d/(1 - \alpha)$ must be seen as critical. As a matter of fact, if $\rho \in L^{d/(1 - \alpha)}$, then $(- \Delta)^{- \alpha/2} \P (\rho g)$ will not be in general Lipschitz, but $\log$-Lipschitz. 

Therefore, if we are to prove uniqueness of global $L^{d/(1 - \alpha)}$ solutions, we must turn to the theory of transport equations with non-Lipschitz velocity fields, as in \cite{Danchin2005}, which provides estimates at the cost of a loss of regularity.

\begin{thm}\label{it:globalStrongLebesgue}
Consider $d \geq 2$ and $\alpha \in ]1, d[$. Let $p \in [1, d/\alpha [$ and set $q = d/(\alpha - 1)$. Then, for any initial datum $\rho_0 \in L^p (\R^d) \cap L^q (\R^d)$, problem \eqref{ieq:AIPM} has a unique global solution $\rho \in L^\infty(\R_+ ; L^p(\R^d) \cap L^q(\R^d))$.
\end{thm}

In the recent preprint \cite{MS}, Mecherbet and Sueur obtain essentially the same result in the special case $\alpha = 2$ and $d=3$, albeit by very different methods. While we rely on stability estimates, the authors of \cite{MS} prove well-posedness in $L^1 \cap L^3$ using estimates in the Wasserstein distance $W_1$ to prove uniqueness.

\medskip

Theorem \ref{it:globalStrongLebesgue} has two limitations: two settings of special interest are excluded from its scope. The first one is the case of the 2D Stokes-Transport problem $\alpha = d = 2$. The second is the case $\alpha = 1$, which has the same scaling as the 2D Euler system.

In those endpoint cases, we may obtain global well-posedness, but at the cost of working in smaller (but still critical) function spaces.

\begin{thm}\label{it:globalWPBesov}
Consider $d \geq 2$ and $\alpha \in [1, d]$. Let $p = d/\alpha$ and $q = d/(\alpha - 1)$. Then, for any initial datum $\rho_0 \in \dot{B}^0_{p, 1} (\R^d) \cap B^0_{q, 1} (\R^d)$, problem \eqref{ieq:AIPM} has a unique global solution $\rho \in C^0(\R_+ ; \dot{B}^0_{p, 1} (\R^d) \cap B^0_{q, 1} (\R^d))$.
\end{thm}

The assumption $\rho_0 \in \dot{B}^0_{p, 1}$, stronger than simple $L^p$ integrability, serves to insure that the velocity field is well-defined. The operator $(- \Delta)^{- \alpha}\P g$ has a symbol which is singular at low frequencies, so we require the Fourier transform of $\rho$ to be regular enough that $\what{u}$ be properly defined. When $\alpha = d$, even $\rho_0 \in L^1$ is not enough, and we must resort to the stronger $\rho \in \dot{B}^0_{1, 1} \subsetneq L^1$ condition. In previous works for the 2D Stokes Transport system, \textsl{i.e.} when $\alpha = d = 2$, global well-posedness has been obtained under other integrability conditions, namely a compact support assumption on the initial data \cite{Grayer}, or in the framework of solutions on a bounded domain $\Omega \subset \R^2$ or the infinite strip $\Omega = \R \times ]0, 1[$, see \cite{Leblond}.

Concerning the $B^0_{q, 1}$ assumption, it has two uses. The first one is that for $\dot{B}^0_{p, 1}$ regularity to be propagated by the transport equation, we need a Lipschitz vector field. Therefore, working with initial data in $L^{d/(\alpha - 1)}$ will not do, as it will only give a $\log$-Lipschitz vector field. Secondly, in the case where $\alpha = 1$, finding uniqueness in $L^{d/(\alpha - 1)} = L^\infty$ seems fraught with extreme difficulties.

This may seem surprising, given the closeness of the case $\alpha = 1$ to the 2D Euler problem, where uniqueness with $L^\infty$ vorticity is well-known. However, all $L^\infty$ uniqueness results we are aware of for 2D Euler (such as the method of Yudovich \cite{Yudovich} or that of Serfati \cite{Serfati}, \cite{AKLNH}) rely the fact that the vorticity equation is the curl of a lower order equation: the original Euler system. This ``integration'' of the vorticity equation is not available in our case.

As we present Theorem \ref{t:globalWPBesov} below, we give additional details and comments on why the proof used for Theorem \ref{it:globalStrongLebesgue} does not work in the case $\alpha = 1$. We refer to the discussion below.

\subsubsection{Lifespan of Solutions for $\alpha \rightarrow 1^-$}

To conclude this article, we study the way the lifespan of the solution depends on the value of $\alpha < 1$. The idea is quite simple: since the equation is globally well-posed for $\alpha = 1$, then continuous dependency on the parameters should imply that the lifespan becomes arbitrarily large as $\alpha \rightarrow 1^-$. This is Theorem \ref{it:lifespanIncrease} below.

\begin{thm}\label{it:lifespanIncrease}
Assume that $d \geq 2$ and consider an exponent $p \in [1, d[$. For any $\frac{1}{2} < \alpha < 1$ and initial datum $\rho_0 \in X := L^{p}(\R^d) \cap B^{1-\alpha}_{\infty, 1} (\R^d)$, the lifespan $T^*$ of the unique associated solution of \eqref{ieq:AIPM} is at least
\begin{equation*}
T^* \geq \frac{C}{\| \rho_0 \|_{X}}\log \left( 1 + \frac{C}{1 - \alpha} \right).
\end{equation*}
\end{thm}

In fact, in Theorem \ref{t:lifespanIncrease} below is even more general, as we study the lifespan around an equilibrium solution $R(x_d)$. This amounts to adding an extra linear term to the equation and changes the inequality above to account for the norm of derivatives of $R(x_d)$.

\medskip

Theorem \ref{it:lifespanIncrease} is not the first of its kind: several authors have proved similar ``asymptotically global well-posedness'' results. For instance, Danchin and Fanelli \cite{DF} study the 2D non-homogeneous incompressible Euler equations and find the lifespan to be arbitrarily large in the limit of constant densities, with a similar logarithmic lower bound (see also \cite{Sbaiz} by Sbaiz for the case of the 2D non-homogeneous Euler equations with fast rotation), while Fanelli and Liao do so in \cite{FL} for a zero Mach number limit system.

Similarly, Fanelli and the present author examine the case of 2D ideal magnetohydrodynamics in \cite{CF3} and \cite{CF4} in the limit of weak magnetic fields, but because of the complexity of the system, the method only yields a triply iterated logarithmic bound. 

\medskip

The proof of Theorem \ref{it:lifespanIncrease} is a bit different from the lower bounds from the other works mentioned here. These usually rely on estimates for the transport equation, in Besov spaces of regularity $s=0$, that are linear with respect to the transport field, as in the proof of global well-posedness for the 2D Euler problem in $B^1_{\infty, 1}$ (see the articles of Vishik \cite{Vis} and Hmidi-Keraani \cite{HK}). In our setting, we have to adapt these estimates to a regularity level of $\epsilon = 1 - \alpha \rightarrow 0^+$.

\subsubsection{Outline of the paper}

We start by defining precisely the notion of weak solution of \eqref{ieq:AIPM} in Section \ref{s:DefweakSol}. This requires some clarification on the Stokes problem, and we explain how $u$ can be uniquely determined by $\rho$.

In Sections \ref{s:globalWeak} to \ref{s:lifespan}, we state comment and prove all of our results in the order they were mentioned in this Introduction. 

The paper ends with an Appendix containing definitions related to Littlewood-Paley decompositions, Besov spaces and various results on transport equations.

\subsubsection{Notation}

We adopt the following notation throughout the paper.

\begin{itemize}
\item The notation $C$ will be used to denote constants whose precise value are irrelevant to the computations, and which may differ from one line to the next. When needed, we may indicate dependency on a given parameter by writing the constant as $C(\, . \,)$.

\item We note $f \lesssim g$ for $f \leq C g$. The notation $f \approx g$ will mean $\frac{1}{C}f \leq g \leq Cf$.

\item Unless otherwise mentioned, all derivatives or pseudo-differential operators refer to the space variables only. For example, if $f : (t, x) \in \R \times \R^d \longmapsto f(t,x) \in \R$, the Laplacian $\Delta f$ is given by $\Delta f(t,x) = \partial_1^2 f(t,x) + \cdots + \partial_d^2 f(t,x)$.

\item If $X$ is a Banach space and $T > 0$, we note $C^0_T(X) := C^0([0, T[ ; X)$, or $C^0(X) := C^0(\R_+ ; X)$ when $T = +\infty$. We define the spaces $L^p_T(X)$, $L^p(X)$, $W^{1, \infty}_T(X)$, $W^{1, \infty}(X)$, \textsl{etc}. in the same way.

\item In order to simplify notations, we will omit the reference to the $\R^d$ when using function spaces, so we write for example $L^p := L^p(\R^d)$, $H^s := H^s(\R^d)$ and so on. For instance, we have
\begin{equation*}
C^0_T (L^p) = C^0([0, T[ ; L^p(\R^d)).
\end{equation*}

\item We note $D = -i \nabla$. More generally, for $\phi : \R^d \tend \R^d$, the Fourier multiplication operator $\phi(D)$ is defined by the Fourier transform
\begin{equation*}
\forall f \in \mc S, \qquad \what{\phi(D)f}(\xi) = \phi(\xi) \what{f}(\xi).
\end{equation*}
\item We note $C^0_b(X)$ the space of continuous and bounded functions on $\R_+$ with values in a Banach space $X$, while $C^0(X)$ without the subscript refers to the space of continuous functions $\R_+ \tend X$ that are not necessarily bounded.
\end{itemize}

\subsubsection*{Acknowledgements}

The author is grateful to Francesco Fanelli for many long discussions on this problem and his constant help throughout the project. Some of the ideas behind Theorem \ref{t:globalWeakSol} were kindly suggested by Rafael Granero-Belinch\'on. 

This work has been partially supported by the Deutsche Forschungsgemeinschaft (DFG, German Research Foundation) Project ID 211504053 - SFB 1060, and by the project CRISIS (ANR-20-CE40-0020-01), operated by the French National Research Agency (ANR).

\section{A Definition of Weak solution}\label{s:DefweakSol}

In this paragraph, we define the notion of weak solution of \eqref{ieq:AIPM} and give the precise meaning of the fractional Laplace operator we use throughout the article.

\subsection{The Fractional Stokes Equation}

We focus for a while on the fractional Stokes equation:
\begin{equation}\label{eq:StokesEQ}
\begin{cases}
(- \Delta)^{\alpha/2} u + \nabla \pi = \rho g\\
\D(u) = 0.
\end{cases}
\end{equation}
Here, $\rho$ is considered as a datum for the system while the unknowns are $u$ and $\pi$. The fractional Laplace operator is defined by its Fourier transform (see for example Proposition \ref{p:FracLapBesov}).

The first remark we make is that solutions of this system are in general highly non-unique. For example, in the case of the Stokes system $\alpha = 2$, the velocity $u$ is only given up to the addition of a (divergence free) harmonic function, and $\pi$ up to the addition of a constant. But on top of that, there is a further ambiguity that cannot be reduced to harmonic or constant summands: for example, if $\rho = 1$, then the pairs
\begin{equation*}
\begin{split}
& \pi(x) = x \cdot g \qquad u(x) = 0 \\
& \pi(x) = 0, \qquad u(x) = -\frac{1}{2d}|x|^2 g
\end{split}
\end{equation*}
are both solutions of \eqref{eq:StokesEQ}. This phenomenon is due to the lack of a far-field condition, such as decay at $|x| \rightarrow + \infty$, and is not specific to the Stokes system. We refer to \cite{Cobb1} and the references therein for a discussion of a similar problem in the Euler equations set on $\R^d$.

\medskip

In order to map the density $\rho$ to a unique velocity field that is solution of \eqref{eq:StokesEQ}, we impose sufficient decay at $|x| \rightarrow + \infty$. Our goal in this paragraph is to prove the following result, which involves the space $\mc S'_h$ defined in the Appendix (see Definition \ref{d:SpH}).

\begin{prop}\label{p:StokesExUn}
Consider $d \geq 2$. The following assertions hold:
\begin{enumerate}[(i)]
\item Assume that $\rho \in L^p$ with $p \in [1, d/p[$. Then there is a unique $u \in \mc S'_h$ such that \eqref{eq:StokesEQ} is satisfied for some pressure function $\pi$. In addition, $\pi$ is uniquely determined up to a constant.
\item The same holds when $\rho \in \dot{B}^0_{p, 1}$ with $p = \frac{d}{\alpha}$.
\end{enumerate}
\end{prop}

\begin{proof}
Formally, applying the Leray projection operator $\P = {\rm Id} + \nabla (- \Delta)^{-1} \D$ to the first equation in \eqref{eq:StokesEQ} gives the velocity as a function of the density:
\begin{equation*}
u = (- \Delta)^{- \alpha/2} \P (\rho g).
\end{equation*}
The obvious problem is that the operator $(- \Delta)^{- \alpha/2} \P g$ may not be well defined, as it is a Fourier multiplier whose symbol is singular at $\xi = 0$. We set
\begin{equation*}
v := \sum_{j \in \Z} \dot{\Delta}_j (- \Delta)^{- \alpha/2} \P (\rho g).
\end{equation*}
Now, thanks to the embedding $L^p \subset \dot{B}^0_{p, \infty}$, which holds because $p < + \infty$, Lemma \ref{l:FourierMultiplier} shows that this sum defines an element of the homogeneous Besov space $\dot{B}^\alpha_{p, \infty}$. In particular, we have $v \in \mc S'_h$.

Similarly, note that the pressure force $F = \nabla \pi$ can be defined as the sum of its homogeneous Littlewood-Paley decomposition:
\begin{equation*}
F = \sum_{j \in \Z} \dot{\Delta}_j \nabla (- \Delta)^{-1} \D (\rho g),
\end{equation*}
and this sum defines an element of the homogeneous Besov space $F \in \dot{B}^0_{p, \infty}$. The existence of a pressure function $\pi \in \mc S'$ such that $F = \nabla \pi$ follows from, for example, Lemma 3.1 in \cite{CF3}.

\medskip

Now consider another solution $(v', \pi')$ of the problem such that $v' \in \mc S'_h$ and let $F' = \nabla \pi'$. Because the symbol of the operator $(- \Delta)^{- \alpha/2} \P g$ is smooth outside of $\xi = 0$, we see that the Fourier transforms of $v$ and $v'$ must agree on $\R^d \setminus {0}$, namely
\begin{equation*}
\what{v}(\xi) = \what{v'}(\xi) \qquad \text{in } \mc D'(\R^d \setminus {0}),
\end{equation*}
so the difference $v-v'$ must be reduced to a polynomial. As there is no polynomial function in $\mc S'_h$, we conclude that $v = v'$, and hence $F = F'$ so the pressure functions $\pi$ and $\pi'$ are equal up to a constant summand.
\end{proof}

The consequence the computations in the proof of Proposition \ref{p:StokesExUn} is that we have an explicit representation of the velocity field. This makes it possible to bound it in appropriate norms.

\begin{prop}\label{p:LPoperatorBound}
Consider $0 \leq \alpha \leq d$ and assume that $\rho : \R^d \tend \R$ is a density such that one of the assumptions of Proposition \ref{p:StokesExUn} holds. Then, the uniquely determined velocity field $u \in \mc S'_h$ satisfies the following properties:
\begin{enumerate}[(i)]
\item if $\rho \in L^p$ for some $p < d/\alpha$ then $\| \Delta_{-1} u \|_{L^\infty} \lesssim \| \rho \|_{L^p}$. Moreover, if $\rho \in L^p \cap B^s_{\infty, 1}$ for some $s \in \R$ then
\begin{equation*}
\| u \|_{B^{s+\alpha}_{\infty, 1}} \lesssim \| \rho \|_{L^p} + \| \rho \|_{B^s_{\infty, 1}} \, ;
\end{equation*}
\item if $\rho \in \dot{B}^0_{p, 1}$ for $p = d/p$ then $\| \Delta_{-1} u \|_{L^\infty} \lesssim \| \rho \|_{\dot{B}^0_{p, 1}}$. Moreover, if $\rho \in \dot{B}^0_{p, 1} \cap B^s_{\infty, 1}$ for some $s \in \R$ then
\begin{equation*}
\| u \|_{B^{s+\alpha}_{\infty, 1}} \lesssim \| \rho \|_{\dot{B}^0_{p, 1}} + \| \rho \|_{B^s_{\infty, 1}}.
\end{equation*}

\end{enumerate}

\end{prop}

\begin{proof}
The proof is fairly straightforward. Start by assuming that $\rho \in L^p$ for some $p < d/ \alpha$. Then, as $u$ belongs to the Besov space $\dot{B}^\alpha_{p, \infty}$ the series
\begin{equation}
\Delta_{-1} u = \sum_{j \leq 1} \dot{\Delta}_j (- \Delta)^{- \alpha/2} \P (\rho g)
\end{equation}
is normally convergent in $L^\infty$, and we deduce that $\| \Delta_{-1} u \|_{L^\infty} \lesssim \| \rho \|_{L^p}$. On the other hand, if we also have $\rho \in B^s_{\infty, 1}$ for some $s \in \R$, then it follows from Lemma \ref{l:FourierMultiplier} that
\begin{equation*}
\begin{split}
\| u \|_{B^s_{\infty, 1}} & \lesssim \| \Delta_{-1} u \|_{L^\infty} + \sum_{j \geq 0} 2^{j(s+\alpha)} \big\| \Delta_j (- \Delta)^{- \alpha/2} \P (\rho g) \big\|_{L^\infty} \\
& \lesssim \| \rho \|_{L^p} + \sum_{j \geq 0} 2^{js} \| \Delta_j \rho \|_{L^\infty} \\
& \lesssim \| \rho \|_{L^p} + \| \rho \|_{B^s_{\infty, 1}}.
\end{split}
\end{equation*}
This proves point \textit{(i)} of the Proposition. Point \textit{(ii)} follows from nearly identical computations.
\end{proof}

\subsection{Defining Weak Solutions}

In this paragraph, we define the notion of weak solution of problem \eqref{ieq:AIPM}. Throughout this article, we will work with solutions which do not necessarily have locally integrable derivatives, so we will always understand the non-linear convective term in the sense of the divergence form:
\begin{equation*}
u \cdot \nabla \rho = \D(\rho u),
\end{equation*}
which is possible thanks to incompressibility $\D(u) = 0$. Even then, we will have to take care that the product $\rho u$ is always properly defined.

\begin{defi}
Consider $ \alpha \in [0, d]$ and $q \in [1, + \infty]$ such that $q > \frac{2}{1 + \frac{\alpha}{d}}$ and $T > 0$. A function $\rho \in L^2_{\rm loc}([0, T[ ; L^q)$ is said to be a weak solution of the fractional Stokes-Transport problem \eqref{ieq:AIPM} with initial datum $\rho_0 \in L^q$ if
\begin{enumerate}[(i)]
\item we have $\rho \in L^2_{\rm loc}([0, T[ ; L^p)$ for some $p \in [1, d/\alpha [$ or $\rho \in L^2_{\rm loc}([0, T[ ; \dot{B}^0_{p, 1})$ for $p = d/\alpha$ ;
\item the density $\rho$ solves the weak form of the transport equation
\begin{equation*}
\forall \phi \in \mc D ([0, T[ \times \R^d ; \R), \qquad \iint \Big\{ \rho \partial_t \phi + \rho u \cdot \nabla \phi \Big\} \dx \dt + \int \rho_0(x) \phi(0, x) \dx = 0
\end{equation*}
with the velocity field given, for almost all times, as the unique $\mc S'_h$ solution of the Stokes system \eqref{eq:StokesEQ}, as specified by Proposition \ref{p:StokesExUn} above. In that case, it follows from Lemma \ref{l:LQproduct} below (whose proof is independent from this definition) that the product $\rho u$ is well defined as an element of the space $\rho u \in L^1_{\rm loc}([0, T[ ; B^{\alpha - d/q}_{q, \infty})$.
\end{enumerate}

\end{defi}

\section{Global Weak Solutions for $\alpha > 0$}\label{s:globalWeak}

\begin{thm}\label{t:globalWeakSol}
Consider $\alpha \in ]0, d[$ and an exponent $p$ such that
\begin{equation}\label{eq:pCondition}
\frac{2}{1 + \frac{\alpha}{d}} < p < \frac{d}{\alpha}.
\end{equation}
For all $\rho_0 \in L^p$, there exists a global weak solution $\rho \in L^\infty(L^p)$ of \eqref{ieq:AIPM} associated to the initial datum $\rho_0$. In addition, for any $q \in [1, +\infty]$, if $\rho_0 \in L^q$ then this solution satisfies
\begin{equation}\label{eq:LqBound}
\| \rho \|_{L^\infty(L^q)} \leq \| \rho_0 \|_{L^q}.
\end{equation}
\end{thm}

\begin{rmk}
We remark that condition \eqref{eq:pCondition} is not void because $0 < \alpha < d$. In addition, as long as $\alpha < d/2$, the value $p=2$ always works. It should also be noted that the bound \eqref{eq:LqBound} remains valid in the case where $q = 1$.
\end{rmk}

\begin{rmk}
The case of a plane problem $d=2$ could seem to be a major limitation of our result, as is will not then apply to $\alpha = 2$, which is a value of special interest. However, we will have other means of proving \textsl{global} well-posedness whenever $\alpha \geq 1$ (see Theorem \ref{t:globalWPBesov} below), so we should not be too worried about this issue.
\end{rmk}

Before giving the full proof, let us sketch ou argument. The main difficulty is the propagation of $L^p$ norms in the transport equation at such a low level of regularity for $u$. For instance, Di Perna-Lions theory requires the velocity field to be at least $W^{1, 1}_{\rm loc}$, and such a bound is unavailable here, as $u$ will have a regularity level of at most $\alpha$.

In order to replace Di Perna-Lions theory, we will Sarto by using energy estimates and a Friedrichs scheme to construct global $L^\infty(L^2)$ solutions. Incidentally, it should be noted that fixed-point arguments (a.k.a. iterative schemes) do not work at this level of regularity, as they require stability estimates to hold. This will be the first part of the proof.

Secondly, in order to propagate $L^p$ norms, we work with a regularization of the velocity field and construct approximate solutions that are uniformly bounded in $L^\infty(L^p)$. The proof will then be a matter of finding the estimates we need to obtain convergence of the approximate solutions.

\medskip

We therefore start by the case where $p = 2$. In order to prepare the way for further estimates, we already introduce a regularization of the velocity field.

\begin{prop}\label{p:L2weakSol}
Consider a Fourier multiplier $T = \phi(D)$ whose symbol $\phi(\xi)$ is a $C^\infty$ function with compact support bounded away from $\xi = 0$ such that $|\phi| \leq 1$ and let $\rho_0 \in L^2$. Then the PDE system
\begin{equation}\label{eq:L2appSyst}
\begin{cases}
\partial_t \rho + \D(\rho u) = 0\\
u = T \P (- \Delta)^{- \alpha/2} (\rho g)
\end{cases}
\end{equation}
has a global weak solution $\rho \in L^\infty(L^2)$ associated to the initial datum $\rho_0$.
\end{prop}

\begin{rmk}
Although the proof of this Proposition is by no means challenging, the PDE problem \eqref{eq:L2appSyst} still is a fully non-linear set of equations, so that the existence of solutions cannot be considered as totally trivial.
\end{rmk}

\begin{proof}[Proof (of the Proposition)]
We start by noting that, by assumption on the function $\phi(\xi)$, the operator $T \P (- \Delta)^{- \alpha / 2}g$ has a smooth, bounded and compactly supported symbol. It therefore is $L^2 \tend B^k_{\infty, 1}$ continuous for all $k > 0$. This fact will be of constant use throughout the proof.

\medskip

\textbf{STEP 1: approximate system.} In order to implement a Friedrichs scheme, we introduce the following set of approximate equations: for all $n \geq 1$, consider
\begin{equation}\label{eq:L2ODE}
\begin{cases}
\partial_t \rho_n + A_n \D (\rho_n u_n) = 0 \\
u_n = T \P (- \Delta)^{- \alpha/2} (\rho_n g),
\end{cases}
\qquad \text{with initial data } \rho_n(0) = A_n \rho_0.
\end{equation}
Here, $A_n$ is the spectral projector on the ball $|\xi| \leq n$. More precisely, for any $f \in L^2$, the function $A_n f \in L^2$ has Fourier transform $\mathds{1}_{|\xi| \leq n} \what{f}(\xi)$. We wish to make use of the Cauchy-Lipschitz theorem to construct solutions of \eqref{eq:L2ODE} in the Banach space $L^2$. We must therefore prove that $\partial_t \rho_n = - A_n \D (\rho_n u_n)$ is a locally Lipschitz function of $\rho_n$ in $L^2$. 

Firstly, since the function $A_n \D (\rho_n u_n)$ has a compactly supported Fourier transform, the first Bernstein inequality gives
\begin{equation*}
\big\| A_n \D (\rho_n u_n) \big\|_{L^2} \leq C(n) \| \rho_n u_n \|_{L^2} \leq C(n) \| u_n \|_{L^\infty} \| \rho_n \|_{L^2}.
\end{equation*}
Next, thanks to the boundedness properties of the operator $T \P (- \Delta)^{- \alpha / 2}g$ mentioned at the beginning of the proof, we get
\begin{equation*}
\big\| A_n \D (\rho_n u_n) \big\|_{L^2} \leq C(n) \| \rho_n \|_{L^2}^2,
\end{equation*}
so that the quantity $A_n \D (\rho_n u_n)$ is indeed locally Lipschitz in $L^2$ with respect to $\rho_n$.

Consequently, the Cauchy-Lipschitz theorem provides, for every $n \geq 1$, a solution of \eqref{eq:L2ODE} $\rho_n \in C^0([0, T_n[;L^2)$ associated to the initial datum $\rho_0$, with positive lifespan $T_n > 0$.

\medskip

\textbf{STEP 2: uniform estimates.} Because the velocity fields $u_n$ are divergence-free and smooth, we may test the evolution equation in \eqref{eq:L2ODE} against $\rho_n$ and obtain conservation of the $L^2$ norms of the approximate solutions:
\begin{equation*}
\forall t \in [0, T_n[, \qquad \| \rho_n(t) \|_{L^2} = \| \rho_n(0) \|_{L^2} \leq \| \rho_0 \|_{L^2}.
\end{equation*}
As a consequence, the approximate solutions $\rho_n$ provided by the Cauchy-Lipschitz theorem do not blow-up in $L^2$, and so they must be global $T_n = +\infty$.

We deduce from the continuity of the operator $T\P (- \Delta)^{- \alpha /2} g : L^2 \tend B^k_{\infty, 1}$ that the sequence of velocity fields $(u_n)_n$ is uniformly bounded in $L^\infty(B^k_{\infty, 1})$ for every $k > 0$. In order to also provide time compactness, we use the equation and bound the time derivative $\partial_t u_n$. First, we see that
\begin{equation*}
\| \partial_t \rho_n \|_{H^{-1}} = \big\| A_n \D (\rho_n u_n) \big\|_{H^{-1}} \leq \| \rho_n \|_{L^2} \| u_n \|_{L^\infty} \lesssim \| \rho_0 \|_{L^2}^2,
\end{equation*}
so that, since the operator $T \P (- \Delta)^{- \alpha / 2}g$ has a smooth and compactly supported symbol, the Bernstein inequalities provide, for any $k > 0$,
\begin{equation*}
\| \partial_t u_n \|_{B^k_{\infty, 1}} \lesssim \| \partial_t \rho_n \|_{H^{-1}} \lesssim \| \rho_0 \|_{L^2}^2,
\end{equation*}
and the sequence $(u_n)_n$ is therefore uniformly bounded in the space $W^{1, \infty}_T(B^k_{\infty, 1})$ for every $T > 0$.

\medskip

\textbf{STEP 3: convergence to a solution.} Thanks to the uniform bounds we have obtained above, namely
\begin{equation*}
\| \rho_n \|_{L^\infty(L^2)} \leq \| \rho_0 \|_{L^2} \qquad \text{and} \qquad \| u_n \|_{W^{1, \infty}_T(B^k_{\infty, 1})} \leq \| \rho_0 \|_{L^2} + \| \rho_0 \|_{L^2}^2
\end{equation*}
for every $T > 0$ and $k > 0$, we have convergence of the sequences $(\rho_n)_n$ and $(u_n)_n$ up to an extraction: there is a $\rho \in L^\infty (L^2)$ and a $u \in L^\infty_{\rm loc}(L^\infty)$ such that, as $n \rightarrow + \infty$,
\begin{equation*}
\begin{split}
& \rho_n \wtend^* \rho \qquad \text{in } L^\infty(L^2), \\
& \rho_n(0) \rightarrow \rho_0 \qquad \text{in } L^2, \\
& u_n \rightarrow u \qquad \text{in } L^\infty_T(L^\infty_{\rm loc}).
\end{split}
\end{equation*}
for every $T > 0$. We now study the convergence of all the terms in the PDE to show that the limit $(\rho, u)$ is in fact a solution of the system. Firstly, we have,
\begin{equation*}
\partial_t \rho_n \tend \partial_t \rho \qquad \text{in } \mc D'(\R_+^* \times \R^d).
\end{equation*}
Next, we study the convergence of the nonlinear term $\rho_n u_n$. Let $\psi \in \mc D (]0, T[ \times \R^d)$. Then, if $K \subset \R^d$ is a compact set such that ${\rm supp}\big( \psi(t) \big) \subset K$ for all times $t \in ]0, T[$,
\begin{equation*}
\begin{split}
\left| \langle \rho_n u_n - \rho u, \psi \rangle \right| & \leq \| u_n - u \|_{L^\infty(K)} \| \psi \|_{L^2} \| \rho_n \|_{L^2} + \left| \langle \rho_n - \rho, u \psi \rangle \right| \\
& \leq \| u_n - u \|_{L^\infty(K)} \| \psi \|_{L^2} \| \rho_0 \|_{L^2} + \left| \langle \rho_n - \rho, u \psi \rangle \right|
\end{split}
\end{equation*}
On the one hand, $u_n \tend u$ in the space $L^\infty_T \big(L^\infty(K) \big)$, so the first summand tends to zero, and on the other hand, the fact that $u \psi \in L^\infty_T(L^2)$ insures that the bracket converges to zero. We have shown that $\rho$ solves the equation
\begin{equation*}
\partial_t \rho + \D(\rho u) = 0 \qquad \text{in } \mc D'(\R_+^* \times \R^d),
\end{equation*}
and the assertion concerning the initial datum follows from the convergence $A_n \rho_0 \tend \rho_0$ in $L^2$.
\end{proof}

Now that we have constructed $L^2$ solutions, we may now finish the proof of Theorem \ref{t:globalWeakSol}. 

\begin{proof}[Proof (of Theorem \ref{t:globalWeakSol})]
Our argument runs pretty much along the same lines as that of Proposition \ref{p:L2weakSol}, starting from a set of approximate solutions whose convergence must be shown, except that this time much more care is required to find estimates in the correct spaces.

\medskip

\textbf{STEP 1: approximate solutions.} We use Proposition \ref{p:L2weakSol} in order to construct a family of approximate solutions. For any $N \geq 1$, we define a sequence of initial data $\rho_{0, N}$ by
\begin{equation*}
\rho_{0, N}(x) = \Psi(x/N) S_N \rho_0(x),
\end{equation*}
where $\Psi \in \mc D$ is such that $0 \leq \Psi \leq 1$ and $\Psi(x) = 1$ for all $|x| \leq 1$, and $S_N$ is the Littlewood-Paley operator defined in Subsection \ref{ss:LP} of the appendix. As a consequence, $\rho_{0, N} \in L^2$ for every $N \geq 1$ and $\| \rho_{0, N} \|_{L^r} \leq \| \rho_0 \|_{L^r}$ for every $N \geq 1$ and $r \in [1, + \infty]$. In addition, we define the smoothing operator $T_N$ by
\begin{equation*}
T_N := \sum_{-N}^N \dot{\Delta}_j = \chi(2^{-N} D) \big({\rm Id} - \chi(2^N D) \big),
\end{equation*}
where $\chi(\xi)$ is the Littlewood-Paley function from Subsection \ref{ss:LP}. Note that the operator $T_N$ is uniformly bounded in the $L^r \tend L^r$ topology, as it is a composition of scaled versions of the Fourier multiplier $\chi(D) = \Delta_{-1}$.

We consider the following approximate system, for which the velocity field has been regularized thanks to the operator $T_N$. For every $N \geq 1$, consider the problem
\begin{equation}\label{eq:LPapprox}
\begin{cases}
\partial_t \rho_N + \D (\rho_N u_N) = 0\\
u_N = T_N \P (- \Delta)^{- \alpha / 2} (\rho_N g),
\end{cases}
\qquad \text{with initial data } \rho_N(0) = \rho_{0, N}.
\end{equation}
The existence of global solutions $\rho_N \in L^\infty(L^2)$ is guaranteed by Proposition \ref{p:L2weakSol}.

\medskip

\textbf{STEP 2: uniform estimates.} First of all, the $\rho_N$ solve pure transport equations with divergence-free velocity fields $u_N$. Since, by construction, each $u_N$ is $C^\infty$ smooth, all the Lebesgue norms of the solutions are conserved: for all $r \in [1, + \infty]$,
\begin{equation*}
\forall t \geq 0, \qquad \| \rho_N(t) \|_{L^r} = \| \rho_{0, N} \|_{L^r} \lesssim \| \rho_0 \|_{L^r}.
\end{equation*}

Now, we seek uniform estimates on the velocity field in order to provide compactness on the sequence $(u_N)_N$. Here, unlike in the proof of Proposition \ref{p:L2weakSol} above, we have to deal with the full operator $(- \Delta)^{- \alpha/2} \P g$ whose symbol involves a singularity at low frequencies $\xi = 0$. To better understand the specific low and high frequency issues, we split $u_N$ into
\begin{equation*}
u_N = \Delta_{-1} u_N + ({\rm Id} - \Delta_{-1}) u_N
\end{equation*}
and will prove compactness separately on each part. On the one hand, because of Proposition \ref{p:LPoperatorBound}, we may handle the low-frequency part
\begin{equation}\label{eq:LFvelocitySpace}
\| \Delta_{-1} u_N \|_{L^\infty} = \big\| \Delta_{-1} T_N (- \Delta)^{- \alpha/2} \P (\rho_N g) \big\|_{L^\infty} \lesssim \| \rho_N \|_{L^p} \leq \| \rho_0 \|_{L^p}.
\end{equation}
And since $\Delta_{-1} u_N$ has a compactly supported Fourier transform, this means that we in fact have the uniform bound $\| \Delta_{-1} u_N \|_{B^k_{\infty, 1}} \lesssim \| \rho_0 \|_{L^p}$. Concerning the high frequency part, we have, thanks to Lemma \ref{l:FourierMultiplier},
\begin{equation}\label{eq:HFvelocitySpace}
\big\| ({\rm Id} - \Delta_{-1}) u_N \big\|_{B^\alpha_{p, \infty}} \lesssim \big\| ({\rm Id} - \Delta_{-1}) \rho_N \big\|_{B^0_{p, \infty}} \lesssim \| \rho_0 \|_{L^p}.
\end{equation}

These two estimates provide space compactness on the velocity field. In order to achieve time compactness, we will need to estimate the derivative $\partial_t \rho_N$. This requires a careful analysis of the product $\rho_N u_N$. Once again, we use the low and high frequency decomposition of the velocity field and write
\begin{equation}\label{eq:productLHFdecomposition}
\rho_N u_N = \rho_N \Delta_{-1} u_N + \rho_N ({\rm Id} - \Delta_{-1})u_N.
\end{equation}
Estimating the first of these two terms is straightforward: because $\Delta_{-1} u_N$ is already known to be bounded in $L^\infty$, we have, by \eqref{eq:LFvelocitySpace}
\begin{equation*}
\| \rho_N \Delta_{-1}u_N \|_{L^p} \leq \| \Delta_{-1} u_N \|_{L^\infty} \| \rho_N \|_{L^p} \lesssim \| \rho_0 \|_{L^p}^2.
\end{equation*}
For the other summand, things are much more delicate: the function $({\rm Id} - \Delta_{-1})u_N$ has $B^\alpha_{p, \infty}$ regularity, and it is not obvious that the function product is well defined in $B^0_{p, \infty} \times B^\alpha_{p, \infty}$. This is the purpose of the following Lemma.

\begin{lemma}\label{l:LQproduct}
Consider $\beta > 0$ and $q \in [1, + \infty]$ such that the inequality
\begin{equation}\label{eq:LQproductCondition}
\frac{2}{1 + \frac{\beta}{d}} < q
\end{equation}
holds. Then, for all $(f, h) \in B^0_{q, \infty} \times B^\beta_{q, \infty}$, the product $fh$ is well defined as an element of $B^{\beta - d/q}_{q, \infty}$ and we have the inequality
\begin{equation*}
\| fh \|_{B^{\beta - d/q}_{q, \infty}} \lesssim \| f \|_{B^0_{q, \infty}} \| h \|_{B^\beta_{q, \infty}}
\end{equation*}
\end{lemma}

\begin{proof}[Proof (of the Lemma)]
The proof is based on the Bony decomposition for the product $fh$, which reads
\begin{equation*}
 fh = T_f(h) + \mc T_h(f) + \mc R(f, h).
\end{equation*}
The two paraproducts create no issue. Proposition \ref{p:op} yields
\begin{equation*}
\begin{split}
\| T_f(h) + \mc T_h(f) \|_{B^{\beta - d/q}_{q, \infty}} & \lesssim \| f \|_{B^{-d/q}_{\infty, \infty}} \| h \|_{B^{\beta}_{q, \infty}} + \| f \|_{B^{0}_{q, \infty}} \| h \|_{B^{\beta - d/q}_{\infty, \infty}} \\
& \lesssim \| f \|_{B^0_{q, \infty}} \| h \|_{B^\beta_{q, \infty}},
\end{split}
\end{equation*}
thanks to the embedding properties of Besov spaces (see Proposition \ref{p:BesovEmbed} in the appendix). However, the remainder $\mc R (f, h)$ is more difficult to deal with, because it requires simultaneous conditions on the regularity and intergrability exponents of $f$ and $h$ simply to be defined. There are two distinct settings. The first one is when $q \geq 2$, in which case Proposition \ref{p:op} and the condition $\beta > 0$ insure that the remainder operator 
\begin{equation*}
    \mc R : B^0_{q, \infty} \times B^\beta_{q, \infty} \tend B^\beta_{\frac{q}{2}, \infty} \subset B^{\beta - d/q}_{q, \infty}
\end{equation*}
is bounded. In the case where we on the contrary have $q < 2$, we increase the integrability indices by means of embeddings of Besov spaces. Let $r \geq q$ whose precise value we will fix later on. We have $B^\beta_{q, \infty} \subset B^\sigma_{r, \infty}$ for
\begin{equation*}
\sigma := \beta - d \left( \frac{1}{q} - \frac{1}{r} \right).
\end{equation*}
For the remainder operator $\mc R$ to be defined on $B^0_{q, \infty} \times B^\sigma_{r, \infty}$, Proposition \ref{p:op} requires that $\sigma > 0$ and $1/q + 1/r \leq 1$. We take the optimal value: set the value of $r$ such that
\begin{equation}
\frac{1}{q} + \frac{1}{r} = 1.
\end{equation}
Then the condition $\sigma > 0$ can be written
\begin{equation*}
\frac{2}{q} < 1 + \frac{\beta}{d},
\end{equation*}
which is an equivalent form of \eqref{eq:LQproductCondition}.
\end{proof}

\begin{rmk}
In fact, it is also possible to make sense of the product $fh$ when $q = q_0 := 2 \left( 1 + \beta/d \right)^{-1}$, assuming that $(f, h) \in L^q \times H^\beta_q$, where $H^\beta_q$ is the potential Sobolev space associated to the norm $\| ({\rm Id} - \Delta)^{s/2} h \|_{L^q}$. In that case, the refined embeddings of Proposition \ref{p:refinedEmbed} give, as $q < 2$, the inclusion $L^q \times H^\beta_q \subset B^0_{q, 2} \times B^\beta_{q, 2}$, and Proposition \ref{p:op} then insures the remainder operator is bounded on $B^0_{q, 2} \times B^\sigma_{r, 2} \tend B^0_{1, \infty}$, even though $\sigma = 0$.

However, this will not be useful in the sequel as the strict inequality in \eqref{eq:pCondition} is necessary to provide compactness in appropriate spaces later on. In other words, the non-optimal result in Lemma \ref{l:LQproduct} is not responsible for our being unable to construct weak solutions in the critical space $L^{p_0}$ with $p_0 = 2 (1 + \alpha/d)^{-1}$.

\end{rmk}

Thanks to Lemma \ref{l:LQproduct}, we may bound the second part of the product $\rho_N u_N$. We have, under condition \eqref{eq:pCondition} from the statement of the Theorem, and with \eqref{eq:HFvelocitySpace},
\begin{equation*}
\big\| \rho_N ({\rm Id} - \Delta_{-1})u_N \big\|_{B^{\alpha - d/p}_{p, \infty}} \lesssim \| \rho_N \|_{B^0_{p, \infty}} \| ({\rm Id} - \Delta_{-1})u_N \|_{B^\alpha_{p, \infty}} \lesssim \| \rho_0 \|_{L^p}^2.
\end{equation*}
By noticing that $\alpha - d/p < 0$ since $p < d/\alpha$ by assumption, we may use the embeddings $L^p \subset B^0_{p, \infty} \subset B^{\alpha - d / p}_{p, \infty}$ to see that the whole product $\rho_N u_N$ remains uniformly bounded in the space $L^\infty(B^{\alpha - d/p}_{p, \infty})$. This yields the bound
\begin{equation*}
\| \partial_t \rho_N \|_{B^{\alpha - 1 - d/p}_{p, \infty}} \leq \| \rho_N u_N \|_{B^{\alpha - d/p}_{p, \infty}} \lesssim \| \rho_0 \|_{L^p}^2.
\end{equation*}
Finally, this estimate allows us to bound the time derivatives of both low and high frequency components $\partial_t \Delta_{-1} u_N$ and $\partial_t ({\rm Id} - \Delta_{-1})u_N$ of the velocity field: for all $k > 0$,
\begin{equation}
\| \partial_t \Delta_{-1} u_N \|_{L^\infty(B^k_{\infty, 1})} \lesssim \| \partial_t \rho_N \|_{L^\infty(B^{\alpha - 1 - d/p}_{p, \infty})} \lesssim \| \rho_0 \|_{L^p}^2
\end{equation}
and
\begin{equation*}
\big\| \partial_t ({\rm Id} - \Delta_{-1})u_N \big\|_{L^\infty(B^{-s}_{p, \infty})} \lesssim \big\| \partial_t \rho_N \big\|_{L^\infty(B^{\alpha - 1 - d/p}_{p, \infty})} \lesssim \| \rho_0 \|_{L^p}^2,
\end{equation*}
where we have set $- s = 2 \alpha - 1 - d/p < 0$. As a consequence, both low and high frequency components of the velocity field are bounded in spaces that provide time compactness, namely, for all $T > 0$ and $k > 0$,
\begin{equation}
\| \Delta_{-1} u_N \|_{W^{1, \infty}_T (B^k_{\infty, 1})} + \big\| ({\rm Id} - \Delta_{-1})u_N \big\|_{W^{1, \infty}_T (B^{-1-d}_{p, \infty})} \lesssim \| \rho_0 \|_{L^p}^2 + \| \rho_0 \|_{L^p}.
\end{equation}

Our last estimate will be an interpolation inequality in order to get simultaneous time and space compactness on $({\rm Id} - \Delta_{-1})u_N$ (this is already done for $\Delta_{-1}u_N$). Consider $\theta \in ]0, 1[$ whose precise value we will fix later on. We define the quantity $\sigma := -s \theta + (1 - \theta) \alpha$ and, for simplicity of notation, note $V_N := ({\rm Id} - \Delta_{-1})u_N$. Then, for all $t_1, t_2 \geq 0$, we have by Proposition \ref{p:interpolation},
\begin{equation*}
\begin{split}
\big\| V_N(t_2) - V_N(t_1) \big\|_{B^\sigma_{p, \infty}} & \lesssim \big\| V_N(t_2) - V_N(t_1) \big\|_{B^{-s}_{p, \infty}}^\theta \big\| V_N(t_2) - V_N(t_1) \big\|_{B^{\alpha}_{p, \infty}}^{1 - \theta} \\
& \lesssim |t_2 - t_1|^\theta \big\| V_N \big\|_{W^{1, \infty}_T(B^{-s}_{p, \infty})}^\theta \big\| V_N \big\|_{L^\infty_T(B^{\alpha}_{p, \infty})}^{1 - \theta},
\end{split}
\end{equation*}
so that the functions $V_N = ({\rm Id} - \Delta_{-1}) u_N$ are uniformly bounded in the space $C^{0, \theta}_T(B^\sigma_{p, \infty})$.

\medskip

\textbf{STEP 3: convergence to a solution.} Let us summarize the uniform bounds we have obtained so far. We have, for every $N \geq 1$, $k > 0$ and $T > 0$,
\begin{equation}
\| \rho_N \|_{L^\infty(L^p)} + \| \Delta_{-1} u_N \|_{W^{1, \infty}_T(B^k_{\infty, 1})} + \big\| ({\rm Id} - \Delta_{-1}) u_N \big\|_{C^{0, \theta}_T(B^\sigma_{p, \infty})} \lesssim \| \rho_0 \|_{L^p} + \| \rho_0 \|_{L^p}^2.
\end{equation}
Here, $\sigma \in ]-s, \alpha[$ can be chosen positive and as close to $\alpha$ as desired while still keeping the Hölder exponent positive $\theta > 0$. In addition, if, as in the statement of Theorem \ref{t:globalWeakSol}, the initial datum is $L^q$ for some $q \in [1, + \infty]$, then we also have
\begin{equation}\label{eq:LqIneq}
\| \rho_N \|_{L^\infty(L^q)} = \| \rho_{0, N} \|_{L^q} \leq \| \rho_0 \|_{L^q}.
\end{equation}

Recall that $p > 1$ because $\alpha < d$, so that all spaces involved (except possibly $L^q$) have separable preduals. We deduce the convergence of the sequences $(\rho_N)_N$, $(\Delta_{-1}u_N)_N$ and $\big( ({\rm Id} - \Delta_{-1})u_N \big)_N$, up to an extraction, in the following topologies (here $K \subset \R^d$ is an arbitrary compact set): as $N \rightarrow + \infty$, we have
\begin{equation}\label{eq:LPconvergences}
\begin{split}
& \rho_N \wtend^* \rho \qquad \text{in } L^\infty(L^p)\\
& \Delta_{-1} u_N \rightarrow u_{LF} \qquad \text{in } L^\infty_T(L^\infty_{\rm loc}) \\
& ({\rm Id} - \Delta_{-1})u_N \rightarrow u_{HF} \qquad \text{in } L^\infty_T(B^{\beta}_{p, \infty}(K)).
\end{split}
\end{equation}
In the lines above, $\beta < \sigma$ is an exponent whose value we will chose later, and the functions $u_{LF}$ and $u_{HF}$ are the limits of the Low Frequency and High Frequency parts. We accordingly set $u := u_{LF} + u_{HF}$ so that $u_N \tend u$ in, say, $\mc D'(]0, T[ \times \R^d)$. Convergence of the initial data $\rho_{0, N}$ follows from the fact that $p < + \infty$ (by assumption), and so (recall that $0 \leq \Psi \leq 1$ is a cut-off function),
\begin{equation*}
\begin{split}
\| \rho_{0, N} - \rho_0 \|_{L^p} & \leq \big\| \Psi(x/N) ({\rm Id} - S_N) \rho_0(x) \big\|_{L^p} + \big\| (1 - \Psi(x/N)) \rho_0(x) \big\|_{L^p}\\
& \leq \big\| ({\rm Id} - S_N) \rho_0(x) \big\|_{L^p} + o(1) \tend_{N \rightarrow + \infty} 0.
\end{split}
\end{equation*}

As weak convergence of the densities insures that $\partial_t \rho_N \tend \partial_t \rho$ in $\mc D'(\R_+^* \times \R^d)$, the only real problem is the convergence of the product $\rho_N u_N$. We will once again use decomposition \eqref{eq:productLHFdecomposition}. On the one hand, the convergence results \eqref{eq:LPconvergences} above make it clear that
\begin{equation}\label{eq:LFprodConv}
\rho_N \Delta_{-1} u_N \wtend^* \rho u_{LF} \qquad \text{in } L^\infty(L^p).
\end{equation}
The treatment of the other part of the product is more delicate. For this, we fix a test function $\psi \in \mc D(]0, + \infty[ \times \R^d)$ and study the brackets
\begin{equation}\label{eq:convProdBrackets}
\Big\langle \rho_N ({\rm Id} - \Delta_{-1})u_N - \rho u_{HF}, \psi \Big\rangle = \Big\langle \rho_N \big( ({\rm Id} - \Delta_{-1})u_N - u_{HF} \big), \psi \Big\rangle + \big\langle \rho - \rho_N, \psi u_{HF} \big\rangle.
\end{equation}
Concerning the first bracket, we fix a function $\varphi \in \mc D (]0, + \infty[ \times \R^d)$ such that $\varphi \psi = \psi$. In order to use the strong convergence of the $({\rm Id} - \Delta_{-1})u_N$ from \eqref{eq:LPconvergences}, we wish to apply the product Lemma \ref{l:LQproduct} to $B^0_{p, \infty} \times B^\beta_{p, \infty}$. For this, we need $\beta$ to be close enough to $\alpha$ that
\begin{equation}\label{eq:betaCondition}
\frac{2}{1 + \frac{\alpha}{d}} < \frac{2}{1 + \frac{\beta}{d}} < p,
\end{equation}
and this is always possible due to the strict inequality in \eqref{eq:pCondition}. Under this condition, and using the fact that $B^{\beta - d/p}_{p, \infty}$ has (pre)dual $B^{d/p - \beta}_{p', 1}$, we obtain
\begin{multline*}
\left| \Big\langle \rho_N \big( ({\rm Id} - \Delta_{-1})u_N - u_{HF} \big), \psi \Big\rangle \right| \leq \| \psi \|_{L^1_T(B^{d/p - \beta}_{p', 1})} \, \| \rho_N \|_{L^\infty(B^0_{p, \infty})}\\
\times \Big\| \big(({\rm Id} - \Delta_{-1})u_N - u_{HF} \big) \varphi \Big\|_{L^\infty_T(B^\beta_{p, \infty})} \; \tend_{N \rightarrow + \infty} 0.
\end{multline*}
For the second bracket in \eqref{eq:convProdBrackets}, since we know that $\rho_N$ converges to $\rho$ in the weak-$(*)$ topology of $L^\infty_T(L^p)$, we must make sure that $u_{HF} \psi$ lies in the predual $L^1(L^{p'})$.  Here, it is useful to consider two cases. Firstly, if $p \geq 2$, then we must have $p' \leq 2 \leq p$, and since $\beta > 0$, the function $u_{HF}$ must be locally $B^\beta_{p, \infty} \subset L^p \subset L^{p'}_{\rm loc}$. Secondly, if $p \leq 2 \leq p'$, we must check that $B^\beta_{p, \infty} \subset L^{p'}$. As $p' \geq p$, embeddings of Besov spaces give the chain of inclusions
\begin{equation*}
B^\beta_{p, \infty} \subset B^{\beta - d \left( \frac{1}{p} - \frac{1}{p'} \right)}_{p', \infty} \subset L^{p'},
\end{equation*}
as long as the condition
\begin{equation*}
\beta - d \left( \frac{1}{p} - \frac{1}{p'} \right) > 0
\end{equation*}
holds. But this is equivalent to \eqref{eq:betaCondition}, which we already enforced on $\beta$. Therefore, in both cases, we have shown that $u_{HF} \psi \in L^1(L^{p'})$, so that we indeed have
\begin{equation*}
\big\langle \rho - \rho_N, \psi u_{HF} \big\rangle \tend 0 \qquad \text{as } N \rightarrow + \infty.
\end{equation*}
Together with \eqref{eq:LFprodConv}, this shows convergence of the whole product $\rho_N u_N$ to
\begin{equation*}
\rho_N u_N \tend \rho u_{LF} + \rho u_{HF} = \rho u \qquad \text{in } \mc D'(]0, + \infty[ \times \R^d).
\end{equation*}

Finally, we have to check that $(\rho, u)$ is indeed a solution of the fractional Stokes problem. Because of the convergence properties \eqref{eq:LPconvergences}, we gather that the Fourier transforms of $u_N$ and $\rho_N$ converge locally, \textsl{e.g.} in the sense of distributions:
\begin{equation*}
(\what{u_N}, \what{\rho_N}) \tend (\what{u}, \what{\rho}) \qquad \text{in } \mc D'(]0, +\infty[ \times \R^d).
\end{equation*}
Let us fix $\phi(t, \xi) \in \mc D(]0, +\infty[ \times \R^d)$ that is supported away from $\xi = 0$. Then $T_N \phi = \phi$ for $N$ large enough. By noting $m(\xi)$ the symbol of the operator $\P g$, we see that
\begin{equation*}
\langle \what{u_N}, \phi \rangle = \left\langle \what{\rho_N}(t, \xi), \frac{m(\xi)}{|\xi|^\alpha} \phi(t, \xi) \right\rangle \tend_{N \rightarrow + \infty} \left\langle \what{\rho}(t, \xi), \frac{m(\xi)}{|\xi|^\alpha} \phi(t, \xi) \right\rangle,
\end{equation*}
where the brackets refer to the $\mc D'(]0, + \infty[ \times \R^d) \times \mc D(]0, + \infty[ \times \R^d)$ duality. Uniqueness of the limit implies that
\begin{equation*}
\what{u}(t, \xi) = \frac{m(\xi)}{|\xi|^\alpha} \what{\rho}(\xi) \qquad \text{in } \mc D' \big( ]0, +\infty[ \times (\R^d \setminus \{ 0 \} ) \big),
\end{equation*}
so that $u$ is indeed equal, up to a polynomial, to the unique $\mc S'_h$ solution of the Stokes equation with datum $\rho$, as given by Proposition \ref{p:StokesExUn}, and we only need to check that $u(t) \in \mc S'_h$ for almost all $t > 0$. This follows from Remark \ref{r:Fatou} and the fact that the sequence $(u_N)_N$ is uniformly bounded in a homogeneous Besov space:
\begin{equation*}
\| u_N \|_{\dot{B}^{\alpha}_{p, \infty}} \lesssim \| \rho_N \|_{L^p},
\end{equation*}
so that we must have $u \in L^\infty(\dot{B}^{\alpha}_{p, \infty})$, and hence $u(t) \in \mc S'_h$ for almost all $t > 0$.

\medskip

\textbf{STEP 4:} $L^q$ \textbf{estimates.} Finally, we only have to prove the uniform $L^q$ bounds \eqref{eq:LqBound} hold for the solution. If $q > 1$, it is a simple consequence of the $L^q$ boundedness \eqref{eq:LqIneq} of the approximate solutions, their weak-$(*)$ convergence and the Banach-Steinhaus theorem. When $q = 1$ however, the space $L^1$ does not have a separable predual, so the argument fails. Instead, we view $L^1$ as a subset of the space $\mc M$ of finite measures, whose predual is the space $C_0$ of continuous functions that tend to zero at infinity. Applying the Banach-Steinhaus theorem to $\mc M$ gives the bound
\begin{equation*}
\| \rho \|_{\mc M} \leq \| \rho_0 \|_{L^1}.
\end{equation*}
But we know that $\rho$ also is an $L^p$ function with $1 < p < + \infty$. Consequently, $\rho$ is locally integrable and the $L^1$ and $\mc M$ norms therefore agree for $\rho$. In other words
\begin{equation*}
\| \rho \|_{L^1} = \| \rho \|_{\mc M} \leq \| \rho_0 \|_{L^1}.
\end{equation*}
This remark ends the proof.
\end{proof}

\section{Local well-posedness}

In this Section, we are concerned with proving local well-posedness of \eqref{ieq:AIPM} by means of a fixed-point argument. This result is mainly interesting when $\alpha < 1$, as we will have means of proving global well-posedness when $\alpha \geq 1$. In particular, this means that we do not worry too much about the case $d=2$, where the value $\alpha = 2$ is of special interest, although Theorem \ref{t:localWP} could be adapted to fit that setting.

\begin{thm}\label{t:localWP}
Consider a dimension $d \geq 2$ as well as indices $s > 0$ and $p \in [1, + \infty]$ such that
\begin{equation*}
p < \frac{d}{\alpha}, \qquad \text{ and } \qquad s \geq 1 - \alpha.
\end{equation*}
For all initial datum $\rho_0 \in B^s_{\infty, 1} \cap L^p$, there exists a time $T > 0$ such that the system \eqref{ieq:AIPM} has a unique solution in that space $\rho \in C^0([0, T[; B^s_{\infty, 1} \cap L^p)$. Moreover, the lifespan $T^*$ of this solution satisfies the inequality
\begin{equation}\label{eq:LifespanLB1}
T^* \geq \frac{C}{\| \rho_0 \|_{B^{s}_{\infty, 1} \cap L^p}}.
\end{equation}
Finally, the lifespan $T^*$ is finite if and only if
\begin{equation*}
\int_0^{T^*} V(t) \dt = +\infty,
\end{equation*}
where the quantity $V(t)$ is defined by $V(t) = \| \nabla u \|_{L^\infty}$ if $s < 1$ and $V(t) = \| \nabla u \|_{L^\infty} + \| \nabla \rho \|_{L^\infty}$ otherwise.
\end{thm}

Let us make a few remarks concerning Theorem \ref{t:localWP}, and in particular about the lifespan of the local solution.

\begin{rmk}
A direct consequence of the continuation criterion in the Theorem above is that the problem \eqref{ieq:AIPM} is \textsl{globally} well-posed when $\alpha > 1$, under appropriate conditions on the initial datum. The following Corollary explains why this is the case. 

Note that, for technical reasons, the 2D Transport-Stokes system $\alpha = d = 2$ is excluded (in that case, we would need $p < 1$). However, this case of special interest will be covered by Theorem \ref{t:globalWPBesov} below, which is more general than Theorem \ref{t:localWP} on the question of global solutions when $\alpha \geq 1$. 

Similarly, the case $\alpha = 1$, which is especially relevant as system \eqref{ieq:AIPM} then behaves nearly identically to the 2D Euler system, is not in the scope of Corollary \ref{c:globalFromLocal}. Just as the 2D Euler equations, the setting $\alpha = 1$ necessitates different methods which we will employ to prove the more advanced results of Theorem \ref{t:globalWPBesov} below.
\end{rmk}

\begin{cor}\label{c:globalFromLocal}
With the notation of Theorem \ref{t:localWP}, we make the additional assumptions that $s < 1$ and $\alpha > 1$. Then the unique solution provided by Theorem \ref{t:localWP} is in fact global.
\end{cor}

\begin{proof}[Proof (of the Corollary)]
Let $\rho \in C^0([0, T^*[ ; L^p \cap B^s_{\infty, 1})$ be the unique solution given by Theorem \ref{t:localWP}, where $T^*$ is, as above, the lifespan of the solution. Then, the velocity field has Lipschitz regularity $u \in L^1([0, T^*[ ; W^{1, \infty})$, as Proposition
\ref{p:LPoperatorBound} gives
\begin{equation}\label{eq:corEst1}
\| u \|_{B^1_{\infty, 1}} \lesssim \| \rho \|_{L^p} + \| \rho \|_{B^{1 - \alpha}_{\infty, 1}}.
\end{equation}
As a consequence, the conservation of Lebesgue norms holds: for every $q \in [1, + \infty]$, we have $\| \rho(t) \|_{L^q} = \| \rho_0 \|_{L^q}$. In particular, if we choose $q > d/(\alpha - 1)$, then the solution will be $L^\infty(L^q)$ provided $\rho_0 \in L^q$. Let us show that this is indeed the case: because $\rho_0 \in B^s_{\infty, 1} \subset L^\infty$ and
\begin{equation*}
p < \frac{d}{\alpha} < \frac{d}{1 - \alpha} < q < + \infty,
\end{equation*}
interpolation between $L^p$ and $L^\infty$ shows that we must have $\| \rho_0 \|_{L^q} \leq \| \rho_0 \|_{L^p}^{\theta} \| \rho_0 \|_{L^\infty}^{1 - \theta}$ with $\theta = p(\alpha - 1)/d$. Together with \eqref{eq:corEst1}, the embeddings
\begin{equation*}
L^q \subset B^0_{q, \infty} \subset B^{-d/q}_{\infty, \infty} \subset B^{1 - \alpha}_{\infty, 1},
\end{equation*}
show that the Lipschitz norm of $u$ is always bounded $\| u \|_{B^1_{\infty, 1}} \lesssim \| \rho_0 \|_{L^q}$, so that the integral
\begin{equation*}
\int_0^{T} \| \nabla u \|_{L^\infty} \dt \lesssim T \| \rho_0 \|_{L^q}
\end{equation*}
remains finite for every $T > 0$. By the continuation criterion of Theorem \ref{t:localWP}, we deduce that the solution is global $T^* = + \infty$.
\end{proof}

\begin{rmk}
Due to the global existence of weak solutions in Theorem \ref{t:globalWeakSol} above, explosion of solutions should be thought of as a loss of regularity (and therefore possible loss of uniqueness) at the time of blow-up. In that spirit, note that the finiteness of the usual $\| \nabla u \|_{L^1_T(L^\infty)}$ norm is not enough to provide continuation of solutions with $s \geq 1$. If $\| \nabla \rho \|_{L^\infty}$ blows-up but not $\| \nabla u \|_{L^\infty}$, there may be a loss of regularity at the singularity, but while still retaining uniqueness of solutions. 
\end{rmk}

\begin{rmk}
Finally, as is usually the case with such equations, the continuation criterion implies that the lifespan of the $C^0_T(B^s_{\infty, 1})$ solution is essentially independent of $s$. Therefore, the lower bound \eqref{eq:LifespanLB1} can be reformulated as
\begin{equation*}
T^* \geq \frac{C}{\| \rho_0 \|_{B^{1 - \alpha}_{\infty, 1} \cap L^p}} \text{ if } s < 1, \qquad \text{or} \qquad T^* \geq \frac{C}{\| \rho_0 \|_{B^{1}_{\infty, 1} \cap L^p}} \text{ if } s \geq 1.
\end{equation*}
\end{rmk}

\begin{rmk}
By taking into account the structure of the equations implied by gravity stratification, we will be able to exhibit a continuation which involves the derivative of the density in a single direction $\partial_d \rho$. We refer to Theorem \ref{t:ContCrit} below.
\end{rmk}

\subsection{\textsl{A Priori} Estimates}

\begin{prop}\label{p:aPrioriWP}
Assume that $s \geq 1 - \alpha$ satisfies $s > -1$ and that $\alpha < 1$ or $p < d/(\alpha - 1)$. For any regular solution $\rho$ of \eqref{ieq:AIPM} associated to the initial datum $\rho_0$, there exists a time $T > 0$ such that
\begin{equation*}
\| \rho \|_{L^\infty_T (B^s_{\infty, 1} \cap L^p)} \lesssim \| \rho_0 \|_{B^s_{\infty, 1} \cap L^p}.
\end{equation*}
In addition, inequality \eqref{eq:LifespanLB1} holds for $T$.
\end{prop}

\begin{proof}
Finding \textsl{a priori} estimates for problem \eqref{ieq:AIPM} is based on the theory of transport equations in Besov spaces. The idea is to bound every dyadic block $\Delta_j \rho$ of the solution in $L^\infty$ to find a $B^s_{\infty, 1}$ estimate. Fix a $j \geq -1$ and apply the operator $\Delta_j$ to the transport equation in \eqref{ieq:AIPM}. We have:
\begin{equation}\label{eq:transportCommutator}
\big( \partial_t + u \cdot \nabla \big) \Delta_j \rho = \big[ u \cdot \nabla, \Delta_j \big]\rho.
\end{equation}
Now, Lemma \ref{l:CommBCD} lets us bound the commutator in the righthand side in terms of $\rho$ and $\nabla u$. Let us treat separately the case wheres $s < 1$ and $s \geq 1$. 

\medskip

\textbf{First case.} In the first case $s < 1$, applying Lemma \ref{l:CommBCD} to our setting gives the following inequality:
\begin{equation*}
2^{js} \left\| \big[ u \cdot \nabla, \Delta_j \big]\rho \right\|_{L^p} \lesssim c_j(t) \| \nabla u \|_{L^\infty} \| \rho \|_{B^s_{\infty, 1}} = c_j(t) V(t) \| \rho \|_{B^s_{\infty, 1}},
\end{equation*} 
where the sequences $\big( c_j(t) \big)_{j \geq -1}$ have unit norm in the space $\ell^1(j \geq -1)$ and where $V(t)$ is defined as in the statement of Theorem \ref{t:localWP}. In order to have an upper bound that depends only on the unknown $\rho$, we express the derivatives of the velocity field as the image of the density $\nabla u = \nabla (- \Delta)^{- \alpha / 2} \P (\rho g)$ by a Fourier multiplier of degree $1 - \alpha$ and use Proposition \ref{p:LPoperatorBound} to obtain
\begin{equation}\label{eq:Vineq1}
V(t) = \| \nabla u \|_{L^\infty} \lesssim \| \rho \|_{L^p} + \| \rho \|_{B^{1 - \alpha}_{\infty, 1}}  \lesssim \| \rho_0 \|_{L^p} + \| \rho \|_{B^s_{p, r}}.
\end{equation}
Here, we have used the fact that, for regular solutions, the $L^p$ norms are simply transported by the (incompressible) flow of $u$. Overall, we may write the following inequality:
\begin{equation}\label{eq:commIneq}
2^{js} \left\| \big[ u \cdot \nabla, \Delta_j \big]\rho \right\|_{L^p} \lesssim c_j(t) V(t) \| \rho \|_{B^s_{\infty, 1}} \lesssim c_j(t) \big( \| \rho_0 \|_{L^p} + \| \rho \|_{B^s_{\infty, 1}} \big)^2.
\end{equation}

\medskip

\textbf{Second case.} In the second case, where $s \geq 1$, an additional term appears in the commutator inequalities of Lemma \ref{l:CommBCD}, so that we instead get
\begin{equation*}
\begin{split}
2^{js} \left\| \big[ u \cdot \nabla, \Delta_j \big]\rho \right\|_{L^p} & \lesssim c_j(t)\Big(  \| \nabla u \|_{L^\infty} \| \rho \|_{B^s_{\infty, 1}} + \| \nabla \rho \|_{L^\infty} \| \nabla u \|_{B^{s - 1}_{\infty, 1}} \Big) \\
& = c_j(t) V(t) \Big( \| \rho \|_{B^s_{\infty, 1}} + \| \nabla u \|_{B^{s-1}_{\infty, 1}} \Big),
\end{split}
\end{equation*}
where, once again, the sequences $\big( c_j(t) \big)_{j \geq -1}$ have unit norm in the space $\ell^1(j \geq -1)$ and the quantity $V(t)$ is the one defined in Theorem \ref{t:localWP}. Now, because $s \geq 1$, the space $B^s_{\infty, 1} \subset B^1_{\infty, 1} \subset W^{1, \infty}$ is embedded in the space of Lipschitz functions and $\| \nabla \rho \|_{L^\infty} \lesssim \| \rho \|_{B^s_{\infty, 1}}$. We may therefore write an inequality similar to \eqref{eq:Vineq1}, namely
\begin{equation*}
V(t) = \| \nabla u \|_{L^\infty} + \| \nabla \rho \|_{L^\infty} \lesssim \| \rho \|_{L^p} + \| \rho \|_{B^s_{\infty, 1}}.
\end{equation*}
Likewise, we have
\begin{equation*}
\| \nabla u \|_{B^{s-1}_{\infty, 1}} \lesssim \| \nabla u \|_{B^{\alpha + s-1}_{\infty, 1}} \lesssim \| \rho \|_{L^p} + \| \rho \|_{B^s_{\infty, 1}},
\end{equation*}
so that the whole commutator is bounded by an inequality that is nearly identical to \eqref{eq:commIneq}, 
\begin{equation}\label{eq:commIneq2}
2^{js} \left\| \big[ u \cdot \nabla, \Delta_j \big]\rho \right\|_{L^p} \lesssim c_j(t) V(t) \Big( \|\rho_0 \|_{L^p} + \| \rho \|_{B^s_{\infty, 1}} \Big) \lesssim c_j(t) \big( \| \rho_0 \|_{L^p} + \| \rho \|_{B^s_{\infty, 1}} \big)^2.
\end{equation}
Note that both inequalities in \eqref{eq:commIneq2} also hold in the first case, where $s < 1$.

\medskip

Inequality \eqref{eq:commIneq2}, which, recall, is valid regardless of whether $s < 1$ or not, allows us to perform an $L^\infty$ estimate on equation \eqref{eq:transportCommutator}. We obtain, for all $j \geq -1$ and $0 \leq t \leq T$,
\begin{equation}\label{eq:LPblockEstimate}
\begin{split}
2^{js} \| \Delta_j \rho (t) \|_{L^\infty} & \leq 2^{js} \| \Delta_j \rho_0 \|_{L^\infty} + C \int_0^T c_j(t) V(t) \Big( \| \rho \|_{L^p} + \| \rho \|_{B^s_{\infty, 1}} \Big) \dt\\
& \leq 2^{js} \| \Delta_j \rho_0 \|_{L^\infty} + C_0 \int_0^T c_j(t) \Big( \| \rho \|_{L^p} + \| \rho \|_{B^s_{\infty, 1}} \Big)^2 \dt.
\end{split}
\end{equation}
In the above, $C_0$ is a fixed constant which we will use in further estimates. Taking the sum over all $j \geq -1$ provides a integral inequality involving the norms of the solution. In order to simpify notations further, we set
\begin{equation*}
E(t) = \| \rho \|_{L^p} + \| \rho \|_{B^s_{\infty, 1}} = \| \rho_0 \|_{L^p} + \| \rho \|_{B^s_{\infty, 1}},
\end{equation*}
so that 
\begin{equation}\label{eq:intIneq}
E(T) \leq E(0) + C_0 \int_0^T V(t) E(t) \dt \leq E(0) + C_0 \int_0^T E(t)^2 \dt.
\end{equation}
Incidentally, note that if $V(t)$ is integrable on $[0, T]$, then $E(t)$ cannot blow-up. In order to deduce upper bounds on the solution from this estimate, we work on a time interval $[0, T_0]$ where the integral of the previous inequality is not too large. More precisely, set
\begin{equation*}
T_0 := \sup \left\{ T > 0, \quad \int_0^T E(t)^2 \dt \leq E(0) \right\},
\end{equation*}
so that for all $T \leq T_0$, we must have $E(t) \leq (1 + C_0) E(0)$. In particular, so long as $T \leq T_0$, the integral in \eqref{eq:intIneq} cannot be larger than
\begin{equation*}
\int_0^T E(t)^2 \dt \leq (1 + C_0)^2 E(0)^2,
\end{equation*}
which implies that $T_0$ is nonzero and bounded from below by
\begin{equation*}
T_0 \geq \frac{(1 + C_0)^{-1}}{E(0)}.
\end{equation*}
This fact ends the proof.
\end{proof}

\subsection{Stability Estimates}

The purpose of this paragraph is to prove stability estimates that will serve to implement a fixed point argument, thus constructing unique solutions. We will have to study a system of the form 
\begin{equation*}
\begin{cases}
\big( \partial_t + u \cdot \nabla \big) \delta \rho = - (\delta u \cdot \nabla) \rho \\
\delta u = (- \Delta)^{\alpha / 2} \P (\delta \rho g),
\end{cases}
\end{equation*}
where $(\rho, u)$ is a solution of \eqref{ieq:AIPM} and $(\delta \rho, \delta u)$ are, loosely stated, perturbations of this solution. Our work in this Section is summarized in the following Proposition.

\begin{prop}\label{p:stabEst}
Assume that the pair $(s, p)$ fulfills the assumptions of Theorem \ref{t:localWP}, that is $s > 0$, $s \geq 1 - \alpha$ and $p < d/p$. Consider $\rho \in B^s_{\infty, 1} \cap L^p$ and $\delta \rho \in B^{s-1}_{\infty, 1} \cap B^{-1}_{p, \infty}$ associated to $\delta u$ through $\delta u = (- \Delta)^{- \alpha / 2} \P (\delta \rho g)$. Then the product $\D (\rho \delta u)$ belongs to $B^{s-1}_{\infty, 1} \cap B^{-1}_{p, \infty}$ and the estimate
\begin{equation*}
\big\| \D(\rho \delta u) \big\|_{B^{s-1}_{\infty, 1} \cap B^{-1}_{p, \infty}} \lesssim \| \rho \|_{B^s_{\infty, 1} \cap L^p} \| \delta \rho \|_{B^{s-1}_{\infty, 1} \cap L^p}
\end{equation*}
holds.
\end{prop}

\begin{proof}
We start by finding $B^{s-1}_{\infty, 1}$ estimates for the product. Our argument is based on the Bony decomposition: owing to the fact that $\delta u$ is divergence free, we may write
\begin{equation*}
\D (\rho \delta u) = \mc T_{\delta u_k}(\partial_k \rho) + \mc T_{\partial_k \rho}(\delta u_k) + \partial_k \mc R (\delta u_k, \rho),
\end{equation*}
where, here and below, there is an implicit summation on the repeated index $k = 1, ..., d$. Let us start by finding bounds for the velocity perturbation $\delta u$ which involve only $\delta \rho$. By using Proposition \ref{p:LPoperatorBound} we find that
\begin{equation}\label{eq:duBounded}
\| \delta u \|_{B^{s+\alpha - 1}_{\infty, 1}} \lesssim \| \delta \rho \|_{B^{-1}_{p, \infty}} + \| \delta \rho \|_{B^s_{\infty, 1}}.
\end{equation}
We now look at the paraproducts. The first one is straightforward to bound with Proposition \ref{p:op}, as $\delta u \in B^0_{\infty, 1} \subset L^\infty$ is a bounded function (recall that $s + \alpha - 1 \geq 0$). We may therefore use \eqref{eq:duBounded} to obtain
\begin{equation*}
\| \mc T_{\delta u_k} (\partial_k \rho) \|_{B^{s-1}_{\infty, 1}} \lesssim \| \nabla \rho \|_{B^{s-1}_{\infty, 1}} \| \delta u \|_{L^\infty} \lesssim \| \rho \|_{B^s_{\infty, 1}} \Big( \| \delta \rho \|_{B^{-1}_{p, \infty}} + \| \delta \rho \|_{B^s_{\infty, 1}} \Big).
\end{equation*}
For the second one, we must consider whether $s \geq 1$, in which case $\partial_k \rho \in L^\infty$ or not. If indeed we have $s \geq 1$, then the paraproduct assumes the regularity of $\delta u$ and
\begin{equation*}
\begin{split}
\| \mc T_{\partial_k \rho} (\delta u_k) \|_{B^{s - 1}_{\infty, 1}} & \leq \| \mc T_{\partial_k \rho} (\delta u_k) \|_{B^{s+\alpha - 1}_{\infty, 1}} \\
& \lesssim \| \rho \|_{B^s_{\infty, 1}} \| \delta u \|_{B^{s+\alpha - 1}_{\infty, 1}} \\
& \lesssim \| \rho \|_{B^s_{\infty, 1}} \Big( \| \delta \rho \|_{B^{-1}_{p, \infty}} + \| \delta \rho \|_{B^s_{\infty, 1}} \Big).
\end{split}
\end{equation*}
However, in the case where $s < 1$, the derivatives $\partial_k \rho$ induce a loss of regularity, according to Proposition \ref{p:op}, and the solution cannot be simply to factor out the divergence $\partial_k \mc T_\rho (\delta u_k)$, as we would not be making full use of the regularity of $\rho$. By use of Proposition \ref{p:op}, we have
\begin{equation*}
\| \mc T_{\partial_k \rho} (\delta u_k) \|_{B^{2s + \alpha - 2}_{\infty, 1}} \lesssim \| \nabla \rho \|_{B^{s-1}_{\infty, \infty}} \| \delta u \|_{B^{s+\alpha - 1}_{\infty, 1}},
\end{equation*}
but given that $s \geq 1 - \alpha$, we must have $2s-2 + \alpha \geq s-1$, and we still obtain the desired bound:
\begin{equation*}
\| \mc T_{\partial_k \rho} (\delta u_k) \|_{B^{s - 1}_{\infty, 1}} \lesssim \| \rho \|_{B^s_{\infty, 1}} \Big( \| \delta \rho \|_{B^{-1}_{p, \infty}} + \| \delta \rho \|_{B^s_{\infty, 1}} \Big).
\end{equation*}

We now only are left with the remainder $\partial_k \mc R(\delta u_k, \rho)$. In order for it to be properly defined (as an element of $B^0_{\infty, \infty}$), it requires that the sum of the regularities of $\partial_k \rho$ and $\delta u_k$ be nonnegative $2s - 1 + \alpha \geq 0$, or in other words $s \geq \frac{1}{2}(1 - \alpha)$. However, our assumptions bearing on $s$ show that in fact $s > \frac{1}{2}(1 - \alpha)$, as we either have $\alpha < 1$ and
\begin{equation*}
s \geq 1 - \alpha > \frac{1}{2}(1 - \alpha),
\end{equation*}
or we have $\alpha \geq 1$ and 
\begin{equation*}
s > 0 \geq \frac{1}{2}(1 - \alpha) \geq 1 - \alpha.
\end{equation*}
In both cases, $2s-1 + \alpha > 0$ and we may apply Proposition \ref{p:op} to get
\begin{equation*}
\begin{split}
\big\| \partial_k \mc R( \delta u_k, \rho) \big\|_{B^{s - 1}_{\infty, 1}} & \lesssim \| \mc R (\delta u_k, \rho) \|_{B^{2s + \alpha - 1}_{\infty, 1}} \\
& \lesssim \| \rho \|_{B^s_{\infty, 1}} \| \delta u \|_{B^{s + \alpha}_{\infty, 1}} \\
& \lesssim \| \rho \|_{B^s_{\infty, 1}} \Big( \| \delta \rho \|_{B^{-1}_{p, \infty}} + \| \delta \rho \|_{B^s_{\infty, 1}} \Big).
\end{split}
\end{equation*}
Overall, we have shown that the quantity $\D(\rho \delta u)$ is bounded in the space $B^{s-1}_{\infty, 1}$ by
\begin{equation*}
\| \rho \|_{B^s_{\infty, 1}} \Big( \| \delta \rho \|_{B^{-1}_{p, \infty}} + \| \delta \rho \|_{B^s_{\infty, 1}} \Big).
\end{equation*}

\medskip

It now only remains to find the $B^{-1}_{p, \infty}$ bound. Although it is possible to do this by means of the Bony decomposition, a much faster way exists: $\D(\rho \delta u)$ simply is the derivative of a $L^p$ function, so
\begin{equation*}
\begin{split}
\big\| \D (\rho \delta u) \big\|_{B^{-1}_{p, \infty}} & \lesssim \| \rho \delta u \|_{B^0_{p, \infty}} \\
& \lesssim \| \rho \delta u \|_{L^p} \\
& \lesssim \| \rho \|_{L^p} \| \delta u \|_{L^\infty} \\
& \lesssim  \| \rho \|_{L^p} \Big( \| \delta \rho \|_{B^{-1}_{p, \infty}} + \| \delta \rho \|_{B^s_{\infty, 1}} \Big),
\end{split}
\end{equation*}
and this finally ends the proof of the Proposition.

\end{proof}

\subsection{Proof of Theorem\ref{t:localWP}: Iterative Scheme}

Let us now finish proving Theorem \ref{t:localWP}. We will construct a solution of \eqref{ieq:AIPM} by means of an iterative scheme, and prove its uniqueness.

\begin{proof}[Proof (of Theorem \ref{t:localWP})]
\textbf{STEP 1: iterative scheme.} Let us define a sequence $(\rho_n)_{n\geq 1}$ of approximate solutions in the following way. For the initial data, set $\rho_{0, n} = S_n \rho_0$ and define $\rho_1(t,x) := \rho_{0, 1}(x)$ and, for $n \geq 1$, $\rho_{n+1}$ is the unique global in time solution of
\begin{equation}\label{eq:iterativeScheme}
\begin{cases}
\big( \partial_t + u_n \cdot \nabla \big) \rho_{n+1} = 0\\
u_n = (- \Delta)^{\alpha/2} \P (\rho_n g),
\end{cases}
\qquad \text{with initial datum } \rho_{n+1}(0) = \rho_{0, n+1}.
\end{equation}
For notational convenience, we call $X := B^s_{\infty, 1} \cap L^p$ the space in which the estimates for the solution are made, and $Y := B^{s-1}_{\infty, 1} \cap B^{-1}_{p, \infty}$ the one in which we make the stability estimates. Firstly, by performing \textsl{mutatis mutandi} the same computations as those of Proposition \ref{p:aPrioriWP}, we obtain from \eqref{eq:iterativeScheme} that, for all $n \geq 1$ and $T > 0$,
\begin{equation*}
\| \rho_{n+1} \|_{L^\infty_T(X)} \lesssim \| \rho_0 \|_X + \int_0^T \| \rho_n \|_X \| \rho_{n+1} \|_X \dt.
\end{equation*}
Similarly, by the same arguments as those of Proposition \ref{p:aPrioriWP}, by defining the times $T_n > 0$ by
\begin{equation*}
T_{n+1} := \sup \left\{ T > 0, \qquad \int_0^T \| \rho_n \|_X \| \rho_{n+1} \|_X \dt \leq \| \rho_0 \|_X \right\},
\end{equation*}
we see that $\| \rho_{n+1} \|_{L^\infty_{T_{n+1}}(X)} \lesssim \| \rho_0 \|_X$ and an inductive argument allows us to show that the times $T_n$ are uniformly bounded from below by
\begin{equation*}
T_n \geq \frac{C}{\| \rho_0 \|_X} := T_0.
\end{equation*}

\medskip

\textbf{STEP 2: convergence.} We must now show that the sequence $(\rho_n)$ is convergent in some space (in $Y$ in fact). To do so, we define the difference functions
\begin{equation*}
\delta \rho_n := \rho_{n+1} - \rho_n \qquad \text{ and } \qquad \delta u_n := u_{n+1} - u_n,
\end{equation*}
and proceed identically with the initial data $\delta \rho_{0, n}$. We agree that $\delta \rho_0(t,x) = 0$ and $\delta \rho_{0, 0}(x) = 0$. The difference functions solve the following set of equations:
\begin{equation*}
\begin{cases}
\big( \partial_t + u_n \big) \delta \rho_n = - \D (\rho_n \delta u_{n-1}) \\
\delta u_{n-1} := (- \Delta)^{-\alpha/2} \P (\delta \rho_{n-1} g).
\end{cases}
\end{equation*}
In order to show that $(\rho_n)_n$ is Cauchy, we will show that the series
\begin{equation*}
\rho_n - \rho_1 := \delta \rho_{n-1} + \cdots + \delta \rho_{1}
\end{equation*}
is normally convergent in the Banach space $Y$. Now, because of our assumptions on $s > 0$, the norms $Y = B^{s-1}_{\infty, 1} \cap B^{-1}_{p, \infty}$ are propagated by the transport equation: since $u$ is divergence free, Theorem \ref{th:transport} insures that we have, for all $T \leq T_0$,
\begin{equation}\label{eq:stabEst1}
\begin{split}
\| \delta \rho_n \|_{L^\infty_T(Y)} & \lesssim \Big( \| \delta \rho_{0, n} \|_Y + \| \D (\rho_n \delta u_{n-1}) \|_{L^1_T(Y)} \Big) \exp \left\{ C \int_0^T \| \nabla u_n \|_{L^\infty \cap B^{s-2}_{\infty, 1}}  \dt \right\} \\
& \lesssim \left( \| \delta \rho_{0, n} \|_Y + \| \rho_n \|_{L^\infty_T(X)} \int_0^T \| \delta \rho_{n-1} \|_Y \dt \right) \exp \left\{ C \int_0^T \| \nabla u_n \|_{L^\infty \cap B^{s-2}_{\infty, 1}} \dt \right\},
\end{split}
\end{equation}
where this last inequality comes from Proposition \ref{p:stabEst}. As we go on, we will no longer be interested in precise inequalities (we only need to show that $(\delta \rho_n)_n$ is summable in $Y$), so we will treat all quantities bounded by $\| \rho_n \|_{L^\infty_{T_0}(X)} \lesssim \| \rho_0 \|_X$ as irrelevant constants. The previous inequalities then simplify to
\begin{equation*}
\| \delta \rho_n \|_{L^\infty_T(Y)} \lesssim \| \delta \rho_{0, n} \|_Y + \int_0^T \| \delta \rho_{n-1} \|_Y \dt .
\end{equation*}

In order to use these inequalities, we must study the norms $\| \delta \rho_{0, n} \|_Y$ for the initial data more closely. We have
\begin{equation*}
\| \delta \rho_{0, n} \|_{B^{s-1}_{\infty, 1}} = \big\| (S_{n+1} - S_n)\rho_0 \big\|_{B^{s-1}_{\infty, 1}} = 2^{n(s-1)} \| \Delta_n \rho_0 \|_{L^\infty} \leq 2^{-n} \| \rho_0 \|_{B^s_{\infty, 1}}.
\end{equation*}
We may make a similar computation with $\| \delta \rho_{0, n} \|_{B^{-1}_{p, \infty}}$. Overall, we get that $\| \delta \rho_{0, n} \|_Y = O (2^{-n})$. By plugging this into our estimate and setting $r_n(T) := \| \delta \rho_n \|_{L^\infty_T(Y)}$, we have
\begin{equation*}
\forall T \leq T_0, \qquad r_n(T) \leq \frac{C_0}{2^n} + C_0 \int_0^T r_{n-1}(t) \dt,
\end{equation*}
where $C_0 > 0$ is a fixed constant whose precise value depends on $\| \rho_0 \|_X$. By repeated use of this last inequality and remembering that $r_1(T) = r_1$ does not depend on time, we find that, for all $n \geq 1$, 
\begin{equation*}
\begin{split}
r_n(T) & \leq \sum_{k=0}^{n-2} \frac{C^{k+1}T^k}{2^{n-k} k!} \, + \, \frac{C^{n-1}T^{n-1}}{(n-1)!} r_1 \\
& \leq \frac{C}{2^n} \sum_{k=0}^{n-2} \frac{(2CT)^k}{k!} \, + \, \frac{(CT)^{n-1}}{(n-1)!} r_1 \\
& \leq \frac{C}{2^n} e^{2CT} + \frac{(CT)^{n-1}}{(n-1)!} r_1.
\end{split}
\end{equation*}
Now, this proves that the series $\sum \delta \rho_n$ is normally convergent in $Y$, and therefore the sequence $(\rho_n)$ is certainly convergent in $L^\infty_{T_0}(Y)$. Let $\rho$ be the associated limit. We have, as $n \rightarrow + \infty$,
\begin{equation}\label{eq:convApproxStr}
\rho_n \tend \rho \qquad \text{in } L^\infty_{T_0}(Y).
\end{equation}
We define the associated velocity field by $u = (- \Delta)^{- \alpha/2} \P (\rho g) \in L^\infty_{T_0}(B^{s-1+\alpha}_{\infty, 1})$. Note that, although we have shown the convergence of the $\rho_n$ in $L^\infty_{T_0}(Y)$, the sequence is uniformly bounded in $L^\infty_{T_0}(X)$. By Proposition \ref{p:Fatou}, we see that the solution also lies in this space.


\medskip

\textbf{STEP 3: the limit is a solution.} We have to check that the function $\rho \in L^\infty_{T_0}(Y)$ we have constructed is indeed a solution of problem \eqref{ieq:AIPM}. First of all, it is clear from the convergence of the $\rho_n$ shown above that, as $n \rightarrow + \infty$,
\begin{equation*}
\begin{split}
&\rho_{0, n} \tend \rho_0 \qquad \text{in } \mc D'(\R^d),\\
&\partial_t \rho_n \tend \partial_t \rho \qquad \text{in } \mc D'(]0, T_0[ \times \R^d),
\end{split}
\end{equation*}
so that we just have to study the convergence of the product $\rho_n u_n$. We use the following product lemma, whose proof flows directly from Proposition \ref{p:op}.

\begin{lemma}\label{l:prodLim}
Consider $s>0$ such that $s \geq 1 - \alpha$. Then the function product $(f, h) \longmapsto fh$ is bounded in the following topologies
\begin{equation}\label{eq:LtopoProd}
\begin{split}
& B^s_{\infty, 1} \times B^{s+ \alpha - 1}_{\infty, 1} \tend B^{- \alpha}_{\infty, 1}\\
& B^{s-1}_{\infty, 1} \times B^{s+\alpha}_{\infty, 1} \tend B^{- \alpha}_{\infty, 1}.
\end{split}
\end{equation}
\end{lemma}

\begin{proof}[Proof (of the Lemma)]
The proof is based on the Bony decomposition of a product $fh$,
\begin{equation*}
fh = \mc T_f(h) + \mc T_h(f) + \mc R(f, h).
\end{equation*}
Firstly, we see that the remainder $\mc R$ is bounded in both topologies \eqref{eq:LtopoProd}, according to Proposition \ref{p:op}, since the sum of the regularity indices satisfies
\begin{equation*}
s + (s+\alpha - 1) = (s-1) + (s+\alpha) = 2s + \alpha - 1 \geq s > 0 \geq - \alpha.
\end{equation*}
Concerning the paraproduct $\mc T_f(h)$, we must distinguish between the cases where the regularity index of $f$ is, or not, nonnegative. Overall, Proposition \ref{p:op} provides the following continuity properties for the paraproduct:
\begin{equation*}
(f, h) \in B^s_{\infty, 1} \times B^{s+ \alpha - 1}_{\infty, 1} \longmapsto \mc T_f(h) \in 
\begin{cases}
B^{2s+\alpha - 1}_{\infty, 1} \subset B^{- \alpha}_{\infty, 1} \quad \text{if } s < 0,\\
B^{s-1+\alpha}_{\infty, 1} \subset B^{- \alpha}_{\infty, 1} \quad \text{if } s \geq 0,
\end{cases}
\end{equation*}
and
\begin{equation*}
(f, h) \in B^{s-1}_{\infty, 1} \times B^{s+ \alpha}_{\infty, 1} \longmapsto \mc T_f(h) \in 
\begin{cases}
B^{2s+\alpha - 1}_{\infty, 1} \subset B^{- \alpha}_{\infty, 1} \quad \text{if } s < 1,\\
B^{s + \alpha}_{\infty, 1} \subset B^{- \alpha}_{\infty, 1} \quad \text{if } s \geq 1.
\end{cases}
\end{equation*}
For the paraproduct $\mc T_h(f)$, things are much quicker, since $h$ is either an element of $B^{s+\alpha - 1}_{\infty, 1}$ or $B^{s+\alpha}_{\infty, 1}$, and is in any case bounded $h \in L^\infty$. Therefore, the paraproduct $\mc T_h(f)$ has its value in $B^s_{\infty, 1} \subset B^{- \alpha}_{\infty, 1}$ or $B^{s-1}_{\infty, 1} \subset B^{- \alpha}_{\infty, 1}$.

In all cases, the remainder and both paraproducts have range in $B^{- \alpha}_{\infty, 1}$, and this ends proving the Lemma.
\end{proof}

With product Lemma \ref{l:prodLim}, we may finish proving that $\rho$ is a solution of \eqref{ieq:AIPM}. As we have explained, all we have to do is study the convergence of the product $\rho_n u_n$. We have,
\begin{equation*}
\begin{split}
\| \rho_n u_n - \rho u \|_{B^{- \alpha}_{\infty, 1}} & \leq \| \rho_n (u_n - u) \|_{B^{- \alpha}_{\infty, 1}} + \| (\rho_n - \rho) u \|_{B^{- \alpha}_{\infty, 1}} \\
& \lesssim \| \rho_n \|_{B^s_{\infty, 1}} \| u_n - u \|_{B^{s+\alpha - 1}} + \| \rho_n - \rho \|_{B^{s-1}_{\infty, 1}} \| u \|_{B^{s+1}_{\infty, 1}} \\
& \lesssim \| \rho_n \|_X \| \rho_n - \rho \|_Y + \| \rho_n - \rho \|_Y \| \rho \|_X.
\end{split}
\end{equation*}
From the remarks at the end of STEP 2, we know that the solution $\rho$ belongs to $L^\infty_{T_0}(X)$ and from the \textsl{a priori} estimates of STEP 1, we know that the approximate solutions are uniformly bounded in the same space. The convergence \eqref{eq:convApproxStr} of the $\rho_n$ and the previous inequalities show the convergence of the product $\rho_n u_n \tend \rho u$ in the space $L^\infty_{T_0}(B^{- \alpha}_{\infty, 1})$ and hence, as $n \rightarrow + \infty$,
\begin{equation*}
\D (\rho_n u_n) \tend \D (\rho u) \qquad \text{in } \mc D'([0, T_0[ \times \R^d).
\end{equation*}
This proves that $\rho$ is indeed a solution of the transport equation in \eqref{ieq:AIPM}. Reasoning exactly as in the proof of Theorem \ref{t:globalWeakSol}, see especially the end of STEP 3, we see that $u$ is a solution of the fractional Stokes problem.

\medskip

\textbf{STEP 4: uniqueness of the solution.} The uniqueness of the solution we have constructed pretty much flows from the stability estimates we used to prove convergence of the iterative scheme. Consider a given time $T > 0$ and two solutions $\rho_1, \rho_2 \in L^\infty_T(X)$ related to the initial data $\rho_{0, 1}, \rho_{0, 2} \in X$. Then, by defining the difference function $\delta \rho := \rho_2 - \rho_1$, the difference of initial data $\delta \rho_0 := \rho_{0, 2} - \rho_{0, 1}$, and the associated difference of velocity fields $\delta u$, we find that $\delta \rho$ is a (weak) solution of the problem
\begin{equation*}
\begin{cases}
\big( \partial_t + u_2 \cdot \nabla \big) \delta \rho = - \D (\rho_1 \delta u)\\
\delta u = (- \Delta)^{- \alpha/2} \P (\delta \rho g).
\end{cases}
\end{equation*}
Using Theorem \ref{th:transport} and the estimates therein provides an inequality which is nearly identical to \eqref{eq:stabEst1}, namely
\begin{equation*}
\begin{split}
\| \delta \rho \|_{L^\infty_T(Y)} & \lesssim \Big( \| \delta \rho_0 \|_Y + \| \D (\rho_1 \delta u) \|_{L^1_T(Y)} \Big) \exp \left\{ C \int_0^T \| \nabla u_2 \|_{L^\infty \cap B^{s-2}_{\infty, 1}} \dt \right\} \\
& \lesssim \left( \| \delta \rho_0 \|_Y + \| \rho_1 \|_{L^\infty_T(X)} \int_0^T \| \delta \rho \|_Y \dt \right) \exp \left\{ C \int_0^T \| \nabla u_2 \|_{L^\infty \cap B^{s-2}_{\infty, 1}} \dt \right\}.
\end{split}
\end{equation*}
If $\delta \rho_0 = 0$, then the identity of both solutions $\delta \rho = 0$ follows from Grönwall's lemma.

\medskip

\textbf{STEP 5: proof of the continuation criterion.} By standard arguments, it is enough to show that the solution $\rho$ remains bounded in $X$ as long as $V(t)$ remains integrable. Recall from inequality \eqref{eq:intIneq} that the norm  $E(t) = \| \rho(t) \|_X$ of the solution fulfills the integral inequality
\begin{equation*}
\forall T > 0, \qquad E(T) \leq E(0) + C_0 \int_0^T V(t) E(t) \dt.
\end{equation*}
A direct application of Grönwall's lemma gives the desired boundedness.

\medskip

\textbf{STEP 6: time continuity of the solution.} We set $T := T_0$ for the sake of simplicity. So far, we have constructed a solution in the space $L^\infty_{T}(X)$, but the claim that it is continuous in time with respect to the $X = B^s_{\infty, 1} \cap L^p$ topology remains unproved. The arguments are fairly standard, but will be refered to later on, so we give a concise proof. We start by studying continuity in $L^p$, which rests on the following Lemma (the Friedrichs commutator lemma).

\begin{lemma}[see also Lemma II.1 in \cite{DpL}]\label{l:DPLcommutator}
Consider $T > 0$, a divergence free vector field $u \in L^1_T(W^{1, \infty})$ and a function $\rho \in L^\infty_T(L^p)$ for some $p \in [1, + \infty[$. Then, if $(K_\epsilon)_{\epsilon > 0}$ is a mollification sequence with $K_\epsilon (y) = \epsilon^{-d} K(\epsilon^{-1} y)$, we have
\begin{equation*}
\big[ u \cdot \nabla, K_\epsilon * \big] \rho \tend_{\epsilon \rightarrow 0^+} 0 \qquad \text{in } L^1_T(L^p).
\end{equation*}
In the above, $K_\epsilon *$ denotes the convolution operator $f \mapsto K_\epsilon * f$.
\end{lemma}

\begin{proof}
We start by showing that the commutator remains bounded in $L^p$ as $\epsilon \rightarrow 0^+$. By integration by parts, we see that
\begin{equation*}
\big[ u \cdot \nabla, K_\epsilon * \big] \rho (x) = \int \rho(y) \Big( u(x) - u(y) \Big) \cdot \nabla K_\epsilon(x-y) {\rm d}y.
\end{equation*}
By using the mean value theorem, we know the difference of velocities is bounded by $|x - y| \| \nabla u \|_{L^\infty}$. Therefore, the Hausdorff-Young convolution inequality gives the upper bound
\begin{equation*}
\begin{split}
\left\| \big[ u \cdot \nabla, K_\epsilon * \big] \rho \right\|_{L^p} & \lesssim \| \nabla u \|_{L^\infty} \| \rho \|_{L^p} \int |y| |\nabla K_\epsilon (y)| {\rm d} y \\
& \lesssim \| \nabla u \|_{L^\infty} \| \rho \|_{L^p} \int |y| |\nabla K(y)| {\rm d}y.
\end{split}
\end{equation*}
This shows that the commutator remains bounded in $L^1_T(L^p)$ as $\rightarrow 0^+$. On the other hand, we may note that it must converge to zero in $L^1_T(L^p)$ if the density is regular enough, say, $\rho \in L^\infty_T(W^{1, p})$. Therefore, we fix a sequence $(\rho_n)_n$ of $L^\infty_T(W^{1, p})$ functions such that $\rho_n \tend \rho$ in the space $L^\infty_T(L^p)$ (this is possible because we have assumed $p < +\infty$). Therefore, our commutator is bounded by
\begin{equation*}
\left\| \big[ u \cdot \nabla, K_\epsilon * \big] \rho \right\|_{L^p} \lesssim \| \nabla u \|_{L^\infty} \| \rho_n - \rho \|_{L^p} + \left\| \big[ u \cdot \nabla, K_\epsilon * \big] \rho_n \right\|_{L^p},
\end{equation*}
and so
\begin{equation*}
\limss_{\epsilon \rightarrow 0^+} \left\| \big[ u \cdot \nabla, K_\epsilon * \big] \rho \right\|_{L^p} \lesssim \| \nabla u \|_{L^\infty} \| \rho_n - \rho \|_{L^p} \tend_{n \rightarrow + \infty} 0.
\end{equation*}
This concludes the proof of the Lemma.
\end{proof}

Let us return to the equation. By taking a mollification sequence $(K_\epsilon)_{\epsilon > 0}$ as in Lemma \ref{l:DPLcommutator} above, we see that the function $\rho_\epsilon := K_\epsilon * \rho$ is a (weak) solution of the transport equation
\begin{equation*}
\begin{cases}
\partial_t \rho_\epsilon + \D (u \rho_\epsilon) = \big[ u \cdot \nabla, K_\epsilon * \big] \rho \\
\rho_\epsilon(0) = K_\epsilon * \rho_0.
\end{cases}
\end{equation*}
Because the function $\rho_\epsilon$ lies in $L^\infty_T(W^{1, p})$ for each fixed value of $\epsilon$, we deduce from the previous equation that $\rho_\epsilon \in C^0_T(L^p)$. In addition, thanks to Lemma \ref{l:DPLcommutator}, the sequence $(\rho_\epsilon)$ is Cauchy in that space, and so $\rho \in C^0_T(L^p)$.

\medskip

We now have to show that the solution is continuous in the $B^s_{\infty, 1}$ topology. Here, we follow the ideas in \cite{BCD}, see in particular the proof of Theorem 3.19 and pp. 138--139 therein. 

From Lemma \ref{l:prodLim}, we deduce that the product $\rho u$ lies in the space $L^\infty_T(B^{-1-\alpha}_{\infty})$. Therefore $\partial_t \rho \in L^\infty_T (B^{-1 - \alpha}_{\infty, 1})$, and in particular all the Littlewood-Paley blocks of $\rho$ are continuous with respect to time: $\Delta_j \rho \in C^0_T(L^\infty)$, and we have $S_n \rho \in C^0_T(B^s_{\infty, 1})$. To conclude, we must show that the sequence $(S_n \rho)_n$ is in fact Cauchy in the space $C^0_T(B^s_{\infty, 1})$. We have
\begin{equation*}
\|\rho - S_n \rho \|_{B^s_{\infty, 1}} \lesssim \sum_{j \geq n-1} 2^{js} \| \Delta_j \rho \|_{L^\infty}.
\end{equation*}
By using inequality \eqref{eq:LPblockEstimate} from the proof of Proposition \ref{p:aPrioriWP}, we may bound the sum above: there is a family $\big( c_j(t) \big)_{j \geq -1}$ of sequences lying in the unit ball of $\ell^1(j \geq -1)$ such that
\begin{equation*}
\|\rho - S_n \rho \|_{L^\infty_T(B^s_{\infty, 1})} \lesssim \sum_{j \geq n-1} 2^{js} \| \Delta_j \rho_0 \|_{L^\infty} + \int_0^T \| \rho \|_X^2 \sum_{j \geq n-1} c_j(t) \dt.
\end{equation*}
We infer from the dominated convergence theorem that the righthand side of the above inequality converges to zero as $n \rightarrow + \infty$, and so the sequence $(S_n \rho)_n$ is indeed Cauchy in the space $C^0_T(B^s_{\infty, 1})$. This shows that $\rho \in C^0_T(X)$, and concludes the proof of Theorem \ref{t:localWP}.

\end{proof}

\section{A Directional Continuation Criterion}

\begin{thm}\label{t:ContCrit}
Assume that $\alpha < 1$. Under the assumptions and notations of Theorem \ref{t:localWP}, the unique solution $\rho \in C^0([0, T[; B^{1 - \alpha}_{\infty, 1} \cap L^p)$ may be continued beyond time $T$ provided that
\begin{equation}\label{eq:thContcritPartialD}
\int_0^T \big\| \partial_d \rho \big\|_{B^{- \alpha}_{\infty, 1}} \dt < + \infty
\end{equation}
\end{thm}

\begin{proof}
As we have explained in the introduction above, the proof of this continuation criterion is based on the precise form of the velocity field due to gravity stratification. By writing the Leray projection operator explicitly, the fractional Stokes equation reads
\begin{equation*}
u = (- \Delta)^{- \alpha/2} \Big( \rho g + \nabla (- \Delta)^{-1} \partial_d \rho \Big),
\end{equation*}
so that the product $u \cdot \nabla$ involves two different terms which factor linear expressions of $\partial_d \rho$. More precisely, we have
\begin{equation}\label{eq:specialStructure}
u \cdot \nabla \rho = (- \Delta)^{- \alpha / 2} \rho . \, \partial_d \rho + \nabla \rho \cdot \nabla (- \Delta)^{- 1 - \alpha / 2} \partial_d \rho.
\end{equation}
In order to prove the continuation criterion, we will need to revisit the usual theory of transport equations in Besov spaces and highlight, at each step, how the special structure \eqref{eq:specialStructure} will affect the estimates. Our goal, as in STEP 5 of the proof of Theorem \ref{t:localWP}, will be to show that the local solution $\rho$ of Theorem \ref{t:localWP} remains bounded in $X = B^{1 - \alpha}_{\infty, 1} \cap L^p$ on every time interval $[0, T]$ on which condition \eqref{eq:thContcritPartialD} holds. The continuation of the solution will then follow by standard arguments.

\medskip

First of all, because the velocity field $u$ associated to the solution is Lipschitz $u \in L^1_T(B^{1}_{\infty, 1}) \subset L^1_T (W^{1, \infty})$, the $L^p$ norm of the solution is transported by the flow:
\begin{equation*}
\forall t, \qquad \| \rho(t) \|_{L^p} = \| \rho_0 \|_{L^p}.
\end{equation*}
We therefore may focus on showing that the solution remains bounded in $B^{1 - \alpha}_{\infty, 1}$ under condition \eqref{eq:thContcritPartialD}. We therefore study the Littlewood-Paley blocks $\Delta_j \rho$ for $j \geq -1$, which solve the equation
\begin{equation}\label{eq:contCritTransport}
\big( \partial_t + u \cdot \nabla \big) \Delta_j \rho = \big[ u \cdot \nabla, \Delta_j \big] \rho.
\end{equation}
Usually, when dealing with transport equations, we use the standard estimates of Lemma \ref{l:CommBCD} to handle the commutator. However, in order to fully take advantage of the structure of the product $u \cdot \nabla \rho$, we must go through its proof in detail and isolate the contribution of the $d$-th derivative $\partial_d \rho$. Let us write the Bony decomposition for the products involved in the commutator. We have:
\begin{equation*}
\big[ u \cdot \nabla, \Delta_j \big] \rho = \big[ \mc T_{u_k}, \Delta_j \big] \partial_k \rho + \mc T_{\Delta_j \partial_k \rho} (u_k) + \mc R (u_k, \Delta_j \partial_k \rho) - \Delta_j \mc T_{\partial_k \rho}(u_k) - \Delta_j \mc R (\partial_k \rho, u_k).
\end{equation*}
We will deal with each of those terms in the order in which they appeared in the previous line.

\medskip

\textbf{First term:} $\big[ \mc T_{u_k}, \Delta_j \big] \partial_k \rho$. We start by writing the explicit meaning of the paraproduct in the commutator:
\begin{equation*}
\big[ \mc T_{u_k}, \Delta_j \big] \partial_k \rho = \sum_{|j-m| \leq 4} \Big\{ \big[ S_{m-1} (- \Delta)^{- \alpha / 2} \rho, \Delta_j \big] \Delta_m \partial_d \rho + \big[ S_{m-1} \nabla (- \Delta)^{- 1 - \alpha / 2} \partial_d \rho, \Delta_j \big] \Delta_m \nabla \rho \Big\}.
\end{equation*}
The sum in the equation above is supported on the indices $m$ which are not too far from $j$, namely $|j-m| \leq 4$, as can be seen from writing explicitly the commutators in the sum. For example, 
\begin{equation*}
\big[ S_{m-1} (- \Delta)^{- \alpha / 2} \rho, \Delta_j \big] \Delta_m \partial_d \rho = S_{m-1} (- \Delta)^{- \alpha/2} \rho . \Delta_j \Delta_m \partial_d \rho - \Delta_j \big( S_{m-1} (- \Delta)^{- \alpha/2} \rho . \Delta_m \partial_d \rho \big).
\end{equation*}
In the above, the operator $\Delta_j \Delta_m$ is zero if $|j - m| > 4$, and the product $S_{m-1} (- \Delta)^{- \alpha/2} \rho . \Delta_m \partial_d \rho$ is spectrally supported in an annulus of size roughly $2^m$, so that it is canceled by $\Delta_j$ when $|j - m| > 4$. Now, to estimate both commutators above, we apply Lemma \ref{l:basicComm}. On the one hand, we get
\begin{equation*}
\left\| \big[ S_{m-1} (- \Delta)^{- \alpha / 2} \rho, \Delta_j \big] \Delta_m \partial_d \rho \right\|_{L^\infty} \lesssim 2^{-j} \| \nabla S_{m-1} (- \Delta)^{- \alpha/2} \rho \|_{L^\infty} \| \Delta_j \partial_d \rho \|_{L^\infty}.
\end{equation*}
By assumption, we have $\alpha < 1$, so that the operator $\nabla ( - \Delta)^{- \alpha / 2}$ is a Fourier multiplier whose symbol is a homogeneous function of degree $1 - \alpha > 0$. An application of Proposition \ref{p:LPoperatorBound} therefore shows that the operator 
\begin{equation*}
    \nabla (- \Delta)^{- \alpha/2} \P g : B^{1 - \alpha}_{\infty, 1} \tend B^0_{\infty, 1} \subset L^\infty
\end{equation*}
is bounded, so that the previous inequality becomes, for $|j-m| \leq 4$,
\begin{equation}\label{eq:FstTerm1}
\begin{split}
\left\| \big[ S_{m-1} (- \Delta)^{- \alpha / 2} \rho, \Delta_j \big] \Delta_m \partial_d \rho \right\|_{L^\infty} & \lesssim 2^{-j} \|  \rho \|_{B^{1 - \alpha}_{\infty, 1}} \| \Delta_j \partial_d \rho \|_{L^\infty} \\
& \lesssim c_j(t) 2^{-j(1 - \alpha)} \|  \rho \|_{B^{1 - \alpha}_{\infty, 1}} \| \partial_d \rho \|_{B^{- \alpha}_{\infty, 1}},
\end{split}
\end{equation}
where, as usual, the sequences $\big( c_j(t) \big)_{j \geq -1}$ have unit norm in $\ell^1(j \geq -1)$. Similarly, we have, for $|j-m| \leq 4$,
\begin{equation}\label{eq:FstTerm2}
\begin{split}
\left\| \big[ S_{m-1} \nabla (- \Delta)^{- 1 - \alpha / 2} \partial_d \rho, \Delta_j \big] \Delta_m \nabla \rho \right\|_{L^\infty} & \lesssim 2^{-j} \big\| S_{m-1} \nabla^2 (- \Delta)^{-1 - \alpha /2} \partial_d \rho \big\|_{L^\infty} \| \Delta_m \nabla \rho \|_{L^\infty} \\
& \lesssim c_j(t) 2^{-j(1 - \alpha)} \| \rho \|_{B^{1 - \alpha}_{\infty, 1}} \big\| \nabla^2 (- \Delta)^{-1 - \alpha /2} \partial_d \rho \big\|_{L^\infty}.
\end{split}
\end{equation}
In the previous upper bound, the operator $\nabla^2 (- \Delta)^{-1 - \alpha/2} \partial_d$ is of order $1 - \alpha > 0$. However, we wish to have an estimate involving $\partial_d \rho$, so we must set aside the last derivative $\partial_d$ and work with the order $- \alpha$ operator $\nabla^2 (- \Delta)^{- 1 - \alpha / 2}$, which requires integrability to be defined, by Proposition \ref{p:StokesExUn}. We have
\begin{equation*}
\begin{split}
\big\| \nabla^2 (- \Delta)^{-1 - \alpha /2} \partial_d \rho \big\|_{L^\infty} & \lesssim \big\| \Delta_{-1} \nabla^2 (- \Delta)^{-1 - \alpha /2} \partial_d \rho \big\|_{L^\infty} + \big\| ({\rm Id} - \Delta_{-1}) \nabla^2 (- \Delta)^{-1 - \alpha /2} \partial_d \rho \big\|_{L^\infty}\\
& \lesssim \| \Delta_{-1} \partial_d \rho \|_{L^p} + \| \partial_d \rho \|_{B^{- \alpha}_{\infty, 1}} \\
& \lesssim \| \rho \|_{L^p} + \| \partial_d \rho \|_{B^{- \alpha}_{\infty, 1}}.
\end{split}
\end{equation*}
By putting together \eqref{eq:FstTerm1} and \eqref{eq:FstTerm2}, we obtain a bound for the whole commutator,
\begin{equation}\label{eq:1stTerm}
\left\| \big[ \mc T_{u_k}, \Delta_j \big] \partial_k \rho \right\|_{L^\infty} \lesssim c_j(t) 2^{-j(1 - \alpha)} \| \rho \|_{B^{1 - \alpha}_{\infty, 1}} \Big( \| \rho_0 \|_{L^p} + \| \partial_d \rho \|_{B^{- \alpha}_{\infty, 1}} \Big).
\end{equation}

\medskip

\textbf{Second Term:} $\mc T_{\Delta_j \partial_k \rho} (u_k)$. Let us write the paraproduct explicitly in order to clarify the meaning of this term:
\begin{equation*}
\mc T_{\Delta_j \partial_k \rho} (u_k) = \sum_{m \geq -1} \Big\{ S_{m-1} \Delta_j \partial_d \rho . \Delta_m (- \Delta)^{- \alpha/2} \rho + S_{m-1} \Delta_j \nabla \rho \cdot \Delta_m \nabla (- \Delta)^{-1 - \alpha/2} \partial_d \rho \Big\}.
\end{equation*}
First of all, note that the operator $S_{m-1} \Delta_j$ is nonzero only when $m \geq j$. In addition, because $S_{-1} = 0$, none of the terms with the $\Delta_m$ blocks involve low frequencies, and so we will be able to use Lemma \ref{l:Bernstein} on them. By taking advantage of these two remarks, we see that
\begin{equation*}
\begin{split}
\| \mc T_{\Delta_j \partial_k \rho} (u_k) \|_{L^\infty} & \leq \sum_{m \geq j} \Big\{ \big\| S_{m-1} \Delta_j \partial_d \rho \big\|_{L^\infty} \big\| \Delta_m (- \Delta)^{- \alpha/2} \rho \big\|_{L^\infty} \\
& \qquad \qquad \qquad \qquad \qquad + \big\| S_{m-1} \Delta_j \nabla \rho \big\|_{L^\infty} \big\| \Delta_m \nabla (- \Delta)^{-1 - \alpha/2} \partial_d \rho \big\|_{L^\infty} \Big\} \\
& \lesssim \sum_{m \geq j} \Big\{ 2^{-m \alpha} \big\| S_{m-1} \Delta_j \partial_d \rho \big\|_{L^\infty} \big\| \Delta_m \rho \big\|_{L^\infty} \\
& \qquad \qquad \qquad \qquad \qquad + 2^{-m(1+\alpha)} \big\| S_{m-1} \Delta_j \nabla \rho \big\|_{L^\infty} \big\| \Delta_m \partial_d \rho \big\|_{L^\infty} \Big\}.
\end{split}
\end{equation*}
Let us bound first the terms featuring the $\Delta_j$ blocks. Since the operator $S_{m-1} : L^\infty \tend L^\infty$ is bounded, we may ignore it and factor the resulting upper bounds out of the sum:
\begin{equation*}
\begin{split}
\| \mc T_{\Delta_j \partial_k \rho} (u_k) \|_{L^\infty} & \lesssim 2^{j \alpha} c_j(t) \| \partial_d \rho \|_{B^{- \alpha}_{\infty, 1}} \sum_{m \geq j} 2^{-m \alpha} \| \Delta_m \rho \|_{L^\infty} \\
& \qquad \qquad \qquad + 2^{j \alpha} c_j(t) \| \rho \|_{B^{1 - \alpha}_{\infty, 1}} \sum_{m \geq j} 2^{-m(1+\alpha)} \| \Delta_m \partial_d \rho \|_{L^\infty}.
\end{split}
\end{equation*}
For the terms with the $\Delta_m$ blocks, straightforward Bernstein inequalities provide 
\begin{equation}\label{eq:2ndTerm}
\begin{split}
\| \mc T_{\Delta_j \partial_k \rho} (u_k) \|_{L^\infty} & \lesssim 2^{j \alpha} c_j(t) \| \partial_d \rho \|_{B^{- \alpha}_{\infty, 1}} \sum_{m \geq j} 2^{-m} c_m(t) \| \rho \|_{B^{1 - \alpha}_{\infty, 1}} \\
& \qquad \qquad \qquad + 2^{j \alpha} c_j(t) \| \rho \|_{B^{1 - \alpha}_{\infty, 1}} \sum_{m \geq j} 2^{-m} c_m(t) \| \partial_d \rho \|_{B^{- \alpha}_{\infty, 1}}. \\
& \lesssim 2^{-j(1 - \alpha)} c_j(t) \| \partial_d \rho \|_{B^{- \alpha}_{\infty, 1}} \| \rho \|_{B^{1 - \alpha}_{\infty, 1}}.
\end{split}
\end{equation}

\medskip

\textbf{Third term:} $\mc R (u_k, \Delta_j \partial_k \rho)$. As before, let us give an explicit form for this term:
\begin{equation*}
\mc R (u_k, \Delta_j \partial_k \rho) = \sum_{|m-n| \leq 1} \Delta_m u_k . \Delta_n \Delta_j \partial_k \rho.
\end{equation*}
In the above, because $m$ and $n$ are close and $\Delta_n \Delta_j$ is nonzero only if $|n - j| \leq 1$, we see that the sum must be supported on those indices such that $j-2 \leq m, n \leq j+2$. Let $I_j$ be that range of indices. We have:
\begin{multline*}
\big\| \mc R (u_k, \Delta_j \partial_k \rho) \big\|_{L^\infty} \leq \sum_{m, n \in I_j} \Big\{ \big\| \Delta_m (- \Delta)^{- \alpha / 2} \rho \big\|_{L^\infty} \big\| \Delta_n \Delta_j \partial_d \rho \big\|_{L^\infty} \\
+ \big\| \Delta_m (- \Delta)^{-1 - \alpha/2} \nabla \partial_d \rho \big\|_{L^\infty} \big\| \Delta_n \Delta_j \nabla \rho \big\|_{L^\infty} \Big\}.
\end{multline*}
On the one hand, we have to use $L^p$ regularity of the solution to deal with the order negative operators $(- \Delta)^{- \alpha / 2}$ and $\nabla (- \Delta)^{-1 - \alpha/2} \partial_d$. Thanks to Proposition \ref{p:LPoperatorBound}, we have
\begin{equation*}
\begin{split}
& \big\| \Delta_m (- \Delta)^{- \alpha / 2} \rho \big\|_{L^\infty} \lesssim 2^{-m} c_m(t) \Big( \| \rho \|_{L^p} + \| \rho \|_{B^{1 - \alpha}_{\infty, 1}} \Big) \\
& \big\| \Delta_m (- \Delta)^{-1 - \alpha/2} \nabla \partial_d \rho \big\|_{L^\infty} \lesssim 2^{-m} c_m(t) \Big( \| \rho \|_{L^p} + \| \partial_d \rho \|_{B^{- \alpha}_{\infty, 1}} \Big)
\end{split}
\end{equation*}\label{eq:3rdTerm}
which we plug into the estimate for the remainder and obtain
\begin{equation}
\begin{split}
\big\| \mc R (u_k, \Delta_j \partial_k \rho) \big\|_{L^\infty} & \lesssim \sum_{m, n \in I_j} \Big\{ 2^{-m} c_m(t) \Big( \| \rho \|_{L^p} + \| \rho \|_{B^{1 - \alpha}_{\infty, 1}} \Big) 2^{j \alpha} c_j(t) \| \partial_d \rho \|_{B^{- \alpha}_{\infty, 1}} \\
& \qquad \qquad \qquad + 2^{-m} c_m(t) \Big( \| \rho \|_{L^p} + \| \partial_d \rho \|_{B^{- \alpha}_{\infty, 1}} \Big) 2^{j \alpha} c_j(t) \| \rho \|_{B^{1 - \alpha}_{\infty, 1}} \Big\} \\
& \lesssim 2^{-j(1 - \alpha)} c_j(t) \| \rho \|_{L^p \cap B^{1 - \alpha}_{\infty, 1}} \Big( \| \rho \|_{L^p} + \| \partial_d \rho \|_{B^{- \alpha}_{\infty, 1}} \Big).
\end{split}
\end{equation}

\medskip

\textbf{Fourth term:} $\Delta_j \mc T_{\partial_k \rho}(u_k)$. For this term, we may directly use the results of Proposition \ref{p:op}. Let us write the paraproduct as
\begin{equation*}
\mc T_{\partial_k \rho} (u_k) = \mc T_{\partial_d \rho} (- \Delta)^{- \alpha/2} \rho + \mc T_{\partial_k \rho} \partial_k (- \Delta)^{-1 - \alpha / 2} \partial_d \rho.
\end{equation*}
In order to apply Proposition \ref{p:op}, we must note that there are two cases, depending on whether the derivatives $\partial_k \rho \in B^{- \alpha}_{\infty, 1}$ have negative regularity (when $\alpha > 0$) or not (when $\alpha = 0$). In the first case, $\alpha > 0$, the second and third inequality of Proposition \ref{p:op} can be used, whereas in the second case we use the first inequality. In both cases, the embedding $B^{- \alpha}_{\infty, 1} \subset B^{- \alpha}_{\infty, \infty}$ allows us to write:
\begin{equation*}
\| \mc T_{\partial_d \rho} (- \Delta)^{- \alpha/2} \rho \|_{B^{1 - \alpha}_{\infty, 1}} \lesssim \| \partial_d \rho \|_{B^{- \alpha}_{\infty, 1}} \| (- \Delta)^{- \alpha / 2} \rho \|_{B^1_{\infty, 1}} \lesssim \| \partial_d \rho \|_{B^{- \alpha}_{\infty, 1}} \Big( \| \rho \|_{L^p} + \| \rho \|_{B^{1 - \alpha}_{\infty, 1}} \Big)
\end{equation*}
and 
\begin{equation*}
\begin{split}
\| \mc T_{\partial_k \rho} (- \Delta)^{-1 - \alpha / 2} \partial_k \partial_d \rho \|_{B^{- \alpha}_{\infty, 1}} & \lesssim \| \nabla \rho \|_{B^{- \alpha}_{\infty, 1}} \| (- \Delta)^{-1 - \alpha / 2} \partial_k \partial_d \rho \|_{B^1_{\infty, 1}} \\
& \lesssim \| \rho \|_{B^{1 - \alpha}_{\infty, 1}} \Big( \| \rho \|_{L^p} + \| \partial_d \rho \|_{B^{- \alpha}_{\infty, 1}} \Big).
\end{split}
\end{equation*}
In terms of the size of the Littlewood-Paley blocks, this means that we have an upper bound identical to the ones above:
\begin{equation}\label{eq:4thTerm}
\big\| \Delta_j \mc T_{\partial_k \rho} (u_k) \big\|_{L^\infty} \lesssim 2^{-j(1 - \alpha)} c_j(t) \| \rho \|_{L^p \cap B^{1 - \alpha}_{\infty, 1}} \Big( \| \rho \|_{L^p} + \| \partial_d \rho \|_{B^{- \alpha}_{\infty, 1}} \Big).
\end{equation}

\medskip

\textbf{Fifth term:} $\Delta_j \mc R (\partial_k \rho, u_k)$. Just as for the previous term, we may simply use Proposition \ref{p:op} on the remainder
\begin{equation*}
\mc R (\partial_k \rho, u_k) = \mc R \big( \partial_d \rho, (- \Delta)^{- \alpha/2} \rho \big) +  \mc R \big( \partial_k \rho, \partial_k (- \Delta)^{-1 - \alpha/2} \partial_d \rho \big).
\end{equation*}
Because we have assumed that $\alpha < 1$, the remainder is always defined in the $B^{- \alpha}_{\infty, 1} \times B^1_{\infty, 1} \tend B^{1 - \alpha}_{\infty, 1}$ topology. Therefore, Proposition \ref{p:op} gives
\begin{equation*}
\begin{split}
& \big\| \mc R \big( \partial_d \rho, (- \Delta)^{- \alpha/2} \rho \big) \big\|_{L^\infty} \lesssim \| \partial_d \rho \|_{B^{- \alpha}_{\infty, 1}} \Big( \| \rho \|_{L^p} + \| \rho \|_{B^{1 - \alpha}_{\infty, 1}} \Big) \\
& \big\| \mc R \big( \partial_k \rho, \partial_k (- \Delta)^{-1 - \alpha/2} \partial_d \rho \big) \big\|_{B^{1 - \alpha}_{\infty, 1}} \lesssim \| \rho \|_{B^{1 - \alpha}_{\infty, 1}} \Big( \| \rho \|_{L^p} + \| \partial_d \rho \|_{B^{- \alpha}_{\infty, 1}} \Big).
\end{split}
\end{equation*}
In terms of the Littlewood-Paley blocks, this gives, just as before,
\begin{equation}\label{eq:5thTerm}
\big\| \Delta_j \mc R (\partial_k \rho, u_k) \big\|_{L^\infty} \lesssim 2^{-j(1 - \alpha)} c_j(t) \| \rho \|_{L^p \cap B^{1 - \alpha}_{\infty, 1}} \Big( \| \rho \|_{L^p} + \| \partial_d \rho \|_{B^{- \alpha}_{\infty, 1}} \Big).
\end{equation}

\medskip

\textbf{Conclusion:} With inequalities \eqref{eq:1stTerm}, \eqref{eq:2ndTerm}, \eqref{eq:3rdTerm}, \eqref{eq:4thTerm} and \eqref{eq:5thTerm} put together, we may finally use $L^\infty$ estimates on the transport equation \eqref{eq:contCritTransport} so as to obtain, for $j \geq -1$ and $T > 0$,
\begin{equation*}
\begin{split}
2^{j(1 - \alpha)} \| \Delta_j \rho \|_{L^\infty_T(L^\infty)} & \lesssim 2^{j(1 - \alpha)} \| \Delta_j \rho_0 \|_{L^\infty} + \int_0^T \big\| \big[ u \cdot \nabla, \Delta_j \big] \rho \big\|_{L^\infty} \dt \\
& \lesssim 2^{j(1 - \alpha)} \| \Delta_j \rho_0 \|_{L^\infty} + \int_0^T c_j(t) \| \rho \|_{L^p \cap B^{1 - \alpha}_{\infty, 1}} \Big( \| \rho \|_{L^p} + \| \partial_d \rho \|_{B^{- \alpha}_{\infty, 1}} \Big) \dt,
\end{split}
\end{equation*}
and after summation on $j \geq -1$,
\begin{equation*}
\| \rho \|_{L^\infty_T(L^p \cap B^{1 - \alpha}_{\infty, 1})} \lesssim \| \rho_0 \|_{L^p \cap B^{1 - \alpha}_{\infty, 1}} + \int_0^T \| \rho \|_{L^p \cap B^{1 - \alpha}_{\infty, 1}} \Big( \| \rho \|_{L^p} + \| \partial_d \rho \|_{B^{- \alpha}_{\infty, 1}} \Big) \dt.
\end{equation*}
By remembering that because the velocity field is Lipschitz, the $L^p$ norms are conserved $\| \rho(t) \|_{L^p} = \| \rho_0 \|_{L^p}$, we may use Grönwall's lemma on the previous integral inequality to show that $\rho$ remains indeed in the space $L^\infty_T(L^p \cap B^{1 - \alpha}_{\infty, 1})$ as long as condition \eqref{eq:thContcritPartialD} holds.
\end{proof}

\section{Global Existence and Uniqueness: $1 \leq \alpha \leq d$}

In this Section, we focus on finding \textsl{global} well-posedness of problem \ref{ieq:AIPM} when $\alpha \geq 1$. We had already obtained such a result when $\alpha > 1$ under a mild (Hölder) regularity assumption in Corollary \ref{c:globalFromLocal}. However, this leaves open the limit cases $\alpha = 1$ and $\alpha = d$, in addition to letting us wonder if it is possible to get rid of the additional regularity requirement.

To answer these questions, we will use two different setting: that of Lebesgue spaces when $1 < \alpha < d$ and that of Besov spaces when $\alpha \in \{ 1, d\}$.

\medskip

The difference between these two settings is lies in the spaces in which it is possible to prove stability estimates, and hence obtain uniqueness of a solution. If $\rho_1$ and $\rho_2$ are two solutions, the difference $\delta \rho = \rho_2 - \rho_1$ solves
\begin{equation}\label{eq:GSintro}
\big( \partial_t + u_1 \cdot \nabla \big) \delta \rho = - \D ( \rho_2 \delta u).
\end{equation}
Now, for the sake of the argument, assume that the solutions are simply bounded $\rho_1, \rho_2 \in C^0(L^\infty)$. Then, we are faced with two different problems. Firstly, the velocity field $u_1$ is not necessarily Lipschitz, but $\log$-Lipschitz, and this means we have to deal with possible loss of regularity in the estimates for the transport equations. Secondly, the righthand side of \eqref{eq:GSintro} will have, at best, the regularity of $\nabla \rho$, that is a $s = -1$ order of regularity. This is no problem when $\alpha > 1$, as the gain of regularity of order $\alpha$ allows for $B^{-1 - \epsilon}_{q, \infty} \subset B^{- \alpha}_{q, \infty}$ estimates in the transport equation, thus giving order $s = -1$ bounds on $\delta \rho$, but it is a major obstacle when $\alpha = 1$. 

Our way around this will be to treat the case $\alpha = 1$ differently by constructing solutions in the endpoint Besov-Lipschitz space $B^0_{\infty, 1}$, which is slightly smaller than $L^\infty$. In doing so, we will use the method of Hmidi and Keraani \cite{HK} based on linear estimates for the transport equation. Incidentally, we will also take advantage of estimates of this type to construct solutions that also lie in the homogeneous Besov space $\dot{B}^0_{\frac{d}{p}, 1}$, thus covering the special case $d = \alpha$.

\subsection{Well-Posedness in Critical Lebesgue Spaces}

Firstly, we consider the case where $1 < \alpha < d$, where we will be able to prove global well-posedness in \textsl{critical} Lebesgue spaces $L^{d/(\alpha - 1)}$.

\begin{thm}\label{t:globalStrongLebesgue}
Consider $d \geq 2$ and $\alpha \in ]1, d[$. Let $p \in [1, d/\alpha [$ and set $q = d/(\alpha - 1)$. Then, for any initial datum $\rho_0 \in L^p \cap L^q$, problem \eqref{ieq:AIPM} has a unique solution $\rho \in L^\infty(L^p \cap L^q)$.
\end{thm}

\begin{rmk}
In particular, when $d = 3$ and $\alpha = 2$, we recover the result of Mecherbet and Sueur \cite{MS}, who prove global well-posedness for initial data with $L^3$ regularity (and sufficient integrability).
\end{rmk}

\begin{rmk}\label{r:weakSol}
Note that the assumptions of Theorem \ref{t:globalStrongLebesgue} above are stronger than those of Theorem \ref{t:globalWeakSol} giving the existence of global weak solutions in Lebesgue spaces, since the inequality
\begin{equation}\label{eq:remarkExponentsIneq}
q = \frac{d}{\alpha - 1} > \frac{2}{1 + \frac{\alpha}{d}}
\end{equation}
holds as long as $\alpha < d+2$.
\end{rmk}

\begin{rmk}
It is not known whether the solution from Theorem \ref{t:globalStrongLebesgue} is continuous with respect to time: as it will appear in the proof, the velocity field is only $\log$-Lipschitz, so the proof of Lebesgue time continuity by means of the Friedrichs commutator Lemma (Lemma \ref{l:DPLcommutator}), as in STEP 6 of the proof of Theorem \ref{t:localWP}, fails. However, if the initial datum is subcritical $\rho_0 \in L^p \cap L^r$ with $r > q = d/(\alpha - 1)$, then the velocity field is Lipschitz and it becomes possible to prove that $\rho \in C^0_b(L^p \cap L^r)$.
\end{rmk}

\begin{proof}[Proof (of Theorem \ref{t:globalStrongLebesgue})]
First of all, because $1 < d/\alpha < q$ and $\rho_0 \in L^p \cap L^q$, interpolation of Lebesgue spaces will always allow us to assume that $p > 1$.

Thanks to Remark \ref{r:weakSol}, we see that the assumptions of Theorem \ref{t:globalStrongLebesgue} are enough to provide the existence of a global weak solution: because of inequality \eqref{eq:remarkExponentsIneq}, it is always possible to find a $r > p$ such that
\begin{equation*}
\frac{2}{1 + \frac{\alpha}{d}} < r < \frac{d}{\alpha} < \frac{d}{\alpha - 1} = q,
\end{equation*}
since the inequality $2 \left( 1 + \frac{\alpha}{d} \right)^{-1} < \frac{d}{\alpha}$ holds al long as $\alpha < d$. Noting that $p < r < q$, interpolation of Lebesgue spaces shows that the initial datum $\rho_0 \in L^r \subset L^p \cap L^q$ satisfies the assumptions of Theorem \ref{t:globalWeakSol}, thus providing a global weak solution $\rho \in L^\infty(L^p \cap L^q)$ with
\begin{equation*}
\forall t \geq 0, \qquad \| \rho(t) \|_{L^p \cap L^q} \leq \| \rho_0 \|_{L^p \cap L^q}.
\end{equation*}
We only have to show uniqueness of this solution. We therefore consider two solutions $\rho_1, \rho_2 \in L^\infty(L^p \cap L^q)$ related to the same initial datum $\rho_{0, 1} = \rho_{0, 2} = \rho_0$ and form the difference function $\delta \rho = \rho_2 - \rho_1$. Let $\delta \rho_0 = 0$ be defined accordingly. Then, $\delta \rho$ is a weak solution of the following system
\begin{equation}\label{eq:stabEQ}
\begin{cases}
\big( \partial_t + u_2 \cdot \nabla \big) \delta \rho = - \D (\rho_1 \delta u) \\
\delta u = (- \Delta)^{- \alpha / 2} \P (\delta \rho g).
\end{cases}
\end{equation}

The main difficulty we face at this point is that, because $\rho_2 \in L^q$, the associated velocity field $u_2$ may not be Lipschitz, so we must be careful when using estimates related to the transport equation. As a matter of fact, we may only find weaker regularity: thanks to the embeddings of Proposition \ref{p:refinedEmbed}, 
\begin{equation*}
L^q \subset B^0_{q, 2} \text{ if } q \leq 2 \qquad \text{and} \qquad L^q \subset B^0_{q, q} \text{ if } q \geq 2,
\end{equation*}
we find that, by setting $r = 2$ when $q \leq 2$ and $r = q$ when $q \geq 2$,
\begin{equation*}
\| u_2 \|_{B^1_{\infty, r}} \lesssim \| u_2 \|_{B^{\alpha - d/q}_{\infty, r}} \lesssim \| \rho_2 \|_{L^p} + \| \rho_2 \|_{B^{-d/q}_{\infty, r}} \lesssim \| \rho_2 \|_{L^p} + \| \rho_2 \|_{B^{0}_{q, r}} \lesssim \| \rho \|_{L^p \cap L^q},
\end{equation*}
so that $u_2$ belongs to a Besov space $B^1_{\infty, r}$ whose third index if finite $r < + \infty$. Consequently, the velocity field is $\log$-Lipschitz $u_2 \in L^\infty(LL_\beta)$ for $\beta = 1 - \frac{1}{r} > 0$ (see the discussion in the Appendix on these spaces). More precisely, we may separate low and high frequencies to see that
\begin{equation*}
u_2 = \Delta_{-1}u_2 + ({\rm Id} - \Delta_{-1})u_2 \in L^\infty_T(B^k_{\infty, 1}) + L^\infty_T(B^{\alpha}_{q, r})
\end{equation*}
for all $k \geq 0$. Because $\alpha = 1 + \frac{d}{q}$, we have the embedding $B^\alpha_{q, r} \subset LL^q_\beta$, so the velocity field $u_2$ does indeed satisfy the assumptions of Theorem \ref{t:limitedLoss}, and we may perform estimates on the transport equation, but at the price of a loss of regularity.

\medskip

Let us establish the framework in which we will use the estimaetes of Theorem \ref{t:limitedLoss}. Our goal will be to find bounds on $\delta \rho$ by using the first equation in \eqref{eq:stabEQ}. Since the derivative $\nabla \rho_1$ is involved in the righthand side, we will have to work in a space contained in $B^{-1}_{p, \infty} \cap B^{-1}_{q, \infty}$. For practical reasons that will become clear later, we will conduct the estimates at a lower level of regularity $\sigma_1 < -1$. More precisely, we define a regularity index $\sigma_0$
\begin{equation*}
\sigma_0 := - 1 - d \min \left\{ \frac{1}{q}, \frac{1}{q'}, \frac{1}{p'} \right\},
\end{equation*}
where, as usual, $p'$ and $q'$ are the conjuguate exponents of $p$ and $q$. Then, since we have assumed $p > 1$ and $\alpha = 1 + \frac{d}{q} > 1$, we see that $\sigma_0 < -1$. We may therefore consider a $\sigma_1 < -1$ and a $\epsilon > 0$, whose precise values we will fix later on, such that
\begin{equation*}
\sigma_0 < \sigma_1 - \epsilon < \sigma_1 < -1.
\end{equation*}
In order to use Theorem \ref{t:limitedLoss} with initial datum $\delta \rho \in B^{\sigma_1}_{p, \infty} \cap B^{\sigma_1}_{q, \infty}$, we need to check that the righthand side term $-\D (\rho_1 \delta u)$ satisfies
\begin{equation}\label{eq:ineqLossProdEst}
\forall \eta \in [0, \epsilon], \qquad \big\| \D (\rho_1 \delta u) \big\|_{B^{- \sigma_1 - \eta}_{b, \infty}} \lesssim \Big( \| \delta \rho \|_{B^{- \sigma_1 - \eta}_{p, \infty}} + \| \delta \rho \|_{B^{- \sigma_1 - \eta}_{q, \infty}} \Big) \| \rho_1 \|_{L^p \cap L^q},
\end{equation}
for $b$ taking the values $b=p$ and $b=q$.

\medskip

We focus on proving inequality \eqref{eq:ineqLossProdEst}. By invoking the Bony decomposition for the product $\D(\rho_1 \delta u)$, we obtain
\begin{equation*}
\D (\rho_1 \delta u) = \mc T_{\partial_k \rho_1} (\delta u_k) + \mc T_{\delta u_k} (\partial_k \rho_1 ) + \partial_k \mc R (\delta u_k, \rho_1)
\end{equation*}
where, as usual, there is an implicit sum over the repeated index $k = 1, ..., d$. The paraproducts create no special problem: Proposition \ref{p:op} immediately gives
\begin{equation*}
\begin{split}
\| \mc T_{\partial_k \rho_1} (\delta u_k) \|_{B^{\sigma_1 - \eta}_{b, \infty}} + \| \mc T_{\delta u_k} (\partial_k \rho_1) \|_{B^{\sigma_1 - \eta}_{b, \infty}} & \leq \frac{C}{|1 + \sigma_1 - \eta|} \| \delta u \|_{B^{1 + \sigma_1 - \eta}_{\infty, \infty}} \| \rho_1 \|_{L^p \cap L^q} \\
& \leq \frac{C}{|1 + \sigma_1 - \eta|} \Big( \| \delta \rho \|_{B^{\sigma_1 - \eta}_{p, \infty}} + \| \delta \rho \|_{B^{\sigma_1 - \eta}_{q, \infty}} \Big) \| \rho_1 \|_{L^p \cap L^q}.
\end{split}
\end{equation*}
Since we have chosen $\sigma_1 < 1$, the constant appearing in this inequality is bounded by $|1 + \sigma_1 - \eta|^{-1} \leq |1 + \sigma_1|^{-1}$ for all $\eta \in [0, \epsilon]$, and we can safely ignore it from now on. On the other hand, the remainder term $\mc R (\delta u_k, \rho_1)$ is a bit more delicate to handle: it is only properly defined if the sum of the regularities of $\delta u$ and $\rho_1$ is positive. For this, we take advantage of the fact that $\delta u$ is the image of $\delta \rho$ by the operator $(- \Delta)^{- \alpha/2} \P g$, hence giving it positive regularity as $\alpha > 0$, namely
\begin{equation*}
\| \Delta_{-1} \delta u \|_{B^k_{\infty, 1}} \lesssim \| \delta \rho \|_{B^{\sigma_1 - \eta}_{p, \infty}} \qquad \text{and} \qquad \big\| ({\rm Id} - \Delta_{-1}) \delta u \big\|_{B^{\alpha + \sigma_1 - \eta}_{q, \infty}} \lesssim \| \delta \rho \|_{B^{\sigma_1 - \eta}_{q, \infty}},
\end{equation*}
for any $k \geq 0$. For the low frequency terms, there is no problem. Proposition \ref{p:op} gives
\begin{equation*}
\| \mc R(\Delta_{-1} \delta u_k, \rho_2) \|_{B^k_{b, \infty}} \lesssim \| \Delta_{-1} \delta u \|_{B^k_{\infty, 1}} \| \rho_1 \|_{L^b} \lesssim \| \delta \rho \|_{B^{\sigma_1 - \eta}_{p, \infty}} \| \rho_1 \|_{L^b}.
\end{equation*}
For the high frequency part, we resort to a method nearly identical that used when constructing global weak solutions in the proof of Theorem \ref{t:globalWeakSol} above: we want to consider the remainder as a $B^{\alpha + \sigma_1 - \eta}_{q, \infty} \times B^0_{b, \infty}$ operator. One of two things may happen: either $\frac{1}{\gamma} := \frac{1}{q} + \frac{1}{b} \leq 1$ and Proposition \ref{p:op} yields
\begin{equation*}
\big\| \mc R \big( ({\rm Id} - \Delta_{-1}) \delta u, \rho_1 \big),  \big\|_{B^{\alpha + \sigma_1 - \eta}_{\gamma, \infty}} \lesssim \big\| ({\rm Id} - \Delta_{-1}) \delta u \big\|_{B^{\alpha + \sigma_1 - \eta}_{q, \infty}} \| \rho_1 \|_{B^0_{b, \infty}}
\end{equation*}
since $\alpha + \sigma_1 - \eta \geq \alpha + \sigma - \epsilon > 0$ if we chose $\sigma_1$ close enough to $-1$ and if we take $\epsilon$ small enough. Then, the embedding $B^{\alpha + \sigma_1 - \eta}_{\gamma, \infty} \subset B^{1 + \sigma_1 - \eta}_{b, \infty}$ provides the desired inequality:
\begin{equation}\label{eq:uniqueLQeq1}
\big\| \partial_k \mc R \big( ({\rm Id} - \Delta_{-1}) \delta u, \rho_1 \big),  \big\|_{B^{\sigma_1 - \eta}_{b, \infty}}  \lesssim \big\| ({\rm Id} - \Delta_{-1}) \delta u \big\|_{B^{\alpha + \sigma_1 - \eta}_{q, \infty}} \| \rho_1 \|_{B^0_{b, \infty}}.
\end{equation}
When we have $\frac{1}{q} + \frac{1}{b} > 1$ instead, we must first take advantage of the positive regularity of $({\rm Id} - \Delta_{-1}) \delta u$ and the embeddings of Proposition \ref{p:BesovEmbed} to reach a higher integrability index. Precisely, we get $B^{\alpha + \sigma_1 - \eta}_{q, \infty} \subset B^s_{\gamma, \infty}$ with $\frac{1}{\gamma} + \frac{1}{b} = 1$ and
\begin{equation}\label{eq:prodRegExp}
\begin{split}
s & = \alpha + \sigma_1 - \eta - d \left( \frac{1}{q} - \frac{1}{\gamma}  \right)\\
& = 1 + \sigma_1 - \eta + d \left( 1 - \frac{1}{b} \right).
\end{split}
\end{equation}
Because we have assumed that $q > p > 1$, we can always take this quantity to be positive $s \geq s_0 > 0$ by choosing $\sigma_1$ and $\epsilon$ such that $|\sigma_1 + 1| + \epsilon$ is small enough. In that case, we may apply Proposition \ref{p:op} to get
\begin{equation*}
\big\| \mc R \big( ({\rm Id} - \Delta_{-1}) \delta u, \rho_1 \big),  \big\|_{B^{s}_{1, \infty}} \lesssim C(s_0) \big\| ({\rm Id} - \Delta_{-1}) \delta u \big\|_{B^{\alpha + \sigma_1 - \eta}_{q, \infty}} \| \rho_1 \|_{B^0_{b, \infty}}.
\end{equation*}
The embedding $B^s_{1, \infty} \subset B^{1 + \sigma_1 - \eta}_{b, \infty}$ shows that \eqref{eq:uniqueLQeq1} is also true in that case. Therefore, inequality \eqref{eq:ineqLossProdEst} does indeed hold for all $\eta \in [0, \epsilon]$ provided that we choose $|\sigma_1 + 1| + \epsilon$ small enough.

\medskip

We are at the end of the proof. A direct application of Theorem \ref{t:limitedLoss} gives, for any fixed $T > 0$, after imposing again $\epsilon$ to be small enough, 
\begin{equation*}
\sup_{t \in [0, T]} \Big( \| \delta \rho(t) \|_{B^{\sigma_1 - \epsilon}_{p, \infty}} + \| \delta \rho(t) \|_{B^{\sigma_1 - \epsilon}_{q, \infty}} \Big) \lesssim \| \delta \rho_0 \|_{B^{\sigma_1}_{p, \infty} \cap B^{\sigma_1}_{q, \infty}} \exp \left\{ \frac{C}{\epsilon^{r-1}} \left( \int_0^T \| \rho_2 \|_{L^p \cap L^q} \dt \right)^{r} \right\}.
\end{equation*}
This ends the proof of uniqueness.
\end{proof}

\subsection{Well-Posedness in Besov Spaces: the Endpoint Cases}

The proof of Theorem \ref{t:globalStrongLebesgue} fails in two cases of special interest: those where $\alpha = d$ and where $\alpha = 1$. By going back to the proof, we see that there are two main reasons why these ``endpoint'' cases evade us. 

On the one hand, the fact that we could take $p$ to satisfy $1 < p < d/\alpha$ was crucial to proving inequality \eqref{eq:ineqLossProdEst} with $b=p$. If we had $p=1$ (and even then, we could not take $d=\alpha$) then the regularity exponents \eqref{eq:prodRegExp} involved in the products cannot be positive, as $\sigma_1 < -1$.

On the other hand, and this is a similar problem, if $\alpha = 1$, then we must take $q = + \infty$. In that case, the estimates for the transport equation only hold in function spaces of regularity at best $\sigma_0 = -1$, which is the best regularity exponent we can hope for stability estimates because of the derivatives $\nabla \rho_1$ involved in the righthand side of the first equation in \eqref{eq:stabEQ}. Using estimates which involve a loss of regularity is therefore no longer possible, as it would put us at a $-1 - \epsilon < \sigma_0$ regularity level.

\medskip

The solution to both problems will be to work at a higher regularity level, where we will be able to use global Besov-Lipschitz estimates on the velocity field. Our main tool is a bound for solutions of the transport equation with regularity $s=0$ that is linear with respect to the velocity field (see \ref{p:linEstTV}).

\begin{thm}\label{t:globalWPBesov}
Consider $d \geq 2$ and $\alpha \in [1, d]$. Let $p = d/\alpha$ and $q = d/(\alpha - 1)$. Then, for any initial datum $\rho_0 \in \dot{B}^0_{p, 1} \cap B^0_{q, 1}$, problem \eqref{ieq:AIPM} has a unique solution $\rho \in C^0(\dot{B}^0_{p, 1} \cap B^0_{q, 1})$.
\end{thm}

\begin{proof}
Strictly speaking, we have not yet proved existence of such a solution: so far solutions have been constructed only for initial data that has some integrability $\rho_0 \in L^r$ with $r < d/\alpha$ (in addition to being regular enough). However, once we show that global \textsl{a priori} and stability estimates hold in $\dot{B}^0_{p, 1} \cap B^0_{q, 1}$, then it is a straightforward thing to construct a solution by means of a compactness argument (replicating the different steps in the proof of Theorem \ref{t:globalWeakSol}). For the sake of conciseness, we will focus on the estimates and leave aside the full proof of existence.

\medskip

\textbf{SPEP 1: \textsl{a priori estimates}.} Assume that $\rho$ is a regular solution of \eqref{ieq:AIPM} with associated velocity field $u$. Start by noting that, thanks to Proposition \ref{p:LPoperatorBound}, the following Besov-Lipschitz bound holds:
\begin{equation*}
\| u \|_{B^1_{\infty, 1}} \lesssim \| \rho \|_{\dot{B}^0_{p, 1}} + \| \rho \|_{B^0_{q, 1}}.
\end{equation*}
Set $X = \dot{B}^0_{p, 1} \cap B^0_{q, 1}$. We then use Proposition \ref{p:linEstTV}, which provide, as both spaces $\dot{B}^0_{p, 1}$ and $B^0_{q, 1}$ have regularity exponent $s=0$,
\begin{equation*}
\begin{split}
\forall \, T > 0, \qquad \| \rho \|_{L^\infty_T(X)} & \lesssim \| \rho_0 \|_X \left( 1 + \int_0^T \| \nabla u \|_{L^\infty} \dt \right) \\
& \lesssim \| \rho_0 \|_X \left( 1 + \int_0^T \| \rho \|_{X} \dt \right).
\end{split}
\end{equation*}
Grönwall's lemma immediately gives
\begin{equation*}
\forall \, T > 0, \qquad \| \rho \|_{L^\infty_T(X)} \lesssim \| \rho_0 \|_X \exp \big( CT \| \rho_0 \|_X \big),
\end{equation*}
so that the solution naturally lies in $L^\infty_{\rm loc} (\dot{B}^0_{p, 1} \cap B^0_{q, 1})$. Besov continuity of the solution with respect to time follows from the same arguments as in the proof of Theorem \ref{t:localWP}.

\medskip

\textbf{STEP 2: stability estimates.} We intend to prove stability estimates for \eqref{ieq:AIPM}. We therefore consider two solutions $\rho_1, \rho_2 \in L^\infty_{\rm loc}(X)$ with velocity fields $u_1$ and associated to the same initial datum $\rho_{0, 1} = \rho_{0, 2}$. Then the difference $\delta \rho = \rho_2 - \rho_1$ solves
\begin{equation*}
\begin{cases}
\big( \partial_t + u_2 \cdot \nabla \big) \delta \rho = - \D (\rho_1 \delta u)\\
\delta u = (- \Delta)^{- \alpha / 2} \P (\delta \rho g),
\end{cases}
\end{equation*}
We note $\delta \rho_0$ the difference of the initial data $\delta \rho_0 = \rho_{0, 2} - \rho_{0, 1} = 0$. As we mentioned above, it is possible to perform estimates at a $s=-1$ level of regularity in the transport equation, provided we work in Besov spaces whose third index is $r = +\infty$. We therefore set $Y = \dot{B}^{-1}_{p, \infty} \cap B^{-1}_{q, \infty}$. The estimates of Theorem \ref{th:transport} then yield, for $T > 0$,
\begin{equation}\label{eq:uniqueBesovGlobalEQ1}
\| \delta \rho \|_{L^\infty_T(Y)} \lesssim \left( \| \delta \rho_0 \|_{Y} + \int_0^T \big\| \D (\rho_1 \delta u) \big\|_Y \dt \right) \exp \left\{ C \int_0^T \| \nabla u_2 \|_{L^\infty} \right\}.
\end{equation}
Let us focus on the product $\D (\rho_1 \delta u)$, which we wish to estimate in $Y$. The problem is that, if $\delta \rho \in Y$, then the best bound we can find for the velocity difference is $\delta u \in B^0_{\infty, \infty}$, so it may not be bounded and we risk loosing regularity in the product: for example, if $\delta u \in B^0_{\infty, \infty}$, then the best bound Proposition \ref{p:op} gives for the paraproduct $\mc T_{\delta u_k} (\partial_k \rho_1)$ is
\begin{equation*}
\forall \epsilon > 0, \qquad \| \mc T_{\delta u_k} (\partial_k \rho_1) \|_{B^{-1 - \epsilon}_{q, \infty}} \lesssim \frac{1}{\epsilon} \, \| \delta u \|_{B^0_{\infty, \infty}} \| \nabla \rho_1 \|_{B^0_{q, \infty}}.
\end{equation*}
The way we solve this issue is to resort to the fact that, in reality, since $\delta u$ is a difference of velocity fields $\delta u = u_2 - u_1$, then it has same regularity as these. The logarithmic interpolation inequality of Proposition \ref{p:LogInt} gives
\begin{equation*}
\| \delta u \|_{L^\infty} \lesssim \| \delta u \|_{B^0_{\infty, \infty}} \left\{ 1 + \log \left( \frac{\| (u_1, u_2) \|_{B^1_{\infty, 1}}}{\| \delta u \|_{B^0_{\infty, \infty}}} \right) \right\},
\end{equation*}
and therefore, the inequalities
\begin{equation*}
\big\| \D (\rho_1 \delta u) \big\|_{\dot{B}^{-1}_{p, \infty}} \lesssim \| \rho_1 \delta u \|_{L^p} \lesssim  \| \rho_1 \|_{L^p} \| \delta u \|_{L^\infty}
\end{equation*}
and
\begin{equation*}
\big\| \D (\rho_1 \delta u) \big\|_{B^{-1}_{q, \infty}} \lesssim \| \rho_1 \delta u \|_{L^q} \lesssim  \| \rho_1 \|_{L^q} \| \delta u \|_{L^\infty}
\end{equation*}
lead to 
\begin{equation*}
\big\| \D (\rho_1 \delta u) \big\|_{Y} \lesssim \| \delta \rho \|_{Y} \| \rho_1 \|_{L^p \cap L^q} \left\{ 1 + \log \left( \frac{\| (u_1, u_2) \|_{B^1_{\infty, 1}}}{\| \delta \rho \|_{Y}} \right) \right\}.
\end{equation*}
Applying this directly inside inequality \eqref{eq:uniqueBesovGlobalEQ1} leads to the following integral inequality: for any fixed $T > 0$, we have, for all $t \in [0, T]$,
\begin{equation}
\| \delta \rho \|_{L^\infty_t(Y)} \leq C \| \delta \rho_0 \|_Y + C \int_0^t \mu \big( \| \delta \rho \|_Y \big),
\end{equation}
where $\mu$ is a nonnegative nondecreasing function defiend by
\begin{equation*}
    \mu(r) = r \left\{ 1 + \log \left( \frac{1}{r} \| (u_1, u_2) \|_{L^\infty_T(B^1_{\infty, 1})} \right) \right\} = r \big( C - \log(r) \big),
\end{equation*}
and where the constants $C = C(\| \rho_0 \|_X, T)$ appearing in the two previous equations depends on $\| (\rho_1, \rho_2) \|_{L^\infty_T(X)}$. In particular, $\mu$ is an Osgood modulus of continuity, and the Osgood Lemma (see for example Lemma 3.4 in \cite{BCD}) insures that $\| \delta \rho (t) \|_{Y} = 0$ for all $t \in [0, T]$, thus proving uniqueness for the solution.
\end{proof}

\section{Lifespan Estimate for $\alpha \rightarrow 1^-$}\label{s:lifespan}

In this section, we study the lifespan of solutions who evolve around a stationary profile $R(x_d)$, that is functions $\rho$ such that $\tilde{\rho}(t,x) = \rho(t, x) + R(x_d)$ is a solution of \eqref{ieq:AIPM}. These solve the following problem
\begin{equation}\label{eq:IPMaroundEquilibrium}
\begin{cases}
\Big( \partial + u \cdot \nabla \Big) \rho = - \partial_d R(x_d) u_d\\
u = (- \Delta)^{-\alpha / 2} \P (\rho e_d)
\end{cases}
\end{equation}
with the initial datum $\rho_0$ lying in the appropriate space of functions. This problem is not very different from \eqref{ieq:AIPM}, as the added righthand side in the first equation is linear and of lower order with respect to the unknown.

The constant profile $R(x_d)$ is a typical form of stationary solutions. In fact, by using \textsl{mutatis mutandi} an argument of Elgindi (see Lemma 1.1 in \cite{Elgindi}), stationary solutions that decay sufficiently fast at infinity can be seen to be always functions of $x_d$ only.

\medskip

As we have seen above, if $R$ is, say, regular and bounded, such solutions will be global for $\alpha \geq 1$. Our purpose in this section is to study their lifespan. We will show that it can become arbitrarily large when $\alpha \rightarrow 1^-$ or when the initial perturbation $\rho_0$ is small.

\begin{thm}\label{t:lifespanIncrease}
Assume that $d \geq 2$ and $R(x_d) \in B^{3/2}_{\infty, 1}(\mathbb{R}^d)$ and consider an exponent $p \in [1, d[$. For any $\frac{1}{2} < \alpha < 1$ and initial datum $\rho_0 \in X := L^{p} \cap B^{1-\alpha}_{\infty, 1}$, the lifespan $T^*$ of the unique associated solution of \eqref{eq:IPMaroundEquilibrium} is at least
\begin{equation*}
\| \partial_d R \|_{B^{1- \alpha}_{\infty, 1}} {T^*}^2 + T^* \geq \frac{C}{\| \rho_0 \|_{X}}\log \left( 1 + \frac{C}{1 - \alpha} \right).
\end{equation*}
\end{thm}

The proof of the existence of global solutions relies on estimates for the transport equation that are linear with respect to the velocity field (see Proposition \ref{p:linEstTV} and the proof of Theorem \ref{t:globalWPBesov} above). When $\alpha < 1$, we can no longer use these, since there is no Besov space of exponent $s = 0$ that can be embedded in $B^{1 - \alpha}_{\infty, 1}$. Instead, we revisit the proof of the linear estimate of Proposition \ref{p:linEstTV} and adapt it to spaces of positive regularity, while carefully tracking the dependence on $\sigma := 1 - \alpha$.

\begin{lemma}\label{l:transportLin}
Let $v$ be a smooth divergence-free velocity field and $g$ a smooth function. If $f$ is a regular solution of the linear transport equation
\begin{equation*}
\Big( \partial_t + v \cdot \nabla \Big) f = g \qquad \text{with initial datum } f_0,
\end{equation*}
then $f$ satisfies the following \textsl{a priori} estimate: for all $0 < \sigma_0 < 1/2$, the following inequality holds for $0 \leq \sigma \leq \sigma_0$
\begin{equation*}
\| f \|_{{L^\infty_T}(B^{\sigma}_{\infty, 1})} \lesssim \Big( \| f_0 \|_{B^{\sigma}_{\infty, 1}} + \| g \|_{L^1_T (B^{\sigma}_{\infty, 1})} \Big) \left( 1 + \int_0^T \| \nabla v \|_{L^\infty} \dt \right) \exp \left\{ C \sigma \int_0^T \| \nabla v \|_{L^\infty} \dt \right\},
\end{equation*}
where the (possibly implicit) constants appearing in the above may depend on $\sigma_0$.
\end{lemma}

\begin{proof}[Proof (of Lemma \ref{l:transportLin})]
Let, for $k \geq -1$, $f_k$ be the solutions of the initial value problem
\begin{equation*}
\Big( \partial_t + v \cdot \nabla \Big) f_k = \Delta_k g \qquad \text{with initial datum } \Delta_k f_0.
\end{equation*}
Since the transport equation is linear and has unique solutions in Besov spaces (see Theorem \ref{th:transport} below), $f$ is the sum of all the $f_k$. Therefore, for a $N \geq 1$ whose value we will fix later,
\begin{equation*}
\| f \|_{B^{\sigma}_{\infty, 1}} \leq \sum_{j, k \geq -1} 2^{\sigma j} \| \Delta_j f_k \|_{L^\infty} = \sum_{|j - k | < N} 2^{\sigma j} \| \Delta_j f_k \|_{L^\infty} + \sum_{|j - k| \geq N} 2^{\sigma j} \| \Delta_j f_k \|_{L^\infty}.
\end{equation*}

\medskip

We start by looking at the terms $|j - k| < N$. Since $k$ and $j$ are (reasonably) close, we expect the dyadic block $\Delta_j$ to have little influence. We therefore use simple $L^\infty$ estimates so as to obtain
\begin{equation}\label{eq:betterTimeEQ1}
\begin{split}
\sum_{|j - k| < N} 2^{\sigma j} \| \Delta_j f_k \|_{L^\infty} & \lesssim \sum_{|j - k| < N} 2^{\sigma j} \| f_k \|_{L^\infty} = \sum_{|j - k| < N} 2^{\sigma j} \Big( \| \Delta_k f_0 \|_{L^\infty} + \| \Delta_k g \|_{L^1_T(L^\infty)} \Big) \\
& = \sum_{|j - k| < N} 2^{\sigma (j-k)} 2^{\sigma k} \Big( \| \Delta_k f_0 \|_{L^\infty} + \| \Delta_k g \|_{L^1_T(L^\infty)} \Big) \\
& \lesssim N 2^{\sigma N} \sum_k 2^{\sigma k} \Big( \| \Delta_k f_0 \|_{L^\infty} + \| \Delta_k g \|_{L^1_T(L^\infty)} \Big) \\
& \lesssim N 2^{\sigma N} \Big( \| f_0 \|_{B^{\sigma}_{\infty, 1}} + \| g \|_{L^1_T(B^\sigma_{\infty, 1})} \Big).
\end{split}
\end{equation}

\medskip

We now look at the remaining terms $|j - k| \geq N$. We will use standard estimates for the transport equation in Besov spaces, as in Theorem \ref{th:transport}. These provide
\begin{equation*}
\| f_k \|_{B^{\pm 2 \sigma_0}_{\infty, 1}} \lesssim \Big( \| \Delta_k f_0 \|_{B^{\pm 2 \sigma_0}_{\infty, 1}} + \| \Delta_k g \|_{L^1_T(B^{\pm 2 \sigma_0}_{\infty, 1})} \Big) \exp \left\{ C \int_0^T \| \nabla v \|_{L^\infty} \right\}.
\end{equation*}
Note that, in the preceding inequality, the Lipschitz norm $\| \nabla v \|_{L^\infty}$ is enough because we have assumed $\sigma_0 < 1/2$. We deduce the existence of a family of sequences $\big( c_j(t) \big)_{j \geq -1}$ lying in the unit ball of $l^1(j \geq -1)$ such that
\begin{equation*}
\| \Delta_j f_k \|_{L^\infty} \lesssim c_j(t) \, 2^{- 2 \sigma_0 |j - k|} \, \Big( \| \Delta_k f_0 \|_{L^\infty} + \| \Delta_k g \|_{L^1_T(L^\infty)} \Big) \exp \left\{ C \int_0^T \| \nabla v\|_{L^\infty} \dt \right\}.
\end{equation*}
In light of this, we may write
\begin{equation}\label{eq:betterTimeEQ2}
\begin{split}
\sum_{|j - k| \geq N} 2^{\sigma j} \| \Delta_j f_k \|_{L^\infty} & \lesssim \sum_{|j - k| \geq N} 2^{- 2\sigma_0 |j - k|} \, 2^{\sigma j} c_j(t) \, \Big( \| \Delta_k f_0 \|_{L^\infty} + \| \Delta_k g \|_{L^1_T(L^\infty)} \Big) \\
& \qquad \qquad \qquad \qquad \qquad \qquad \times \exp \left\{ C \int_0^T \| \nabla v\|_{L^\infty} \dt \right\} \\
& \leq \sum_{|j - k| \geq N} 2^{- \sigma_0 |j - k|} \, 2^{\sigma k} c_j(t) \, \Big( \| \Delta_k f_0 \|_{L^\infty} + \| g \|_{L^1_T(L^\infty)} \Big) \\
& \qquad \qquad \qquad \qquad \qquad \qquad \times \exp \left\{ C \int_0^T \| \nabla v\|_{L^\infty} \dt \right\} \\
& \leq \Big( \| f_0 \|_{B^{\sigma}_{\infty, 1}} + \| g \|_{L^1_T(B^\sigma_{\infty, 1})} \Big) 2^{-\sigma_0 N} \exp \left\{ C \int_0^T \| \nabla v \|_{L^\infty} \dt \right\}.
\end{split}
\end{equation}

\medskip

We may now finish the proof of the lemma. By putting both inequalities \eqref{eq:betterTimeEQ1} and \eqref{eq:betterTimeEQ2} together, we get
\begin{equation*}
\| f \|_{L^\infty_T (B^{\sigma}_{\infty, 1})} \lesssim \| f_0 \|_{B^{\sigma}_{\infty, 1}} \left( N 2^{\sigma N} + 2^{-\sigma_0 N} \exp \left\{ C \int_0^T \| \nabla v \|_{L^\infty} \dt \right\} \right).
\end{equation*}
Setting $N \sigma_0 \log(2) = 1 + C \int_0^T \| \nabla v \|_{L^\infty}$ so as to exactly balance the exponential term, we end up with the estimate we were seeking.
\end{proof}

\begin{proof}[Proof (of Theorem \ref{t:lifespanIncrease})]
We are now equipped to complete the proof of Theorem \ref{t:lifespanIncrease}. Because the added righthand side term in \eqref{eq:IPMaroundEquilibrium} is linear and of lower order with respect to the unknown, there is no added difficulty to prove existence and uniqueness of solutions when $R \neq 0$ compared to what we have already done above. We therefore will be content with finding \textsl{a priori} estimates.

\medskip

Applying Lemma \ref{l:transportLin} to the transport equation in \eqref{eq:IPMaroundEquilibrium}, we may write a $B^{1 - \alpha}_{\infty, 1}$ estimate for the solution:
\begin{multline}\label{eq:betterTimeEQ3}
\| \rho \|_{B^{1 - \alpha}_{\infty, 1}} \lesssim \Big( \| \rho_0 \|_{B^{1 - \alpha}_{\infty, 1}} + \int_0^T \| \partial_d R u_d \|_{B^{1 - \alpha}_{\infty, 1}} \dt \Big) \left( 1 + \int_0^T \| \nabla u \|_{L^\infty} \dt \right) \\
\times \exp \left\{ (1 - \alpha)C \int_0^T \| \nabla u \|_{L^\infty} \dt \right\}.
\end{multline}
We focus for a while on the product $\partial_d R . u_d$. Proposition \ref{p:op} and the fact that $1/2 < \alpha < 1$ insure that we have $\| \partial_d R u\|_{B^{1-\alpha}_{\infty, 1}} \lesssim \| \partial_d R \|_{B^{1/2}_{\infty, 1}} \| u \|_{B^{1 - \alpha}_{\infty, 1}}$, while the velocity satisfies, by Proposition \ref{p:LPoperatorBound},
\begin{equation*}
\| u \|_{B^{1 - \alpha}_{\infty, 1}} \lesssim \| \rho \|_{L^{p}} + \| \rho \|_{B^{1 - 2 \alpha}_{\infty, 1}} = \| \rho_0 \|_{L^{p}} + \| \rho \|_{B^{1 - 2 \alpha}_{\infty, 1}}.
\end{equation*}
This last inequality shows that the term $\partial_d R(x_d)u_d$ is of lower order than the solution $\rho$, so we expect it to be easier to bound. As a matter of fact, because of the embeddings
\begin{equation*}
B^{1 - \alpha}_{\infty, 1} \subset L^\infty \subset B^0_{\infty, \infty} \subset B^{1 - 2\alpha}_{\infty},
\end{equation*}
which hold for $1/2 < \alpha < 1$, we see that the solution must be uniformly bounded in $B^{1 - 2\alpha}_{\infty, 1}$. More precisely,
\begin{equation*}
\forall t \geq 0, \qquad \| \rho(t) \|_{B^{1 - 2 \alpha}_{\infty, 1}} \lesssim \| \rho(t) \|_{L^\infty} = \| \rho_0 \|_{L^\infty} \lesssim \| \rho_0 \|_{B^{1 - \alpha}_{\infty, 1}}.
\end{equation*}
We may therefore write, instead of inequality \eqref{eq:betterTimeEQ3}, 
\begin{equation*}
\begin{split}
\| \rho \|_{B^{1 - \alpha}_{\infty, 1}} & \lesssim \Big( \| \rho_0 \|_{B^{1 - \alpha}_{\infty, 1}} + T \| \partial_d R \|_{B^{1/2}_{\infty, 1}} \| \rho_0 \|_{L^p \cap B^{1 - \alpha}_{\infty, 1}} \Big) \left( 1 + \int_0^T \| \nabla u \|_{L^\infty} \dt \right) \\
& \qquad \qquad \qquad \qquad \qquad \qquad \qquad \qquad \times \exp \left\{ (1 - \alpha)C \int_0^T \| \nabla u \|_{L^\infty} \dt \right\} \\
& \lesssim \Big( 1 + T \| \partial_d R \|_{B^{1/2}_{\infty, 1}} \Big) \| \rho_0 \|_X \left( 1 + \int_0^T \| \nabla u \|_{L^\infty} \dt \right) \exp \left\{ (1 - \alpha)C \int_0^T \| \nabla u \|_{L^\infty} \dt \right\}.
\end{split}
\end{equation*}
Therefore, by setting $E(t) = \| \rho \|_X$ we obtain an integral inequality
\begin{equation*}
E(T) \lesssim \Big(1 + \| R \|_{B^{3/2}_{\infty, 1}} T \Big) \| \rho_0 \|_X \left( 1 + \int_0^T E(t) \dt \right) \exp \left\{ (1 - \alpha) C \int_0^T E(t) \dt \right\},
\end{equation*}
which we may integrate in order to get a lower bound for the lifespan $T^*$ of the solution:
\begin{equation*}
T^* \Big(1 + \| R \|_{B^{3/2}_{\infty, 1}} T^* \Big) \geq \frac{C}{\| \rho_0 \|_X} \int_0^{+ \infty} \frac{e^{-(1 - \alpha)Cs}}{1 + s} {\rm d}s,
\end{equation*}
so that proving our theorem is only a matter of bounding this integral from below. Our argument runs as follow: a constant fraction of the mass under the exponential function is located in the interval $0 \leq s \leq C(1 - \alpha)^{-1}$, so that nothing is lost by restricting our attention to this interval, except perhaps a change in the constants. Next, we use the same principal to replace the exponential term by an easier linear one: on $0 \leq s \leq (1 - \alpha)^{-1}$, the mass of $e^{-(1 - \alpha)Cs}$ is significantly the same as that of the lower bound $1 - (1 - \alpha)C$. Therefore, we write
\begin{equation*}
\begin{split}
\int_0^{+ \infty} \frac{e^{-(1 - \alpha)Cs}}{1 + s} {\rm d}s & \geq \int_0^{C(1 - \alpha)^{-1}} \frac{1 - (1 - \alpha)Cs}{1+s} {\rm d} s \\
&  = \log \left( 1 + \frac{C}{1 - \alpha} \right) - (1 - \alpha )C \int_0^{C(1 - \alpha)^{-1}} \frac{s}{1+s} {\rm d} s \\
& = \Big( 1 + (1 - \alpha)C \Big) \log \left( 1 + \frac{C}{1 - \alpha} \right) - C \\
& \sim \log \left( 1 + \frac{C}{1 - \alpha} \right), \qquad \text{as } \alpha \rightarrow 0^+.
\end{split}
\end{equation*}
Let us conclude: we have shown that the quantity $E(t) = \| \rho(t) \|_X$ remains bounded on the time interval $[0, T^*]$, and therefore so does $\| \nabla u (t) \|_{L^\infty}$. An application of the continuation criteria of Theorem \ref{t:localWP} shows that the lifespan of the solution is therefore at least $T^*$. This finally ends the proof of Theorem \ref{t:lifespanIncrease}.
\end{proof}

\appendix

\section{Appendix: Littlewood-Paley Analysis and Transport Equations} \label{app:LP}

In this appendix, we collect tools from Fourier analysis and Littlewood-Paley theory which we have freely used throughout all our paper. We refer to Chapters 2 and 3 of \cite{BCD} for details.

\subsection{Littlewood-Paley Decompositions}\label{ss:LP}

In this paragraph, we introduce the notion of Littlewood-Paley decomposition and present some of its basic properties. 

\medskip

The Littlewood-Paley decomposition is based on a dyadic partition of unity in the Fourier variable: fix a smooth, compactly supported and radially symmetric function $\chi \in \mc D$ such that $r \mapsto \chi(re)$ is decreasing for all $e \in \R^d$ and with
\begin{equation*}
\chi(x) = 1 \text{ for } |x| \leq 1 \qquad \text{and} \qquad \chi(x) = 0 \text{ for } |x| \geq 2.
\end{equation*}
Then, by setting $\varphi(\xi) = \chi(\xi) - \chi(2 \xi)$ and $\varphi_j(\xi) = \varphi(2^{-j}\xi)$, we obtain the following partitions of unity in the Fourier variable:
\begin{equation}\label{eq:LPpartition}
1 = \sum_{j \in \Z} \varphi_j (\xi) \qquad \text{and} \qquad 1 = \chi(\xi) + \sum_{j \geq 0} \varphi_j(\xi).
\end{equation}
The first one holds for all $\xi \neq 0$ (with pointwise convergence of the sum) and the second for all $\xi \in \R^d$. These two partitions of unity give rise to two sets of operators: we define the Littlewood-Paley blocks by, on the one hand $\dot{\Delta}_j = \varphi_j(D)$ for $j \in \Z$, and on the other $\Delta_{-1} = \chi(D)$ and $\Delta_j = \dot{\Delta}_j$ for $j \geq 0$. Also set $S_j = \chi(2^{j-1} D)$. 

\begin{rmk}
The operators $\dot{\Delta}_j$, $\Delta_j$ and $S_j$ are scaled versions of $\dot{\Delta}_{0}$ ot $\Delta_{-1}$, and thus are uniformly bounded in the $L^p \tend L^p$ topology for all $p \in [1, + \infty]$. 
\end{rmk}

From the two partitions of unity \eqref{eq:LPpartition}, we may write the formal operator identities
\begin{equation*}
{\rm Id} = \sum_{j \in \Z} \dot{\Delta}_j \qquad \text{and} \qquad {\rm Id} = \sum_{j \geq -1} \Delta_j.
\end{equation*}
The first identity is the \textsl{homogeneous Littlewood-Paley decomposition}, whereas the second is the \textsl{non-homogeneous} one. While the non-homogeneous Littlewood-Paley decomposition holds on the whole space $\mc S'$ of tempered distributions, such is not the case with the homogeneous Littlewood-Paley decomposition (for example, it fails for all polynomials). In fact, there are even examples of bounded functions $f \in L^\infty$ for which the sum $\sum_{j \in \Z} \dot{\Delta}_j f$ may fail to converge in $\mc S'$ (we refer to \cite{Cobb2} for a discussion of this phenomenon).

It will be useful to define a space $\mc S'_h$ of tempered distributions for which the homogeneous Littlewood-Paley decomposition always holds.

\begin{defi}\label{d:SpH}
We note $\mc S'_h$ the set of all $f \in \mc S'$ such that 
\begin{equation*}
\chi(\lambda \xi) \what{f} (\xi) \tend_{\lambda \rightarrow + \infty} 0 \qquad \text{in } \mc S'.
\end{equation*}
In particular, for all $f \in \mc S'_h$, we have $f = \sum_{j \in \mc Z} \dot{ \Delta}_j f$ with convergence of the sum in the $\mc S'$ topology.
\end{defi}

\begin{rmk}
Functions in $\mc S'_h$ should be understood as obeying some some kind of far field. In fact, $\mc S'_h$ functions possess weak (\textsl{i.e.} average) decay properties at infinity, as shown in \cite{Cobb2}.
\end{rmk}

One of the main interests of the Littwood-Paley decomposition is the way it relates to differentiation. Basically, because a block $\dot{\Delta}_j f$ has its Fourier transform supported in an annulus of size roughly $2^j$, the derivative will act as a multiplication by $2^j$. This remark, formalized in the Bernstein inequalities below, can be extended to more general Fourier multiplication operators.

\begin{lemma}[Bernstein inequalities]\label{l:Bernstein}
Let  $0<r<R$.   A constant $C$ exists so that, for any nonnegative integer $k$, any couple $(p,q)$ 
in $[1,+\infty]^2$, with  $p\leq q$,  and any function $u\in L^p$,  we  have, for all $\lambda>0$,
$$
\displaylines{
{\rm supp}\, ( \widehat u ) \subset   B(0,\lambda R)\quad
\Longrightarrow\quad
\|\nabla^k u\|_{L^q}\, \leq\,
 C^{k+1}\,\lambda^{k+d\left(\frac{1}{p}-\frac{1}{q}\right)}\,\|u\|_{L^p}\;;\cr
{\rm supp}\, ( \widehat u ) \subset \{\xi\in\R^d\,|\, r\lambda\leq|\xi|\leq R\lambda\}
\quad\Longrightarrow\quad C^{-k-1}\,\lambda^k\|u\|_{L^p}\,
\leq\,
\|\nabla^k u\|_{L^p}\,
\leq\,
C^{k+1} \, \lambda^k\|u\|_{L^p}\,.
}$$
\end{lemma}   

\begin{lemma}[see Lemma 2.2 in \cite{BCD}]\label{l:FourierMultiplier}
Let $m(\xi)$ be a $C^k$ function away from $\xi = 0$ with $k = 2 \lfloor 1 + d/2 \rfloor$ and assume there is a degree $s \in \R$ such that $|\nabla^l m (\xi)| \leq C |\xi|^{s - l}$ for all $\xi \neq 0$ and $0 \leq l \leq k$. Then there exists a constant depending only on $m$ and on the the dyadic decomposition function $\chi$ such that, for all $p \in [1, +\infty]$,
\begin{equation*}
\forall j \in \mathbb{Z}, \forall f \in L^p, \qquad \| \dot{\Delta}_j m(D) f \|_{L^p} \leq C 2^{js} \| \dot{\Delta}_j f \|_{L^p}.
\end{equation*}
\end{lemma}

These two Lemmas are the corner stone for defining a functional framework that is adapted to the Fourier multiplication operators we use in this article. In particular, we will be interested in the fractional Laplacian $(- \Delta)^{s/2}$, whose symbol is $|\xi|^s$ and the Leray projection operator, whose symbol is
\begin{equation*}
m(\xi) = {\rm Id} - \frac{\xi \otimes \xi}{|\xi|^2}.
\end{equation*}
These operators have symbols which are homogeneous functions, and Lemma \ref{l:FourierMultiplier} applies to both of them.

\subsection{Besov Spaces}

In this paragraph, we introduce the class of Besov spaces (both homogeneous and non-homogeneous). These are based on the Littlewood-Paley decomposition we have just defined, and are tailored to Fourier multiplication operators. Let us start with the non-homogeneous Besov spaces.

\begin{defi} \label{d:c2Besov}
Let $s\in\R$ and $1\leq p,r\leq+\infty$. The \textsl{non-homogeneous Besov space} $B^{s}_{p,r}\,=\,B^s_{p,r}(\R^d)$ is defined as the set of tempered distributions $u \in \mc S'$ for which
$$
\|f\|_{B^{s}_{p,r}}\,:=\,
\left\|\left(2^{js}\,\|\Delta_j f \|_{L^p}\right)_{j \geq -1}\right\|_{\ell^r}\,<\,+\infty\,.
$$
The non-homogeneous space $B^s_{p, r}$ is Banach for all values of $(s, p, r)$.
\end{defi}

As with the usual (potential) Sobolev spaces $W^{s, p}$, the exponent $s \in \R$ acts as a regularity index and $p$ as an integrability exponent. In fact, an immediate consequence of the Bernstein inequalities are the embeddings
\begin{equation}\label{eq:crudeEmbed}
B^k_{p, 1} \subset W^{k, p} \subset B^k_{p, \infty},
\end{equation}
which hold for all $k \in \N$ and all $p \in [1, + \infty]$.

\medskip

In the same way, we define the homogeneous Besov spaces, which are based on the homogeneous Littlewood-Paley decomposition. The only issue is that we must account for the fact that the decomposition might not always converge: this is why we restrict our attention to a subcritical range of exponents and involve the space $\mc S'_h$ in our definition.\footnote{A more general definition (see \cite{Sawano}) of homogeneous Besov spaces involves working with clases of ditributions modulo polynomials, but there is no real difference in our subcritical framework. Here, subcritical spaces refer to homogeneous Besov spaces $\dot{B}^s_{p, r}$ where the exponent $s \in \R$ is \textsl{below} a critical value, see the Definition hereafter.}

\begin{defi}
Let $(s, p, r) \in \R \times [1, + \infty]^2$ such that the following subcriticality condition is fulfilled:
\begin{equation}\label{eq:homoBesovCond}
s < \frac{d}{p} \qquad \text{or} \qquad s = \frac{d}{p} \text{ and } r=1.
\end{equation}
Then we define the \textsl{homogeneous Besov space} $\dot{B}^s_{p, r}$ as being the set of those $f \in \mc S'_h$ such that
\begin{equation*}
\| f \|_{\dot{B}^s_{p, r}} := \left\| \left( 2^{ms} \| \dot{\Delta}_m f \|_{L^p} \right)_{m \in \mathbb{Z}} \right\|_{\ell^r(\mathbb{Z})} < + \infty.
\end{equation*}
The homogeneous Besov space $\dot{B}^s_{p, r}$ is Banach for values of $(s, p, r)$ that satisfy \eqref{eq:homoBesovCond}.
\end{defi}

Similarly to the non-homogeneous case, the Bernstein inequalities let us straightforwardly infer the embeddings
\begin{equation*}
\dot{B}^k_{p, 1} \subset W^{k, p} \subset \dot{B}^k_{p, \infty}
\end{equation*}
as long as the exponent $k \in \N$ satisfies the subcriticality condition $k < \frac{d}{p}$ (in particular, not for $p = +\infty$ and $k=0$).

\medskip

The main reason we consider homogeneous Besov spaces is that they behave better than their non-homogeneous counterparts under the action of Fourier multipliers that exhibit a singular behavior at low frequencies. The following Proposition can be seen as a direct consequence of Lemma \ref{l:FourierMultiplier}.

\begin{prop}\label{p:FracLapBesov}
Consider $\sigma, s \in \R$ and $p, r \in [1, + \infty]$ such that the triplets $(s, p, r)$ and $(\sigma, p, r)$ fulfill the subcriticality condition \eqref{eq:homoBesovCond}. Then the fractional Laplace operator $(- \Delta)^{(s - \sigma)/2}$ defines a bounded operator between the homogeneous Besov spaces
$$
\xymatrix{
(- \Delta)^{(s - \sigma)/2} : \dot{B}^s_{p, r} \ar[r] & \dot{B}^\sigma_{p, r}.
}
$$
\end{prop}

\begin{proof}
We consider $f \in \dot{B}^s_{p, r}$. Lemma \ref{l:FourierMultiplier} makes it plain that we have an inequality between the norms
\begin{equation*}
\| (- \Delta)^{(s - \sigma)/2} f \|_{\dot{B}^s_{p, r}} \lesssim \| f \|_{\dot{B^\sigma}_{p, r}},
\end{equation*}
however it is unclear whether $(- \Delta)^{(s - \sigma)/2} f$ is properly defined as an element of $\mc S'_h$. Let us show that the series
\begin{equation*}
g := \sum_{j \in \Z} \dot{\Delta}_j (- \Delta)^{(s - \sigma)/2} f
\end{equation*}
converges in $\mc S'$. On the one hand, the high frequency part $\sum_{j \geq 0} \dot{\Delta}_j (- \Delta)^{(s - \sigma)/2} f$ creates no problem. For the low frequency part, we note that the Bernstein inequalities imply
\begin{equation*}
\begin{split}
\sum_{j \leq 0} \| \dot{\Delta}_j (- \Delta)^{(s - \sigma)/2} f \|_{L^\infty} & \lesssim \sum_{j \leq 0} 2^{jd/p} \| \dot{\Delta}_j (- \Delta)^{(s - \sigma)/2} f \|_{L^p} \\
& \lesssim \sum_{j \leq 0} 2^{j(- \sigma + d/p)} \| f \|_{B^s_{p, r}},
\end{split}
\end{equation*}
and this last sum must be finite by virtue of the subcriticality condition \eqref{eq:homoBesovCond}, thus the series is normally convergent. Finally, this means that $g$ is indeed an element of $\mc S'_h$.
\end{proof}

The Bernstein inequalities immediately translate into embeddings of Besov spaces. We record these in the following Proposition.

\begin{prop}\label{p:BesovEmbed}
Consider $s \in \R$ and $p_1, p_2, r_1, r_2 \in [1, + \infty]$ with $r_1 \leq r_2$ and $p_1 \leq p_2$. Then we have
\begin{equation*}
B^s_{p_1, r_1} \subset B^{s - d \left( \frac{1}{p_1} - \frac{1}{p_2} \right)}_{p_2, r_2}.
\end{equation*}
\end{prop}

We now give a few additional properties of Besov spaces that are useful in our proofs: two interpolation inequalities and a Fatou-type property.

\begin{prop}\label{p:interpolation}
Consider real numbers $s_1 \leq s_2$ and $p, r \in [1, + \infty]$. Then, for all $\theta \in [0, 1]$ and $s := \theta s_1 + (1 - \theta)s_2$, we have
\begin{equation*}
\forall f \in B^{s_2}_{p, r}, \qquad \| f \|_{B^s_{p, r}} \leq \| f \|_{B^{s_1}_{p, r}}^\theta \| f \|_{B^{s_1}_{p, r}}^{1 - \theta}.
\end{equation*}
\end{prop}


\begin{prop}[See Proposition 2.104 in \cite{BCD}]\label{p:LogInt}
Consider $\epsilon > 0$ and $f \in B^\epsilon_{\infty, \infty}$. Then,
\begin{equation*}
\| f \|_{L^\infty} \lesssim \frac{1}{\epsilon} \, \| f \|_{B^0_{\infty, \infty}} \left\{ 1 + \log \left( \frac{\| f \|_{B^\epsilon_{\infty, \infty}}}{\| f \|_{B^0_{\infty, \infty}}} \right) \right\}.
\end{equation*}
\end{prop}

\begin{prop}[Theorem 2.72 in \cite{BCD}]\label{p:Fatou}
Let $(s, p, r) \in \R \times [1, + \infty]^2$ and consider a sequence of functions $(f_n)_{n \in \N}$ that is uniformly bounded in $B^s_{p, r}$, in other words $\| f_n \|_{B^s_{p, r}} \leq C$. Then the sequence converges in $\mc S'$ up to an extraction to a $f \in \mc S'$ that satisfies
\begin{equation*}
\| f \|_{B^s_{p, r}} \lesssim \limii_{n \rightarrow + \infty} \| f_n \|_{B^s_{p, r}}.
\end{equation*}
\end{prop}

\begin{rmk}\label{r:Fatou}
This property also holds with the homogeneous Besov space $\dot{B}^s_{p, r}$, provided that \eqref{eq:homoBesovCond} is true. We refer to Theorem 2.25 in \cite{BCD} for a proof.
\end{rmk}

Finally, to conclude this paragraph, we focus on embeddings that are more precise than the ones in \eqref{eq:crudeEmbed} above. In particular, we will devote a special attention to the role of the third index $r$ of the Besov space $B^s_{p, r}$.

\begin{prop}[See Theorems 2.40 and 2.41 in \cite{BCD}]\label{p:refinedEmbed}
Consider $p \in [1, + \infty]$. Then we have the embeddings
\begin{equation*}
\begin{split}
& B^0_{p, p} \subset L^p \subset B^0_{p, 2} \qquad \text{if } 1 < p \leq 2 \\
& B^0_{p, 2} \subset L^p \subset B^0_{p, p} \qquad \text{if } 2 \leq p < + \infty.
\end{split}
\end{equation*}
\end{prop}

In the case $p = +\infty$, the space $B^1_{\infty, r}$ is included in the space of $\log$-Lipschitz functions, which play a special role in the theory of differential and transport equations.

\begin{defi}
Consider $\alpha \geq 0$ and define the space $LL_\alpha$ of bounded functions $f \in L^\infty$ such that, for all $x, y \in \R^d$, we have
\begin{equation}\label{eq:LLineq}
|f(x) - f(y)| \leq C(f) \big( 1 - \log(|x-y|) \, \big)^\alpha |x - y|.
\end{equation}
This space is Banach when associated to the norm $\| f \|_{LL_\alpha} := \| f \|_{L^\infty} + C^*(f)$, where $C^*(f)$ is the best constant $C(f)$ for which \eqref{eq:LLineq} holds.
\end{defi}

\begin{prop}[See Proposition 2.111 in \cite{BCD}]
Consider $\alpha \geq 0$. Then the $LL_\alpha$ norm is equivalent to
\begin{equation*}
\forall f \in LL_\alpha, \qquad \| f \|_{LL_\alpha} \approx \| f \|_{L^\infty} + \sup_{j \geq -1} \left[ \frac{\| \nabla S_j f \|_{L^\infty}}{(j+2)^\alpha} \right].
\end{equation*}
In particular, for all $r \in [1, + \infty]$ and $\alpha = 1 - \frac{1}{r}$, we have the embedding $B^1_{\infty, r} \subset LL_\alpha$.
\end{prop}

In fact, for our purposes, we will need to work in spaces of functions that embed in $LL_\alpha$, such as the Besov space $B^{1 + d/p}_{p, r} \subset B^1_{\infty, r} \subset LL_\alpha$. We therefore make the following definition (after (3.33) p. 157 in \cite{BCD}).

\begin{defi}
Consider $\alpha \geq 0$ and define the space $LL_\alpha^p$ with the following norm:
\begin{equation*}
\| f \|_{LL_\alpha^p} := \| f \|_{L^\infty} + \sup_{j \geq -1} \left[ \frac{2^{jd/p} \| \nabla S_j f \|_{L^p}}{(j+2)^\alpha} \right].
\end{equation*}
In particular, the Bernstein inequalities imply that $B^{1+d/p}_{p, r} \subset LL_\alpha^p \subset LL_\alpha^\infty =  LL_\alpha$ for all $p \in [1, + \infty]$ and $\alpha = 1 - \frac{1}{r} \geq 0$.
\end{defi}

\subsection{Non-Homogeneous Paradifferential Calculus}\label{s:NHPC}

Let us now introduce the paraproduct operator (after J.-M. Bony, see \cite{Bony}). Constructing the paraproduct operator relies on the observation that, 
formally, any product  of two tempered distributions $u$ and $v,$ may be decomposed into 
\begin{equation}\label{eq:bony}
u\,v\;=\;\mathcal{T}_u(v)\,+\,\mathcal{T}_v(u)\,+\,\mathcal{R}(u,v)\,,
\end{equation}
where we have defined
$$
\mathcal{T}_u(v)\,:=\,\sum_jS_{j-1}u\Delta_j v,\qquad\qquad\mbox{ and }\qquad\qquad
\mathcal{R}(u,v)\,:=\,\sum_j\sum_{|j'-j|\leq1}\Delta_j u\,\Delta_{j'}v\,.
$$
The above operator $\mc T$ is called ``paraproduct'' whereas
$\mc R$ is called ``remainder''.
The paraproduct and remainder operators have many nice continuity properties. 
The following ones have been of constant use in this paper. 
\begin{prop}\label{p:op}
For any $(s,p,r)\in\R\times[1,+\infty]^2$ and $t>0$, the paraproduct operator 
$\mathcal{T}$ maps continuously $L^\infty\times B^s_{p,r}$ in $B^s_{p,r}$ and  $B^{-t}_{\infty,\infty}\times B^s_{p,r}$ in $B^{s-t}_{p,r}$.
Moreover, the following estimates hold:
$$
\|\mathcal{T}_u(v)\|_{B^s_{p,r}}\,\leq\, C\,\|u\|_{L^\infty}\,\|\nabla v\|_{B^{s-1}_{p,r}}
$$
as well as
$$
\|\mathcal{T}_u(v)\|_{B^{s-t}_{p,r}}\,\leq\, \frac{C}{t} \|u\|_{B^{-t}_{\infty,\infty}}\,\|\nabla v\|_{B^{s-1}_{p,r}} \qquad \text{and} \qquad \|\mathcal{T}_u(v)\|_{B^{s-t}_{p,r}}\,\leq\, \frac{C}{t} \|u\|_{B^{-t}_{p,\infty}}\,\|\nabla v\|_{B^{s-1}_{\infty,r}}.
$$
For any $(s_1,p_1,r_1)$ and $(s_2,p_2,r_2)$ in $\R\times[1,+\infty]^2$ such that 
$s_1+s_2>0$, $1/p:=1/p_1+1/p_2\leq1$ and~$1/r:=1/r_1+1/r_2\leq1$,
the remainder operator $\mathcal{R}$ maps continuously~$B^{s_1}_{p_1,r_1}\times B^{s_2}_{p_2,r_2}$ into~$B^{s_1+s_2}_{p,r}$.
In the case $s_1+s_2=0$, provided $r=1$, operator $\mathcal{R}$ is continuous from $B^{s_1}_{p_1,r_1}\times B^{s_2}_{p_2,r_2}$ with values
in $B^{0}_{p,\infty}$.
\end{prop}

\subsection{Transport Equations and Commutator Estimates}

In this section, we focus on the transport. We refer again to Chapters 2 and 3 of \cite{BCD} for additional details. We study the initial value problem
\begin{equation}\label{eq:TV}
\begin{cases}
\partial_t f + v \cdot \nabla f = g \\
f_{|t = 0} = f_0\,.
\end{cases}
\end{equation}
Here, $g$ is a given function. In the sequel, the velocity field $v=v(t,x)$ will always assumed to be divergence-free, \tsl{i.e.} $\D(v) = 0$, so that Lebesgue norms will be conserved $\| f(t) \|_{L^p} \leq \| f_0 \|_{L^p} + \| g \|_{L^1(L^p)}$ by the flow of \eqref{eq:TV} if $v$ is regular enough ($v \in L^1(W^{1, 1}_{\rm loc})$ if we want to use Di Perna-Lions theory for example).

\medskip

We first focus on the case where the vector field $v$ is globally Lipschitz: we will assume that $v$ possesses Besov-Lipschitz regularity $B^\sigma_{p, r}$ with
\begin{equation}\label{eq:Lip}
\sigma > 1 + \frac{d}{p} \qquad \text{or} \qquad \sigma = 1 + \frac{d}{p} \text{ and } r=1.
\end{equation}
Finding \textsl{a priori} estimates for problem \eqref{eq:TV} in Besov spaces requires to look at the dyadic blocks. Let $j \geq -1$. Applying $\Delta_j$ to the transport equation yields
\begin{equation}\label{eq:AppTransport}
\Big( \partial_t + v \cdot \nabla \Big) \Delta_j f = \big[ v \cdot \nabla, \Delta_j \big] f + \Delta_j g,
\end{equation}
where $\big[ v \cdot \nabla, \Delta_j \big] f$ is the commutator $(v \cdot \nabla) \Delta_j - \Delta_j (v \cdot \nabla)$. The following estimate is of recurring use in this article (see Lemma 2.100). 

\begin{lemma}\label{l:CommBCD}
Consider $s \in \R$, indices $p, r \in [1, + \infty]$ and a function $f \in B^s_{p, r}$. Assume that $q \in [1, + \infty]$ satisfies $q \geq p$. We set $\sigma_0 = -1 - \min \left\{ \frac{1}{q}, \frac{1}{p'} \right\}$ where $p'$ is the conjuguate exponent $\frac{1}{p} + \frac{1}{p'} = 1$. Then the following inequalities hold
\begin{equation*}
\Big\| 2^{js} \big[ v \cdot \nabla, \Delta_j \big] f \Big\|_{\ell^r(L^p)} \lesssim
\begin{cases}
\| \nabla v \|_{B^{d/q}_{q, \infty} \cap L^\infty} \| f \|_{B^s_{p, r}}  & \text{if } \sigma_0 < s < 1 + \frac{d}{q},\\
\| \nabla v \|_{B^{d/q}_{q, 1}} \| f \|_{B^s_{p, \infty}} & \text{if } s = \sigma_0 \text{ and } r = + \infty,\\
\| \nabla v \|_{L^\infty} \| f \|_{B^s_{p, r}} + \| \nabla f \|_{L^p} \| \nabla v \|_{B^{s-1}_{\infty, r}} & \text{if } s > -1
\end{cases}
\end{equation*}
\end{lemma}

The following basic commutator inequality will also be useful.

\begin{lemma}[See Lemma 2.97 in \cite{BCD}]\label{l:basicComm}
    Let $\theta \in C^1$ be such that $(1 + |\xi|) \what{\theta}(\xi) \in L^1$. Then there exists a constant $C > 0$ such that for any Lipschitz $a$ with gradient $\nabla a \in L^p$ and for any $f \in L^q$, we have
    \begin{equation*}
        \forall \lambda > 0, \qquad \big\| \big[ \theta (\lambda^{-1} D), a \big] f \big\|_{L^r} \leq \frac{C}{\lambda} \| \nabla a \|_{L^p} \| f \|_{L^q}, \qquad \text{if } \frac{1}{r} = \frac{1}{p} + \frac{1}{q} \geq 1.
    \end{equation*}
    \end{lemma}

Applying these commutator inequalities to obtain $L^p$ estimates in the transport equation \eqref{eq:AppTransport} results (with some additional work) in a well-posedness theorem for problem \eqref{eq:TV} in general Besov spaces (see Theorem 3.19 in \cite{BCD}).

\begin{thm}\label{th:transport}
Let $s \in \R$ and $p, q, r \in [1, + \infty]$ such that $p \leq q$. Define $\sigma_0$ as in Lemma \ref{l:CommBCD} and suppose that $s > \sigma_0$, or that $s = \sigma_0$ and $r = + \infty$. Given $T > 0$, assume that the (divergence free) vector field $v$ is such that $\Delta_{-1} v \in L^q_T(L^\infty)$ for some $q > 1$ and
\begin{equation*}
\begin{split}
& \nabla v \in L^1_T(B^{d/p}_{p, \infty} \cap L^\infty) \; \, \qquad \text{if } s < 1 + \frac{d}{p}, \\
& \nabla v \in L^1_T(B^{s-1}_{p, r}) \qquad \qquad \quad \text{if } s > 1 + \frac{d}{p}, \text{ or if } s = 1 + \frac{d}{p} \text{ and } r=1.
\end{split}
\end{equation*}
Then, for any initial datum $f_0 \in \B$ and $g \in L^1_T(B^s_{p, r})$, the transport equation \eqref{eq:TV} has a unique solution $f$ in\footnote{If $X$ is a Banach space with predual $Y$, the notation $C_w([0, T[ ; X)$ refers to the space of weakly continuous functions in $X$, that is measurable functions $f : [0, T[ \tend X$ such that the brackets $t \mapsto \langle f(t), \phi \rangle_{X \times Y}$ are continuous for all $\phi \in Y$.} 
\begin{itemize}
\item the space $C^0_T(B^s_{p, r})$, if $r < +\infty$;
\item the space $\left( \bigcap_{s'<s} C^0_T(B^{s'}_{p, \infty}) \right) \cap C^0_{w, T}(B^s_{p, \infty})$, if $r = +\infty$.
\end{itemize}
Moreover, this unique solution satisfies the following estimate:
\begin{equation*} 
\| f \|_{L^\infty_T(\B)} \leq \exp \left( C \int_0^T \| \nabla v \|_{B^{s-1}_{p, r}} \right)
\left\{ \| f_0 \|_{\B} + \int_0^T \exp \left( - C \int_0^\tau \| \nabla v \|_{B^{s-1}_{p, r}} \right) \| g(\tau) \|_{\B} {\rm d} \tau  \right\},
\end{equation*}
for some constant $C = C(d, p, r, s)>0$.
\end{thm}

\begin{rmk}
By Remark 3.16 in \cite{BCD}, Theorem \ref{th:transport} extends to homogeneous Besov spaces: the space $B^s_{p, r}$ may be replaced by $\dot{B}^s_{p, r}$ under the additional requirement that $s < 1 + \frac{d}{q}$, or that $s = 1 + \frac{d}{q}$ and $r=1$.
\end{rmk}

In the special case of Besov spaces of regularity index $s = 0$, a better estimate is available for the solution of \eqref{eq:TV}. When $s \neq 0$, Theorem \ref{th:transport} shows that the norm $\| f(t) \|_{B^s_{p, r}}$ might grow exponentially with respect to $\nabla v$, whereas $L^p$ norms are simply conserved $\| f(t) \|_{L^p} \leq \| f_0 \|_{L^p} + \| g \|_{L^1_T(L^p)}$. Some kind of logarithmic interpolation allows to combine both inequalities and obtain a bound that is linear with respect to $\nabla v$.

\begin{prop}[see Theorem 3.18 in \cite{BCD}]\label{p:linEstTV}
Consider $p, r \in [1, + \infty]$ and let $f \in L^\infty_T(B^0_{p, r})$ be a solution of the transport equation \eqref{eq:TV} for soem initial datum $f_0 \in B^0_{p, r}$ and with $v \in L^1_T(W^{1, \infty})$ and $g \in L^1_T(B^0_{p, r})$ for some $T > 0$. Then the following inequality holds:
\begin{equation*}
\| f \|_{L^\infty_T(B^0_{p, r})} \lesssim \left( \| f_0 \|_{B^0_{p, r}} + \int_0^T \| g \|_{B^0_{p, r}} \dt \right) \left( 1 + \int_0^T \| \nabla v \|_{L^\infty} \dt \right).
\end{equation*}
\end{prop}

\begin{rmk}
Just as for Theorem \ref{th:transport}, the linear estimates of Proposition \ref{p:linEstTV} also are true when replacing $B^0_{p, r}$ by the homogeneous space $\dot{B}^0_{p, r}$, only requiring that $r = 1$ when $p = +\infty$.
\end{rmk}

Finally, we consider the transport equation when the (still divergence free) vector field $v$ is not quite Lipschitz, but $\log$-Lipschitz. In that case, it is possible that the solution will loose regularity with time. The rate at which this loss happens depends on how far $v$ is from actually being Lipschitz: for instance if $v \in L^1(LL_\alpha)$ for some $\alpha > 0$, then loss of regularity can be made arbitrarily small.

\begin{thm}[See Theorem 3.33 and Remark 3.35 in \cite{BCD}]\label{t:limitedLoss}
Consider $p, q \in [1, + \infty]$ and define $\sigma_0$ as in Lemma \ref{l:CommBCD}. Consider $s \in ]\sigma_0, 1 + d/q [$ and an initial datum $f_0 \in B^s_{p, \infty}$. We assume that the velocity field $v$ satisfies
\begin{equation}\label{eq:limitedLossHypV}
\Delta_{-1} v \in L^\infty_T(L^\infty) \qquad \text{and} \qquad ({\rm Id} - \Delta_{-1}) v \in L^1_T(LL^q_\alpha)
\end{equation}
for some $T > 0$. For all $g$ such that the inequality
\begin{equation}\label{eq:limitedLossHypG}
\| g \|_{B^{s - \eta}_{p, \infty}} \lesssim \| f \|_{B^{s-\eta}_{p, \infty}} \Big( \| \Delta_{-1} v \|_{L^\infty} + \big\| ({\rm Id} - \Delta_{-1})v \big\|_{LL^q_\alpha} \Big)
\end{equation}
holds for all values of $\eta$ such that $\sigma_0 < s - \eta \leq s$, the transport equation has a unique solution in the space $C^0_T( \bigcap_{s' < s} B^{s'}_{p, \infty})$ and, for $\epsilon > 0$ small enough, the following inequality holds:
\begin{equation}\label{eq:limitedLossEst}
\| f \|_{B^{s - \epsilon}_{p, \infty}} \lesssim \| f_0 \|_{B^s_{p, \infty}} \exp \left\{ \frac{C}{\epsilon^{\alpha/(1 - \alpha)}} \left( \int_0^T \Big( \| \Delta_{-1} v \|_{L^\infty} + \big\| ({\rm Id} - \Delta_{-1})v \big\|_{LL^q_\alpha} \Big) \dt \right)^{1/(1 - \alpha)} \right\}.
\end{equation}
\end{thm}

\begin{rmk}
Theorem \ref{t:limitedLoss} is a slight modification of Theorem 3.33 in \cite{BCD} to allow for vector fields such that $\Delta_{-1} v$ does not decay at infinity (as it would if it were $L^q$ for $q < + \infty$), and instead use assumptions \eqref{eq:limitedLossHypV}. The proof is identical, once it is noticed that Lemma 3.29 (still in \cite{BCD}) only involves the low frequencies to deal with the commutator $[\Delta_j, S_{j+1}v \cdot \nabla] f$, and then resorts to Lemma \ref{l:CommBCD} (Lemma 2.100 in \cite{BCD}) whose proof is already based on a separation of low and high frequencies as in \eqref{eq:limitedLossHypV}.
\end{rmk}

\begin{rmk}
The inequality \eqref{eq:limitedLossHypG} we have imposed on $g$ can be relaxed even further (see Remark 3.35 in \cite{BCD}), but we do not need such generality.
\end{rmk}

\end{document}